\documentclass{cpamart1}     
\received{Month 200X}       
\volume{000}
\startingpage{1}                      

\authorheadline{H. Hong, A. Ovchinnikov, G. Pogudin, and C. Yap}
\titleheadline{Global Identifiability of Differential Models}

    \usepackage{amsfonts,amssymb,mathrsfs}
        \usepackage[dvipsnames]{xcolor}
        \usepackage[colorlinks=true,linkcolor=RoyalBlue,citecolor=PineGreen,urlcolor=blue]{hyperref}
        \usepackage[ruled,vlined,linesnumbered]{algorithm2e}
        \usepackage{hyperref}
        \usepackage{tablefootnote}
    
          \usepackage[shortlabels]{enumitem}

\usepackage{newtxtext,newtxmath}

\newtheorem{theorem}{Theorem}[section]
\newtheorem{lemma}[theorem]{Lemma}
\newtheorem{corollary}[theorem]{Corollary}
\newtheorem{proposition}[theorem]{Proposition}
        \newtheorem*{claim}{Claim}
\theoremstyle{definition}
\newtheorem{definition}[theorem]{Definition}
\newtheorem{notation}[theorem]{Notation}
\newtheorem{problem}[theorem]{Problem}
\newtheorem{example}[theorem]{Example}
\theoremstyle{remark}
\newtheorem{remark}[theorem]{Remark}

\SetLabelAlign{LeftAlignWithIndent}{\makebox[1.5em][l]{#1}}
        \newcommand{\func}{\operatorname}
        \newcommand{\bs}{\boldsymbol}   
        \newcommand{\bfx}{\bs{x}}
        \newcommand{\bfxs}{\bs{x}^\ast}
        \newcommand{\bfy}{\bs{y}}
        \newcommand{\bff}{\bs{f}}
        \newcommand{\bfg}{\bs{g}}
        \newcommand{\bfmu}{\bs{\mu}}
        \newcommand{\bftheta}{\bs{\theta}}      
        \newcommand{\whbftheta}{\hat{\bftheta}}    
        \newcommand{\whu}{\hat{u}}                 
        \newcommand{\whtheta}{\hat{\theta}}

        \newcommand{\cinfty}{\CC^{\infty}(0)}
        \newcommand{\dd}{,\ldots,}
        \newcommand{\CC}{\mathbb{C}}
        \newcommand{\ZZ}{\mathbb{Z}}

        \DeclareMathOperator{\ord}{ord}
        
                \DeclareMathOperator{\Quot}{\func{Quot}}
        \DeclareMathOperator{\R}{\mathcal R}
        \DeclareMathOperator{\J}{\mathcal J}
        \DeclareMathOperator{\cS}{\mathcal S}
        \DeclareMathOperator{\cT}{\mathcal T}
        \DeclareMathOperator{\cF}{\mathcal F}
        \DeclareMathOperator{\cE}{\mathcal E}
        \DeclareMathOperator{\Out}{out}
        \DeclareMathOperator{\In}{in}
        \DeclareMathOperator{\State}{st}
        \allowdisplaybreaks
               \newcommand{\linkTo}[2]{\hyperlink{#1}{{\color{RoyalBlue}#2}}}
               \newcommand{\linkFro}[2]{\hypertarget{#1}{{\color{PineGreen}#2}}}
                \newcommand{\glssymbolX}[2]{\linkFro{#1}{#2}}
                \newcommand{\glssymbolY}[2]{\linkTo{#1}{#2}}

\newcounter{desccount}
\newcommand{\descitem}[1]{%
  \item[#1] \refstepcounter{desccount}\label{#1}
}
\newcommand{\descref}[1]{\hyperref[#1]{#1}}

\begin{document}                        


\title{Global Identifiability of Differential Models\footnote{The names of the authors are ordered alphabetically by their last names.}}

\author{Hoon Hong}{North Carolina State University}
\author{Alexey Ovchinnikov}{CUNY Queens College and Graduate Center}
\author{Gleb Pogudin}{Courant Institute of Mathematical Sciences\footnote{Current address of Gleb Pogudin: National Research University Higher School of Economics, Moscow, Russia}}
\author{Chee Yap}{Courant Institute of Mathematical Sciences}





\begin{abstract}
Many real-world processes and phenomena are modeled using systems of
ordinary differential equations with parameters. Given such a system, we say
that a parameter is globally identifiable if it can be uniquely recovered
from input and output data. The main contribution of this paper is to
provide theory, an algorithm, and software for deciding global
identifiability. First, we rigorously derive an algebraic criterion for
global identifiability (this is an analytic property), which yields a
deterministic algorithm. Second, we improve the efficiency by randomizing
the algorithm while guaranteeing  the probability of correctness. 
With our new algorithm, we can tackle problems that could not be tackled
before.
 A software based on the algorithm (called SIAN) is available at \url{https://github.com/pogudingleb/SIAN}.
\end{abstract}

\maketitle   



 \tableofcontents



\section{Introduction}


Many real-world processes and phenomena are modeled using systems of
parametric ordinary differential equations (ODEs). 
There are multiple challenges in designing such a model. In this paper, we
address one of them: structural global identifiability. One might be
interested in knowing the values of some select parameters due to their
importance. Usually, one tries to determine (uniquely identify) them by
collecting input and output data. However, due to the structure of the
model, it can be impossible to determine the parameters from input and
output data. When this happens, the parameters are said to be ``not globally
identifiable''. In the process of model design, it is crucial to know
whether the parameters of interest in a potential model are globally
identifiable, which is the topic of this paper.

Checking global identifiability is challenging. 
Thus, a weaker notion, called ``local identifiability'', has been introduced,
which means that a parameter can be identified up to finitely many options.
There has been remarkable progress in checking local identifiability,
resulting in efficient algorithms, including rigorously justified
probabilistic algorithms (see \cite{HK1977,Sed2002,comparison,KAJ2012} and
the references given there).

If a parameter is not locally identifiable, then it is not globally
identifiable. If it is locally identifiable, then it still remains to check
whether it is globally identifiable. Thus, it is desirable and remains a 
\emph{challenge} to have an algorithm for checking whether a locally
identifiable parameter is also globally identifiable.

Due to the importance of the challenge, there have been intensive effort on
it. Roughly stated, there have been three approaches.

\begin{itemize}
\item Taylor series approaches \cite{P1978} with termination bounds in
several  particular cases \cite{V1984,V1987,MRCW2001}.  Such bounds lead to
algorithms in these particular cases, and these algorithms can be practical
if applied to systems of small sizes.

\item Generating series approaches based on composing the vector fields
associated to  the model equations as well as a recursive approach based on 
integrals \cite{WL1982}.  See \cite{WP1996} for a comparison of the  Taylor
and generating series approaches  as well as recent version of
software GenSSI based on the generating series approach \cite%
{BCAB,CBBC,LFCBBCH}.  Having correct termination criteria for
algorithms based on this approach 
has been
an open problem (see also Example~%
\ref{ex:counterexampleGenSSI}).  

\item Differential algebra-based approaches  
 can be divided in two
groups. One group is to treat the parameters as functions with zero
derivatives and use differential elimination (see, for example, \cite{LG1994}%
). This approach can be practical if applied to systems of small sizes.

The other group is to treat the parameters as elements of the field of
coefficients and produce so-called input-output equations. This approach is
followed in COMBOS and DAISY~\cite{DAISY,DAISY_IFAC,DAISY_MED,COMBOS,MRS2016}
under an additional assumption on ``solvability'' (see \cite[%
Remark~3]{SADA2003} and Example~\ref{ex:DAISYerror}). Having an
efficient algorithm that could verify, for a system, whether this assumption
holds is an open problem.  
One of the
software packages based on this approach, DAISY, uses randomization to
increase efficiency. This can become a randomized algorithm once a
probability analysis is carried out.
\end{itemize}

\noindent Summarizing, there has been significant progress (both in
theory and algorithms) toward the challenge of checking global
identifiability.

The main contribution of this paper is to make further progress by
providing a complete and rigorous theory (Theorems~\ref{thm:alg_criterion},%
\ref{thm:probabilictic},\ref{thm:correctness}) and a symbolic-numeric-randomized algorithm (Algorithm~%
\ref{alg:identifiability}) that is\ general, reliable, and
can 
tackle problems that could not be tackled before using the existing
algorithms and software. The algorithm is ``complete'' in the sense that, for any input that satisfies our syntactically-checkable requirements, it halts with an answer that is correct with any user-chosen probability of correctness (the standard correctness notion in the theory of randomized algorithms). The algorithm is the first one for global identifiability with such guarantees.

We briefly sketch our approach, which could be viewed as a combination of the differential algebra and Taylor series approaches with a correct termination criterion. Informally, identifiability problem can be
formulated as a question about fibers of the map that sends the parameter
values and the initial conditions of the system of ODEs to the output data,
which are functions of the corresponding solution. This observation is
formalized using differential algebra in Proposition~\ref{prop:equiv_fields}%
. One way to analyze this map is to reduce it to a map between
finite-dimensional spaces. We do this by replacing the output functions by
truncations of their Taylor series. Theorem~\ref{thm:alg_criterion} provides
a criterion to find the order of truncation that contains enough information
for the identifiability checking. This criterion can be applied efficiently
using rank computation due to Proposition~\ref{prop:dominant}. After that,
the identifiability question is reduced to the question about the generic
fiber of a map between finite-dimensional varieties.

To significantly increase the efficiency at this step, instead of
considering the generic fiber, we consider a fiber over a randomly chosen
point. We estimate the probability of correctness of such an algorithm in
Theorem~\ref{thm:probabilictic}. We do it by first carefully analyzing the
set of special points of this map and then applying  the
Demillo-Lipton-Schwartz-Zippel lemma  (we believe that our
analysis could be extended to improve other software for global
identifiability such as DAISY so that the output of DAISY would be correct
with user-specified probability $p$). 

After considering the fiber over a
random point, the problem turns into checking the consistency of a system of
polynomial equations and inequations
($\neq$).
This can be performed using symbolic,
symbolic-numeric, or numeric methods. It turns out that, in practice,
Gr\"obner bases computations are efficient enough for moderate-size problems
that we encountered and significantly outperformed (unexpectedly!) numerical
algebraic geometry software such as Bertini~\cite{Bertini} in several
examples that we considered (see \ref{method:4} in Remark~\ref{rem:proj_comp}%
). Thus, in our implementation, we used Gr\"obner bases. Any new method  for 
checking the consistency of a system of polynomial
equations and inequations can be potentially used to make our algorithm even
more efficient.


We have implemented the resulting algorithm into a software called SIAN~\cite{SIAN}, which stands for ``Structural Identifiability ANalyser''.
The software is available at \url{https://github.com/pogudingleb/SIAN}.
The article~\cite{SIAN} focuses on the functionality of the software.
The present paper  focuses on a rigorous theoretical and algorithmic foundation
underlying the software.

The paper is structured as follows. 
In Section~\ref{sec:problem},  we state several notational conventions, 
give a
precise statement of the global identifiability problem 
 and illustrate it by
several examples. In Section~\ref{sec:algebraic}, we give an algebraic
criterion for global identifiability. In Section~\ref{sec:probabilistic}, we
give a probabilistic criterion for global identifiability. In Section~\ref%
{sec:algorithm}, we give an algorithm based on the criteria developed in the
previous two sections. In Section~\ref{sec:examples},
 we discuss
the performance of 
our algorithm using challenging  examples taken from the
literature and  discuss how  other existing software packages perform at those examples.

\section{Identifiability Problem}
\label{sec:problem}

In this section, we give a precise statement of the global identifiability
problem of algebraic differential models. For this, we introduce several
notions. 

\subsection{Conventions}
\label{sec:convention}
Throughout the paper, we will use the following notational conventions.
\begin{enumerate}[(a)]
\item We will use plain face for scalar variables and bold face for vector
variables. We will decorate variables (for instance by $\symbol{94}$ or
\thinspace$\symbol{126}$) to indicate particular values of the variables. For
example:%
\[%
\begin{array}
[c]{lcc}
& \text{scalar} & \text{vector}\\
\text{variable name} & \theta & \bftheta\\
\text{variable value} & \whtheta & \whbftheta%
\end{array}
\]
\item For two sets $A$ and $B$, the notation $A\subset B$ will be used to
denote that $A$ is a subset of $B$ (not necessarily a proper subset). The
notation $A\subsetneq B$ will be used to denote that $A \subset B\ \&\ A\ne B$. 
\item For the convenience of the reader, some of the notation is supplied with hyperlinks, highlighting the {\color{PineGreen}first occurrence} and  {\color{RoyalBlue}subsequent uses} in statements of results and in Definitions, Examples, etc. 
\end{enumerate}
\subsection{Definition of Identifiability}

We will start with stating what type of differential models we will be working with.

\begin{definition}
[Algebraic Differential Model]\label{def:sigma}An
   \glssymbolX{model}{algebraic differential model} is a system
\begin{equation}
\Sigma:=
\begin{cases}
\bs{x}^{\prime} & = \;\bff(\bs{x},\bfmu, u),\\
\bs{y} & =\;\bs{g}(\bs{x},\bfmu,u),\\
\bs{x}(0) & =\;\bs{x}^{\ast},
\end{cases}
\label{eq:main}%
\end{equation}
    where $\glssymbolX{statefunc}{\bff} =(f_{1},\ldots, f_{n})$ and
    $\glssymbolX{outputfunc}{\bfg} =(g_{1},\ldots, g_{m})$ with $f_{i}=f_{i}%
(\bs{x},\bfmu,u)$'s and $g_{j}=g_{j}(\bs{x},\bfmu,u)$'s
being rational functions over the field of complex numbers $\CC$.
\end{definition}

The components of the $n$-vector
        $\glssymbolX{states}{\ensuremath{\bfx}}=(x_{1},\ldots,x_{n})$
        are the state variables; the derivative of $\bfx$
with respect to time is denoted
        by~$\bs{x}^{\prime}$. The scalar function $\glssymbolX{input}{u}$
        is the input variable. The components of the $m$-vector
        $\glssymbolX{outputs}{\ensuremath{\bfy}}=(y_{1} \dd y_{m})$ 
are the output variables. The components of the $\lambda
$-vector $\glssymbolX{systemparams}{\bfmu}=(\mu_{1},\ldots,\mu_{\lambda})$ are the
system parameters. Finally, the components of the $n$-vector
$\glssymbolX{initialstates}{\bfxs}=(x_{1}^{\ast},\ldots,x_{n}^{\ast})$ are the
        initial states.

\begin{remark}
We assume that there is only one input function
 \glssymbolY{input}{\ensuremath{u}}
for a simpler presentation. However, all of our results and proofs can be
generalized straightforwardly to the case of several input functions. If the
input
 \glssymbolY{input}{$u$}
is fixed (known) or
 \glssymbolY{input}{$u$}
just does not appear, then
 \glssymbolY{input}{$u$}
can be simply omitted when our algorithms/theoretical results are used.
\end{remark}

\begin{remark}
In our approach, the initial conditions in~\eqref{def:sigma} form a part of
the parameters, as
in~\cite{CD1980,WL1982,V1983,V1984,LG1994,WP1996,ABDSC2001,MRCW2001,DAISY,comparison,CBBC,COMBOS,villaverde:16}%
, among other works. Other approaches to introducing~\eqref{def:sigma} include
considering inequalities, specified initial conditions, etc., see
e.g.~\cite{CD1980,ABDSC2001,SABDA2001,SADA2003}, among others.
\end{remark}

\begin{notation}
\label{not:1}In stating the notion of identifiability precisely, the following
notions and notation are useful. 
\begin{enumerate}[(a)]
\item We will study the identifiability of both 
\glssymbolY{systemparams}{$\bfmu$}
and
\glssymbolY{initialstates}{$\bfxs$}. Hence we collectively denote them by
\[
        \glssymbolX{parameters}{\bftheta}:= (\glssymbolY{systemparams}{\bfmu},
                \glssymbolY{initialstates}{\bfxs}) = (\theta_{1},\ldots,\theta_{s}),
\]
where $s=\lambda+n$.

\item Let \glssymbolX{analytic}{$\cinfty$} denote the set of all functions that
    are complex analytic in some neighborhood of~$t=0$.

\item We will need to avoid division by zero while evaluating
\glssymbolY{statefunc}{$\bff$}
and
\glssymbolY{outputfunc}{$\bfg$}.
        For this, let
\glssymbolX{regular}{$\Omega$}
        denote the set of all
$(\hat{\bftheta},\hat{u})\in\CC^{s}\times
\glssymbolY{analytic}{\ensuremath{\cinfty}}$
such that none of the denominators of \glssymbolY{statefunc}{$\bff$}
and \glssymbolY{outputfunc}{$\bfg$} in
        \glssymbolY{model}{$\Sigma$}
vanishes at $t=0$ after the substitution of $(\hat{\bftheta},\hat{u})$ into
$( \glssymbolY{parameters}{\bftheta},\glssymbolY{input}{u})$.

\item For $(\hat{\bftheta},\hat{u})\in \glssymbolY{regular}{\ensuremath{\Omega}},\ $let $X(\hat
{\bftheta},\hat{u}),Y(\hat{\bftheta},\hat{u})$ denote
the unique solution of the instance $\Sigma(\hat{\bftheta},\hat{u})$ with entries in \glssymbolY{analytic}{\ensuremath{\cinfty}}, which are functions of $t$ (see \cite[Theorem~2.2.2]{Hille}).

\item We will need to capture precisely the informal notions such as
``almost always''\ and ``generically''.
We will do so by using the notion of Zariski open subset. 
Define
\begin{itemize}
\item A subset $\Theta\subset\CC^s$ is called \emph{Zariski open} if
there is a non-zero polynomial
        $P\in\CC[
\glssymbolY{parameters}{\bftheta}
        ]$ such that
\[
\Theta=\{\whbftheta\in\CC^s\mid P(\whbftheta)\neq0\}.
\]

\item A subset $U\subset \glssymbolY{analytic}{\ensuremath{\cinfty}}$ is called \emph{Zariski open} if there
exist $h \in \ZZ_{\geqslant 0}$ and a non-zero 
polynomial~$P(u_0, u_1, \ldots, u_h)\in\CC[u_0, \ldots, u_h]$ such that
\[
U=\big\{\whu \in \glssymbolY{analytic}{\ensuremath{\cinfty}}\mid P\big(\whu, \whu^{(1)}, \ldots, \whu^{(h)}\big)|_{t=0}\neq0\big\}.
\]
\item Let %
    \glssymbolX{zariskiOpenAlgebraic}{$\tau(\CC^{s})$}
        denote the set of all Zariski open non-empty
    subsets of $\CC^{s}$.
\item Let
    \glssymbolX{zariskiOpenAnalytic}{$\tau(\cinfty)$}
denote the set of all Zariski open non-empty subsets
of \glssymbolY{analytic}{\ensuremath{\cinfty}}.
\end{itemize}

Note that the complement to every nonempty Zariski open set is of Lebesque measure zero~\cite[p.~83]{Neeman}.
Thus, if some property is true on a nonempty Zariski open subset, it is true ``almost always''
in the sense of measure theory.

\item For a parameter $\theta \in
\glssymbolY{parameters}{\bftheta}
$ and a subset $S \subset \CC^s$,
let  $\operatorname{proj}_\theta S$ denote the projection of $S$ 
onto the $\theta$-coordinate, that is,
\[
\operatorname{proj}_\theta S := \big\{\tilde{\theta}_1\:\big|\:\exists\, \tilde{\theta}_2,\ldots,\tilde{\theta}_s  \ \text{such that } (\tilde{\theta}_1,\ldots,\tilde{\theta}_s) \in S\big\},
\]
where we assumed that $\theta = \theta_1$ for notational simplicity.
\end{enumerate}

\end{notation}

\noindent Now we are ready to give a precise definition of identifiability.

\begin{definition}
[Identifiability]\label{def:identifiability}  Let $\Sigma$ be an algebraic differential model. A parameter $\theta\in
\glssymbolY{parameters}{\bftheta}
$ is \emph{globally (resp., locally) identifiable} if
\begin{gather*}
\exists\,\Theta\in
\glssymbolY{zariskiOpenAlgebraic}{\tau(\CC^s)}
\ \ \exists\, U\in 
\glssymbolY{zariskiOpenAnalytic}{\tau(\cinfty)}
\ \ \forall\,(\hat{\bftheta},\hat{u})\in  
\glssymbolY{regular}{\Omega}
    \cap \left(  \Theta\times U\right)  \\  \text{the size
of\ } S_{\theta}\big(  \hat{\bftheta},\hat{u}\big)  \ \text{is one
(resp., finite), }
\end{gather*}
where
\[
S_\theta\big(\hat{\bftheta},\hat{u}\big):=\operatorname{proj}_{\theta} \big\{
\tilde{\bftheta}\:\big|\:({\tilde{\bftheta}},\hat{u})\in \glssymbolY{regular}{\ensuremath{\Omega}}\ \text{
and }\  Y\big(  \hat{\bftheta},\hat{u}\big)  =Y\big(
\tilde{\bftheta},\hat{u}\big)\big\}.
\]
\end{definition}

\begin{remark}
One may notice that the above definition looks a bit different  from (although
    mathematically equivalent to, as we show in
    Proposition~\ref{prop:equiv_fields}) the ones
used in the previous literature~\cite{CD1980,WL1981,WL1982,WP1997,MRCW2001,SADA2003}. Below, we provide more details and explain the
differences and our motivation:
\begin{enumerate}[(a)]
\item ``$\theta\in
\glssymbolY{parameters}{\bftheta}$'':\
Some authors~\cite{V1983,V1984,LG1994,WP1996,MRCW2001,DAISY,MXPW2011} defined the identifiability of all the parameters.
We simplify the presentation by defining the
identifiability of a single parameter. Of course, we can  naturally
extend  it to a
subset of parameters by saying that $\bftheta^{\#}\subset
\glssymbolY{parameters}{\bftheta}$
is globally (resp., locally) identifiable if every
parameter in $\bftheta^{\#}$ is globally (resp., locally) identifiable.
See also Example~\ref{ex:predator_prey}.

\item ``$({\hat{\bftheta}},\hat{u}),\ ({\tilde{\bftheta}},\hat{u})\in
\glssymbolY{regular}{\Omega}$'':
In the previous literature, the values of 
$\glssymbolY{parameters}{\bftheta}, \glssymbolY{input}{u}$
are usually assumed to be chosen
so that the denominators of
    \glssymbolY{statefunc}{$\bff$}
and
    \glssymbolY{outputfunc}{$\bfg$}
are non-zero.
We make the assumption precise by using the ``$\in \glssymbolY{regular}{\ensuremath{\Omega}} $''.

\item ``$\exists\,\Theta\in
\glssymbolY{zariskiOpenAlgebraic}{\tau(\CC^s)}
\ \ \exists\, U\in
\glssymbolY{zariskiOpenAnalytic}{\tau(\cinfty)}$'':
Most of the papers about identifiability we have seen
use the notion ``almost  all/generic''%
\ parameter values or inputs while defining identifiability. In
Examples~\ref{ex:genericity}  and~\ref{ex:genericity2}, 
we remind the reader why it is desirable to allow
such a notion. We make the notion precise by using the phrase
``$\exists\,\Theta\in\glssymbolY{zariskiOpenAlgebraic}{\tau(\CC^s)}\ \ \exists\, U\in\glssymbolY{zariskiOpenAnalytic}{\tau(\cinfty)}$'', that is, to restrict the parameter values and inputs 
 to Zariski open sets.

\item ``$\forall\, \hat{ u}$'': 
In most of the papers about identifiability we have seen 
(for example, \cite{CD1980,WL1982,V1983,V1984,MRCW2001,DAISY,WP1997}), 
the quantification ``$\forall\,\hat{u}$'' is put inside the definition of the set 
$S$ or its analogue 
(see Proposition~\ref{prop:equiv_fields}\ref{item:old_def} for a precise formulation). 
We decided to put ``$\forall\,\hat{u}$'' outside,  due to the following reasons:
\begin{enumerate}[(i)]
\item The two resulting definitions of identifiability, though looking
different, turn out to be equivalent (see Proposition~\ref{prop:equiv_fields}).

\item The new definition putting ``$\forall\, \hat{u}$'' outside could be 
more appealing to the model users because of the following reasons:

\begin{itemize}
\item The definition putting ``$\forall\,\hat{u}$''%
\ inside  makes the model users reason that more than one (in the worst case, infinitely
many) inputs might be needed in order to identify a parameter.

\item The definition using ``$\forall\,\hat{u}$''%
\ outside makes the model users reason that only one generically chosen input will be
sufficient 
to identify a parameter.
\end{itemize}
\end{enumerate}
\end{enumerate}
Summarizing, Definition~\ref{def:identifiability} is mathematically equivalent to
the previous definition in the literature, is precisely stated, and could 
be  more appealing to the
model users.
\end{remark}

\subsection{Problem statement and illustrating examples}

Now we are ready to state the problem of identifiability.
 \begin{problem}[Global Identifiability]
\label{prob:id}\ 

\begin{description}[style=nextline] 

\item[In] 

\begin{enumerate} 

\item[$\Sigma$] :  an
    \glssymbolY{model}{algebraic differential model}
given by rational
functions $\bs{f}(\glssymbolY{states}{\ensuremath{\bfx}},\glssymbolY{systemparams}{\bfmu},\glssymbolY{input}{\ensuremath{u}}),$ and $%
\bs{g}(\glssymbolY{states}{\ensuremath{\bfx}},\glssymbolY{systemparams}{\bfmu},\glssymbolY{input}{\ensuremath{u}})$ 

\item[$\bs{\protect\theta}^\ell$] :  a subset of 
$\glssymbolY{parameters}{\bftheta}
=\glssymbolY{systemparams}{\bfmu}\cup \glssymbolY{initialstates}{\ensuremath{\bfxs}}$ 
 such that every parameter in $\bs{\protect\theta}^\ell$ is  locally identifiable 
\end{enumerate}

\item[Out] 

\begin{enumerate}

\item[$\bs{\protect\theta}^g$] : the set of all globally
identifiable parameters in $\bftheta^\ell$
\end{enumerate}
\end{description}
\end{problem}

\begin{remark}
Note that we require  every parameter in $\bs{\protect\theta}^\ell$ to be locally identifiable. The requirement is not essential for the theory developed in this paper. It is only for the sake of simple  presentation of the theory.
  Furthermore, the resulting algorithm  is not practically restrictive because there are already very efficient algorithms for assessing local identifiability (see, for example,~\cite{Sed2002,KAJ2012,AKJ2012,RKSJT2014}).
\end{remark}

In the following, we will illustrate Problem~\ref{prob:id} on
several simple examples.

\begin{example}
Consider the system

\begin{description}[style=nextline]

\item[In] 

\begin{enumerate}

\item[$\Sigma$] $=  
\begin{cases}
x_1^{\prime }=\theta_1, \\ 
y_1=x_1, \\ 
x_1(0)=\theta_2%
\end{cases}%
$ 
\item[$\bs{\protect\theta}^{\ell}$] $=  \{\theta_1,\theta_2\}$ 
\end{enumerate}

\item[Out] 

\begin{enumerate}

\item[$\bs{\protect\theta}^g$] $=\{\theta_1,\theta_2\}$ 
\end{enumerate}
\end{description}

\medskip\noindent {\sc Reason}: 
We will explicitly show that $\theta_1$ is globally identifiable,
thus also showing that it is locally identifiable (a proof that $\theta_2$
is globally identifiable can be done mutatis mutandi). For this, first note
that $Y(\glssymbolY{parameters}{\bftheta})=\theta_1 t + \theta_2$ and choose $\Theta = 
\CC^2$. For all $(\hat{\theta}_1, \hat{\theta}_2) \in \Theta$, we
have 
\begin{align*}
S_{\theta_1}(\hat{\theta}_1, \hat{\theta}_2) &=\big\{ \tilde{\theta}_1
\in \CC \: \big|\: \exists\, \tilde{\theta}_2 \in \CC \text{
such that } \hat{\theta}_1 t + \hat{\theta}_2 = \tilde{\theta}_1 t + 
\tilde{\theta}_2 \big\} \\
&= \big\{\tilde{\theta}_1 \in \CC\: \big|\: \exists\, \tilde{%
\theta}_2 \in \CC \text{ such that }\hat{\theta}_1 = \tilde{\theta%
}_1, \hat{\theta}_2 = \tilde{\theta}_2 \big\}= \big\{\hat{\theta}_1 %
\big\}.
\end{align*}
Therefore, $|S_{\theta_1}(\hat{\theta}_1, \hat{\theta}_2)|=1$, and thus $\theta_1$ is globally identifiable.
\end{example}

In Examples~\ref{ex:genericity} and~\ref{ex:genericity2}, we will see the importance of genericity (choosing open subsets) that appears in Definition~\ref{def:identifiability}.
\begin{example}\label{ex:genericity}
Consider the system

\begin{description}[style=nextline]

\item[In] 

\begin{enumerate}

\item[$\Sigma$] $=  
\begin{cases}
x_1^{\prime }=\theta_1x_1, \\ 
y_1=x_1, \\ 
x_1(0)=\theta_2%
\end{cases}%
$ 

\item[$\bs{\protect\theta}^{\ell}$] $=  \{\theta_1,\theta_2\}$ 
\end{enumerate}

\item[Out] 

\begin{enumerate}

\item[$\bs{\protect\theta}^g$] $=\{\theta_1,\theta_2\}$ 
\end{enumerate}
\end{description}

\medskip\noindent {\sc Reason}: 
We will explicitly show that $\theta_1$ is globally identifiable,
thus also showing that it is locally identifiable (a proof that $\theta_2$
is globally identifiable can be done mutatis mutandi). For this, first note
that $Y(\glssymbolY{parameters}{\bftheta})=\theta_2e^{\theta_1 t} $ and choose $\Theta = 
\big\{(z_1,z_2)\in\CC^2\:\big|\: z_2\ne 0\big\}$. For all $(\hat{\theta}_1, \hat{\theta}_2) \in \Theta$, we
have 
\begin{align*}
S_{\theta_1}(\hat{\theta}_1, \hat{\theta}_2) &=\big\{ \tilde{\theta}_1
\in \CC \: \big|\: \exists\, \tilde{\theta}_2 \in \CC \text{
such that }  \hat{\theta}_2e^{\hat{\theta}_1 t}  = \tilde{\theta}_2e^{\tilde{\theta}_1 t} 
 \big\} \\
&= \big\{\tilde{\theta}_1 \in \CC\: \big|\: \exists\, \tilde{%
\theta}_2 \in \CC \text{ such that }\hat{\theta}_1 = \tilde{\theta%
}_1, \hat{\theta}_2 = \tilde{\theta}_2 \big\}= \big\{\hat{\theta}_1 %
\big\}.
\end{align*}
Therefore, $|S_{\theta_1}(\hat{\theta}_1, \hat{\theta}_2)|=1$, and so $\theta_1$ is globally identifiable according to Definition~\ref{def:identifiability}.

However, note that choosing $\Theta = 
                \CC^2$ would result in requiring us to consider $S_{\theta_1}(\hat{\theta}_1, 0)=\CC$, which is infinite. Therefore, if the genericity (restriction to open sets) for $\hat{\bftheta}$ in Definition~\ref{def:identifiability} were not used, then this system would not be globally/locally identifiable.
\end{example}


\begin{example}\label{ex:genericity2}

Consider the system

\begin{description}[style=nextline]

\item[In] 

\begin{enumerate}

\item[$\Sigma$] $=  
\begin{cases}
  x_1^{\prime }=\theta_1\glssymbolY{input}{\ensuremath{u}}, \\ 
  y_1=x_1, \\ 
  x_1(0)=\theta_2%
\end{cases}%
$ 

\item[$\bs{\protect\theta}^{\ell}$] $=  \{\theta_1\}$ 
\end{enumerate}

\item[Out] 

\begin{enumerate}

\item[$\bs{\protect\theta}^g$] $=\{\theta_1\}$ 
\end{enumerate}
\end{description}

\medskip\noindent {\sc Reason}: Note
that $Y(\glssymbolY{parameters}{\bftheta}, \glssymbolY{input}{\ensuremath{u}})' = \theta_1 \glssymbolY{input}{\ensuremath{u}}$ and choose 
$\Theta = \CC^2$ and $U = \{ \hat{u}\in \glssymbolY{analytic}{\ensuremath{\cinfty}} \: |\: \hat{u}(0) \neq 0\}$. 
For all $(\hat{\theta}_1, \hat{\theta}_2) \in \Theta$ and $\hat{u} \in U$, we have 
\begin{align*}
S_{\theta_1}(\hat{\theta}_1, \hat{\theta}_2, \hat{u}) &= 
\big\{ \tilde{\theta}_1 \in \CC \: \big|\: \exists\, \tilde{\theta}_2 \in \CC \text{
such that }  Y(\hat{\bftheta}, \hat{u}) = Y(\tilde{\bftheta}, \hat{u})
 \big\} \\
&\subset \big\{ \tilde{\theta}_1 \in \CC \: \big|\: \exists\, \tilde{\theta}_2 \in \CC \text{
such that }  Y(\hat{\bftheta}, \hat{u})' = Y(\tilde{\bftheta}, \hat{u})' \big\}\\
&\subset \big\{ \tilde{\theta}_1 \in \CC \: \big|\: \hat{\theta}_1\hat{u}(0) = \tilde{\theta}_1\hat{u}(0) \big\} = \big\{\hat{\theta}_1\big\}.
\end{align*}
Therefore, $|S_{\theta_1}(\hat{\theta}_1, \hat{\theta}_2, \hat{u})|=1$, 
and so $\theta_1$ is globally identifiable according to Definition~\ref{def:identifiability}.

                However, note that choosing $U = \glssymbolY{analytic}{\ensuremath{\cinfty}}$ would result in requiring us to consider 
$S_{\theta_1}(\hat{\bftheta}, 0) = \CC$, which is infinite. 
Therefore, if the genericity (restriction to open sets) for $\hat{u}$ in Definition~\ref{def:identifiability} were not used, then this system would not be globally/locally identifiable.
\end{example}


\begin{example}
Consider the system

\begin{description}[style=nextline]

\item[In] 

\begin{enumerate}

\item[$\Sigma$] $=  
\begin{cases}
x_1^{\prime }=\theta_1^2, \\ 
y_1=x_1, \\ 
x_1(0)=\theta_2%
\end{cases}%
$ 

\item[$\bs{\protect\theta}^{\ell}$] $=  \{\theta_1,\theta_2\}$ 
\end{enumerate}

\item[Out] 

\begin{enumerate}

\item[$\bs{\protect\theta}^g$] $=\{\theta_2\}$ 
\end{enumerate}
\end{description}

\medskip\noindent {\sc Reason}: To see that $\theta_1$ is locally identifiable, note that $Y(\glssymbolY{parameters}{\bftheta})=\theta_1^2 t + \theta_2$ and choose $\Theta = \CC%
^2$. For all $(\hat{\theta}_1, \hat{\theta}_2) \in \Theta$, 
\begin{align*}
&S_{\theta_1}(\hat{\theta}_1, \hat{\theta}_2) \\
=& \big\{ \tilde{\theta}_1
\in \CC \:\big|\:\exists\, \tilde{\theta}_2 \in \CC \text{
such that } \hat{\theta}_1^2 t + \hat{\theta}_2 = \tilde{\theta}_1^2 t + 
\tilde{\theta}_2 \big\} \\
=& \big\{\tilde{\theta}_1 \in \CC \: \big|\:\exists\, \tilde{%
\theta}_2 \in \CC \text{ such that } \hat{\theta}_1^2 = \tilde{%
\theta}_1^2, \hat{\theta}_2 = \tilde{\theta}_2\big\} \\
=& \big\{\hat{\theta}%
_1, -\hat{\theta}_1\big\}.
\end{align*}
Therefore, $|S_{\theta_1}(\hat{\theta}_1, \hat{\theta}_2)|\leqslant 2$.
Similarly, we conclude that $|S_{\theta_2}(\hat{\theta}_1, \hat{\theta}_2)|=
1$, thus showing that $\bftheta^{\ell} = \{\theta_1,\theta_2\}$ and $%
\theta_2\in \bftheta^{g}$.

To see that $\theta_1 \notin \bftheta^{g}$, let $\Theta\subset 
\CC^2$ be non-empty open such that, for all $(\hat{\theta}_1, \hat{%
\theta}_2) \in \Theta$, $|S_{\theta_1}(\hat{\theta}_1, \hat{\theta}_2)|=1$.
Since, for all $(\whtheta_1,\whtheta_2)\in\CC^2$, $\whtheta_1 \neq 0$
implies $|S_{\whtheta_1}(\whtheta_1, \whtheta_2)| = 2$, 
\[
\Theta\subset \big\{ (\whtheta_1, \whtheta_2)\in\CC^2 \mid
\whtheta_1=0\big\},
\]
so $\Theta$ cannot be a non-empty open set.
\end{example}

In the above elementary examples (see also Examples~\ref%
{ex:DAISYerror} and~\ref{ex:counterexampleGenSSI}), we were able to solve
the 
 ODEs
explicitly and decide the identifiability directly. We now give
an example with a system to which explicit solutions are not known to us but
we are still able to check identifiability using the algorithm from this
paper. The example also illustrates  that our approach  allows working with {\em\ proper subset} of
the parameters.

\begin{example}\label{ex:predator_prey}
Consider the system (a predator-prey model)

\begin{description}[style=nextline]

\item[In] 

\begin{enumerate}

\item[$\Sigma$] $=  
\begin{cases}
x^{\prime }_1 = \theta_1 x_1 - \theta_2 x_1x_2, \\ 
x^{\prime }_2 = -\theta_3x_2 + \theta_4x_1x_2, \\ 
y_1=x_1, \\ 
x_1(0)=\theta_5, \\ 
x_2(0)=\theta_6.%
\end{cases}%
$ 

\item[$\bs{\protect\theta}^{\ell}$] $= 
\{\theta_1,\theta_3,\theta_4,\theta_5\}$ 
\end{enumerate}

\item[Out] 
\begin{enumerate}

\item[$\bs{\protect\theta}^g$] $=
\{\theta_1,\theta_3,\theta_4,\theta_5\}$ 
\end{enumerate}

\end{description}

\medskip\noindent {\sc Reason}: 
This has been determined by Algorithm~\ref{alg:identifiability} (see Section~\ref{sec:algorithm}).
Additionally, a calculation based on differential elimination terminates
with representations for $\theta_1, \theta_3, \theta_4$ that are
consequences of
        $\Sigma$
of the form 
\[
\theta_i = \frac{P_i(y, \ldots, y^{(5)})}{Q_i(y, \ldots, y^{(5)})}, \quad 
\text{for } i = 1, 3, 4.
\]
Note that, if the values of $\theta_1$ and $\theta_3$ are known, we can also
write in a simpler form: 
\[
\theta_4 = \frac{\theta_1\theta_3y^2 - \theta_3yy^{\prime }- y^{\prime
\prime }y + y^{\prime 2}}{\theta_1 y^3 - y^{\prime 2}}.
\]
One can also show directly by definition that neither $\theta_2$ nor $%
\theta_6$ is locally identifiable.
        This result tells us that, if we can only
        observe the prey population ($x_1$), then it is impossible
        to identify the rate of decrease ($\theta_2$) of the prey population
        due to the prey-predator interactions ($x_1x_2$).
        But we can identify the rate of increase ($\theta_4$) of the
        predator population.
\end{example}

\subsection{Subtleties with other approaches}

In this section, we will consider a few subtleties in applying DAISY and
GenSSI, two existing software packages to test for global identifiability (for more details on them, see the beginning of Section~\ref{sec:examples}).

\begin{example}
\label{ex:DAISYerror}Consider the system

\begin{description}[style=nextline]

\item[In] 

\begin{enumerate}

\item[$\Sigma$] $=  
\begin{cases}
x_1^{\prime }=0, \\ 
y_1=x_1, \\ 
y_2=\theta_1x_1+\theta_1^2, \\ 
x_1(0)=\theta_2.%
\end{cases}%
$ 

\item[$\bs{\protect\theta}^{\ell}$] $=  \{\theta_1,\theta_2\}$ 
\end{enumerate}

\item[Out] 
\begin{enumerate}

\item[$\bs{\protect\theta}^g$] $=\{\theta_2\}$ 
\end{enumerate}
\end{description}

\medskip\noindent {\sc Reason}: 
To see that $\theta_1$ is locally identifiable, note that 
\begin{equation*}
Y(\glssymbolY{parameters}{\bftheta}) =  
\begin{pmatrix}
\theta_2 \\ 
\theta_1 \theta_2 + \theta_1^2%
\end{pmatrix}
\end{equation*}
and $\Omega = \CC^2$ and choose $\Theta=\CC^2$. For all $(\hat{%
\theta}_1, \hat{\theta}_2) \in \Theta$, we have 
\begin{align}  \label{eq:Sab3}
S_{\theta_1}(\hat{\theta}_1, \hat{\theta}_2) &= \big\{ \tilde{\theta}_1
\in\CC\:\big|\: \exists\, \tilde{\theta}_2\in\CC \text{ s.t.
 } \tilde{\theta}_2 = \hat{\theta}_2, \tilde{\theta}_1 
\tilde{\theta}_2 + \tilde{\theta}_1^2 = \hat{\theta}_1 \hat{\theta}%
_2 + \hat{\theta}_1^2 \big\} \\
&=\big\{ \tilde{\theta}_1 \in\CC\:\big|\: \tilde{\theta}_1 
\hat{\theta}_2 + \tilde{\theta}_1^2 = \hat{\theta}_1 \hat{\theta}_2 + 
\hat{\theta}_1^2 \big\} = \big\{\hat{\theta}_1,-\hat{\theta}_1 - \hat{\theta}%
_2\big\},  \notag
\end{align}
and so $|S_{\theta_1}(\hat{\theta}_1, \hat{\theta}_2)|\leqslant 2$.
Similarly, we conclude that $|S_{\theta_2}(\hat{\theta}_1, \hat{\theta}_2)|=
1$, thus showing that $\bftheta^{\ell} = \{\theta_1,\theta_2\}$ and $%
\theta_2\in \bftheta^{g}$.

To see that $\theta_1$ is not globally identifiable, let $\Theta \subset 
\CC^2$ be non-empty open such that, for all $(\hat{\theta}_1, \hat{%
\theta}_2) \in \Theta$, $|S_{\theta_1}(\hat{\theta}_1, \hat{\theta}_2)|=1$.
Formula \eqref{eq:Sab3} implies that, for all $(\whtheta_1,\whtheta_2)\in\mathbb{%
C}^2$,  
\begin{equation*}
|S_{\theta_1}(\whtheta_1, \whtheta_2)| = 2\iff\whtheta_1 \neq - \whtheta_1 -
\whtheta_2.
\end{equation*}
Therefore, 
\begin{equation*}
\Theta\subset\big\{ (\whtheta_1, \whtheta_2) \in\CC^2\mid \whtheta_1 =
-\whtheta_1 - \whtheta_2 \big\},
\end{equation*}
so $\Theta$ cannot be non-empty open.

However, the software tool DAISY (see~\cite{DAISY,DAISY_MED,DAISY_IFAC}),
which is based on theory that uses 
input-output equations, returned that $\theta_1$ is \emph{globally identifiable}. A
possible explanation for this is as follows. 
The equation 
\begin{equation}  \label{eq:io}
y_2 - \theta_1y_1 - \theta_1^2 = 0,
\end{equation}
which is a consequence of $\Sigma$, is an input-output equation (see \cite[%
formula~(11)]{SADA2003}) in this case. The approach based on input-output
equations assumes that the coefficients of input-output equations as
polynomials with respect to the $y$'s and $u$'s and their derivatives are globally
identifiable (such an additional condition is called ``solvability'' in \cite%
[Remark~3]{SADA2003}). In this example, however, this condition does not hold.
Indeed, if the coefficients of~\eqref{eq:io}, $1$, $-\theta_1$, and $%
-\theta_1^2$, 
were globally identifiable, 
$\theta_1$ 
would be identifiable. 
However, we showed above that the values of $y_1$ and $y_2$ are not
sufficient for a unique recovery of the value of $\theta_1$ in the generic
case.
\end{example}

\begin{remark}
One can modify Example~\ref{ex:DAISYerror} so that so that the solutions of the state variables are non-constant but with
the same conclusion about ``solvability'' and the output of DAISY, e.g., by
considering 

\begin{enumerate}

\item[$\Sigma$] $=  
\begin{cases}
x_1^{\prime }=1, \\ 
x_2^{\prime }=-1, \\ 
y_1=x_1+x_2, \\ 
y_2=\theta_1(x_1+x_2)+\theta_1^2, \\ 
x_1(0)=\theta_2, \\ 
x_2(0)=\theta_3.%
\end{cases}%
$ 
\end{enumerate}
The same argument as in Example~\ref{ex:DAISYerror} shows that $\theta_1$ is locally but not globally identifiable, while DAISY would output that $\theta_1$ is globally identifiable.
\end{remark}

\begin{example}
\label{ex:counterexampleGenSSI} Consider the system

\begin{description}[style=nextline]

\item[In] 

\begin{enumerate}

\item[$\Sigma$] $=  
\begin{cases}
x_1^{\prime }= x_1, \\ 
y_1 = x_1, \\ 
y_2 = \theta_1 + \theta_1^2 x_1, \\ 
x_1(0)=\theta_2%
\end{cases}%
$ 

\item[$\bs{\protect\theta}^{\ell}$] $=  \{\theta_1,\theta_2\}$ 

\end{enumerate}

\item[Out] 
\begin{enumerate}

\item[$\bs{\protect\theta}^g$] $=\{\theta_1,\theta_2\}$ 
\end{enumerate}
\end{description}

\medskip\noindent {\sc Reason}: 
First note that  
\begin{equation*}
Y(\glssymbolY{parameters}{\bftheta})=%
\begin{pmatrix}
\theta_2\cdot e^t \\ 
\theta_1+\theta_1^2\theta_2\cdot e^t%
\end{pmatrix}
.
\end{equation*}
To show that $\theta_2$ is \emph{globally identifiable}, choose $\Theta = 
\CC^2$. For all $(\hat{\theta}_1, \hat{\theta}_2) \in \Theta$, we
have  
\[
S_{\theta_2}(\hat{\theta}_1,\hat{\theta}_2)\subset\big\{\tilde{\theta}%
_2\in\CC\:|\: \exists\,\tilde{\theta}_1\in\CC\text{ such
that }\hat{\theta}_2\cdot e^t=\tilde{\theta}_2\cdot e^t\big\} =\big\{ 
\hat{\theta}_2\big\}.
\]
Therefore, $|S_{\theta_2}(\hat{\theta}_1,\hat{\theta}_2)|\leqslant 1$. To
show that $\theta_1$ is \emph{globally identifiable}, choose $\Theta = 
\CC^2$. For all $(\hat{\theta}_1, \hat{\theta}_2) \in \Theta$, we
have  
\begin{align*}
& S_{\theta_1}(\hat{\theta}_1,\hat{\theta}_2)\\
=&\left\{\tilde{\theta}_1\in%
\CC\:\left|\: \exists\,\tilde{\theta}_2\in\CC\text{ such that }%
\begin{aligned}
&\hat{\theta}_2\cdot e^t=\tilde{\theta}_2\cdot e^t,\\ 
&\hat{\theta}_1+%
\hat{\theta}_1^2\cdot\hat{\theta}_2\cdot e^t= \tilde{\theta}_1+%
\tilde{\theta}_1^2\cdot\tilde{\theta}_2\cdot e^t 
\end{aligned}
\right.\right\}
 \\
=&\big\{\tilde{\theta}_1\in\CC\:|\:\hat{\theta}_1+\hat{\theta}%
_1^2\cdot\hat{\theta}_2\cdot e^t= \tilde{\theta}_1+\tilde{\theta}%
_1^2\cdot\hat{\theta}_2\cdot e^t \big\} \\
=&\big\{\tilde{\theta}_1\in\CC\:|\:\hat{\theta}_1= \tilde{%
\theta}_1\ \&\ \hat{\theta}_1^2\cdot\hat{\theta}_2= \tilde{\theta}%
_1^2\cdot\hat{\theta}_2 \big\}\\
=&\big\{\hat{\theta}_1\big\}.
\end{align*}
Therefore, $|S_{\theta_2}(\hat{\theta}_1,\hat{\theta}_2)|= 1$.

However, the software GenSSI~2.0 (see \cite{BCAB,CBBC,LFCBBCH}) has returned
that $\theta_1$ is \emph{locally identifiable} and is not able to conclude
the global identifiability because of limitations of the method that is used
in the software. More precisely, GenSSI~2.0 checks that the Jacobian of 
\begin{equation}  \label{eq:y0sys}
\begin{cases}
y_1(0) = \theta_2, \\ 
y_2(0) = \theta_1 + \theta_1^2 \theta_2%
\end{cases}%
\end{equation}
w.r.t. $\theta_1$ and $\theta_2$, which is 
\begin{equation*}
\begin{pmatrix}
0 & 1 \\ 
1+2\theta_1\theta_2 & \theta_1^2%
\end{pmatrix}%
, 
\end{equation*}
has rank generically equal to the number of parameters. Having positively
checked this equality, following the steps outlined in ``Electronic
supplementary material'' of \cite{BCAB} as well as by running the code,
GenSSI~2.0 does not differentiate $y_1 = x_1, y_2 = \theta_1 + \theta_1^2 x_1
$ and checks whether $\theta_1$ and $\theta_2$ are uniquely determined by $%
y_1(0)$ and $y_2(0)$ in~\eqref{eq:y0sys}. GenSSI~2.0 observes
that $\theta_1$ is not uniquely determined by $y_1(0)$ and $y_2(0)$ from~%
\eqref{eq:y0sys} and thus outputs that $\theta_1$ is only locally
identifiable, while $\theta_2$ is uniquely determined by $y_1(0)$ and so is
globally identifiable.

We will show how to remedy this limitation in checking global
identifiability for this example based on the above particular stopping
criterion used in GenSSI. It turns out that one additional derivative $%
y_2^{\prime }= \theta_1^2 x_1^{\prime }= \theta_1^2 x_1$ of $y_2 = \theta_1
+ \theta_1^2 x_1$ is sufficient to conclude the global identifiability of $%
\theta_1$ in addition to using~\eqref{eq:y0sys}: 
\begin{equation*}
\theta_1= y_2(0)-\theta_1^2\cdot y_1(0)=y_2(0)-\theta_1^2\cdot
x_1(0)=y_2(0)-y_2^{\prime }(0).
\end{equation*}
\end{example}


\section{Algebraic Criteria}

\label{sec:algebraic} 

In this section, the analytic definition of global identifiability from the
previous section will be characterized algebraically. First, we provide an
equivalence in terms of field extensions in Proposition~\ref%
{prop:equiv_fields} of Section~\ref{sec:equiv}. This section culminates in
Theorem~\ref{thm:alg_criterion} of Section~\ref{subsec:alg_crit}, which gives
a constructive algebraic criterion for global identifiability.


\subsection{Basic Terminology}
\label{sec:notation} 
A derivation $\delta$ on a commutative ring $R$ is a map $\delta: R\to R$
such that, for all $a,b \in R$, 
\begin{equation}  \label{eq:derprop}
\delta(a+b) = \delta(a)+\delta(b),\quad\delta(ab)=\delta (a)b+a\delta(b).
\end{equation}
Call $(R,\delta)$ a differential ring.
For a domain $R$, $\glssymbolX{quot}{\Quot}(R)$ denotes the field of fractions
of $R$. The
derivation $\delta$ can be extended uniquely to a derivation on
$\glssymbolY{quot}{\Quot}(R)$
using the quotient rule. An ideal $I \subset R$ is said to be 
 a differential ideal
if $a^{\prime }\in I$ for every $a \in I$. 

The 
ring of differential polynomials
in $z_1\dd z_s$ with coefficients in $\CC$ is denoted by $\CC\{z_1\dd z_s\}$.
As a ring, it is the ring of polynomials in the algebraic indeterminates 
        \begin{equation*}
        z_{1},\ldots,z_{s},z_{1}^{\prime},\ldots,z_{s}^{\prime},
            z_{1}^{\prime\prime}
        ,\ldots,z_{s}^{\prime\prime},\ldots,z_{1}^{(q)},\ldots,
            z_{s}^{(q)},\ldots. 
        \end{equation*}
A differential ring structure is defined by, for all $i$, $1 \leqslant
i\leqslant s$ and $q \geqslant0$, 
        \begin{equation*}
        \big( z_{i}^{(q)}\big) ^{\prime}:= z_{i}^{(q+1)},\quad z_{i}^{(0)}
            := z_{i}
        \end{equation*}
and extended to $\CC\{z_{1},\ldots,z_{s}\}$ by the Leibniz rule,
additivity (see~\eqref{eq:derprop}) and $c^{\prime}= 0$ for all $c \in%
\CC$. For example, 
        \begin{equation*}
        \left( 2z_{1}^{\prime2}z_{3} + 3z_{5}^{\prime\prime}\right) ^{\prime}=
        4z_{1}^{\prime}z_{1}^{\prime\prime}z_{3}+2z_{1}^{\prime2}z_{3}^{%
        \prime}+3z_{5}^{\prime\prime\prime}. 
        \end{equation*}
For all $i$, $1\leqslant i \leqslant s$, and $P \in
 \CC\{z_1\dd z_s\}$,
we define
$\glssymbolX{difforder}{\ord}_{z_i}P$
to be the largest integer $q$
such that $z_i^{(q)}$ appears in $P$ if such a $q$ exists and $-1$ if such
a $q$ does not exist. For non-empty $S, T \subset R$, we define 
        \begin{equation*}
           {T:S^\infty}
            = \{r\in R\:|\:\text{there exist } s\in S \text{ and } n\in%
                \ZZ_{\geqslant 0}\text{ such that } s^nr \in T\}. 
        \end{equation*}
If $T$ is an ideal of $R$, then $T:S^\infty$ is an ideal of $R$. For subsets 
$X \subset \CC^n$ and $J \subset \CC[x_1,\ldots,x_n]$, we
denote 
        \begin{align*}
            \glssymbolX{ideal}{I}(X)
                &:= \{f\in \CC[x_1,\ldots,x_n]\:|\: f(\bs{p})=0\text{
                for all } \bs{p}\in X\}, \\
            \glssymbolX{zero}{Z}(J)
                &:= \{\bs{p} \in \CC^n\:|\: f(\bs{p})=0\text{
                for all } f \in J\}.
        \end{align*}


\subsection{Algebraic Preparation}
\label{subsec:alg_prep}


\label{subsec:fields} We will start with introducing technical notation and
proving several auxiliary statements in this subsection.

Fix 
\glssymbolY{model}{$\Sigma$}
from Definition~\ref{def:sigma} (see Notation~\ref{not:1} as
well). Let
\glssymbolX{Q}{$Q$}
be the lcm of all denominators in
    \glssymbolY{statefunc}{$\bff$}
and
    \glssymbolY{outputfunc}{$\bfg$}.
Then all these rational functions can be written as 
        \begin{equation*}
        f_{i}=\frac{ F_{i} }{\glssymbolY{Q}{\ensuremath{Q}}}\text{ for }1\leqslant i\leqslant n,\text{ and
            }g_{j}=
        \frac{G_{j}}{\glssymbolY{Q}{\ensuremath{Q}}}\text{ for }1\leqslant j\leqslant m. 
        \end{equation*}
Denote
\glssymbolX{F}{$\bs{F}$}
$:= (F_1,\ldots, F_n)$             
and
\glssymbolX{G}{$\bs{G}$}
$:= (G_1,\ldots, G_m)$.           
Let 
        \begin{equation}  \label{eq:defR}
            \glssymbolX{R}{\R}
            := \CC[\glssymbolY{systemparams}{\bfmu}] \{ \glssymbolY{states}{\ensuremath{\bfx}}, \glssymbolY{outputs}{\ensuremath{\bfy}}, \glssymbolY{input}{\ensuremath{u}}\}
        \end{equation}
considered as a differential ring with $\glssymbolY{systemparams}{\bfmu}^{\prime }=0$ (which is set up
this way due to the time-independence of the
parameters). 
For all $i$, $1\leqslant i\leqslant n$ and $j$, $1\leqslant
j\leqslant m$, we consider $F_i$, $G_j$, and \glssymbolY{Q}{\ensuremath{Q}} as elements of $\CC[%
\glssymbolY{systemparams}{\bfmu},\glssymbolY{states}{\ensuremath{\bfx}}, \glssymbolY{input}{\ensuremath{u}}]\subset \glssymbolY{R}{\R}$. Let 
 \begin{equation}  \label{eq:defJ}
     \glssymbolX{J}{\J}
                = \left\{r \in \glssymbolY{R}{\R}\:|\:\forall ((\hat{\bfmu},
            \hat{\bs{x}}^\ast),\hat{u})\in\glssymbolY{regular}{\ensuremath{\Omega}}\ \  
            r\big( \hat{\bfmu},
            X\big((\hat{\bfmu},\hat{\bs{x}}^\ast),
            \hat u\big),Y\big((\hat{\bfmu},\hat{\bs{x}}^\ast),
            \hat u\big), \hat{u}\big) =0 \right\},
        \end{equation}
which is a (differential) ideal in
            \glssymbolY{R}{$\R$}.
For all $((\hat{\bfmu},%
\hat{\bs{x}}^\ast),\hat{u})\in\glssymbolY{regular}{\ensuremath{\Omega}}$, we have 
\begin{equation*}  \label{eq:Qne0}
\glssymbolY{Q}{\ensuremath{Q}}(\hat{\bfmu},\hat{\bs{x}}^\ast,\hat{u})\ne 0.
\end{equation*}
Therefore, for every $P\in \glssymbolY{R}{\R}$, 
\begin{equation}  \label{eq:dividebyQ}
\glssymbolY{Q}{\ensuremath{Q}} \cdot P \in \glssymbolY{J}{\J} \implies P \in \glssymbolY{J}{\J}.
\end{equation}

\begin{lemma}
\label{lem:empty_intersection} 
        $\glssymbolY{J}{\J}\cap
                \CC [\glssymbolY{systemparams}{\bfmu},
                \glssymbolY{states}{\ensuremath{\bfx}}]
                \{ \glssymbolY{input}{\ensuremath{u}}\} = \{0\}$.
\end{lemma}

\begin{proof}
Assume that there is nonzero $P\big(\bfmu, \bs{x},u,
u^\prime,\ldots,u^{(h)}\big) \in \J \cap \CC [\bfmu, \bs{x}%
]\{ u\}$.  
Since $P$ and $Q$ are nonzero polynomials, there exist $\hat{\bftheta} := (\hat{\bs{%
\mu}}, \hat{\bs{x}}^\ast) \in \CC^s$ and $\hat{u} \in
        \CC[t] \subset \CC^{\infty}(0)$ such that  
\begin{equation*}
Q\big(\hat{\bfmu}, \hat{\bs{x}}^\ast, \hat{u}(0) \big) %
\neq 0\ \text{ and }\ P\big( \hat{\bfmu}, \hat{\bs{x}}^\ast,%
\hat{u}(0), \hat{u}^{\prime }(0), \ldots, \hat{u}^{(h)}(0) \big) \neq
0,  
\end{equation*}
which contradicts~\eqref{eq:defJ}.  
\end{proof}

\begin{lemma}
\label{lem:gens_and_prime}  
We have
\begin{equation}  \label{eq:J}
\glssymbolY{J}{\J}= \big( \left(\glssymbolY{Q}{\ensuremath{Q}} x_i^{\prime }- F_i \right)^{(j)}, \left( \glssymbolY{Q}{\ensuremath{Q}} y_k - G_k
\right)^{(j)} \:\big|\: 1\leqslant i \leqslant n, 1 \leqslant k \leqslant m,
j \geqslant0\big) \colon \glssymbolY{Q}{\ensuremath{Q}}^{\infty}
\end{equation}
and 
        \glssymbolY{J}{$\J$}
is a prime ideal.  
\end{lemma}

\begin{proof}
Consider any ordering $>$ of the variables in
            $\R$
such that 
    \begin{enumerate}[(a)]

\item $y_{i}^{(k + 1)} > y_j^{(k)}$ for every $1 \leqslant i, j \leqslant m$
and $k \geqslant 0$; 

\item $y_{i}^{(k)}$ is larger that any variable in $\CC[\bs{%
\mu}]\{\bs{x}, u\}$  for every $1 \leqslant i \leqslant m$ and $k
\geqslant 0$; 

\item $x_{i}^{(k + 1)} > x_j^{(k)}$ for every $1 \leqslant i, j \leqslant n$
and $k \geqslant 0$; 

\item $x_{i}^{(k)}$ is larger that any variable in $\CC[\bs{%
\mu}]\{u\}$  for every $1 \leqslant i \leqslant n$ and $k \geqslant 0$; 
\end{enumerate}

Then the set of polynomials  
\begin{equation}  \label{eq:triangular}
\cS := \big\{\big(Q x_i^{\prime }- F_i \big)^{(j)},
\left( Q y_k - G_k \right)^{(j)} \mid 1\leqslant i \leqslant n, 1 \leqslant
k \leqslant m, j\geqslant0 \big\}
\end{equation}
is a triangular set (see~\cite[Definition~4.1 and page~10]{Hubert1}).

In order to prove~\eqref{eq:J}, we consider $\widetilde{\J} := (\cS) \colon
Q^{\infty}$.  For all $((\hat{\bfmu},\hat{\bs{x}}^\ast),%
\hat{u})\in\Omega$, it follows from the definition of $X((\hat{\bs{%
\mu}},\hat{\bs{x}}^\ast), \hat u),Y((\hat{\bfmu},\hat{%
\bs{x}}^\ast),\hat u)$ that  
\begin{align*}
Q(\hat{\bfmu},\hat{\bs{x}}^\ast, \hat u)\cdot X((\hat{%
    \bfmu},\hat{\bs{x}}^\ast), \hat u)^{\prime }-\bs{F}
\big(\hat{\bfmu},X((\hat{\bfmu},\hat{\bs{x}}%
        ^\ast), \hat u),\hat u\big)&=0, \\
Q(\hat{\bfmu},\hat{\bs{x}}^\ast, \hat u)\cdot Y((\hat{%
        \bfmu},\hat{\bs{x}}^\ast), \hat u)-
    \bs{G}
    \big(%
\hat{\bfmu},X((\hat{\bfmu},\hat{\bs{x}}^\ast),
\hat u),\hat u\big)&=0.
\end{align*}
Since
        $\J$                 
is a differential ideal, we therefore have $\cS \subset \J$.  By~%
\eqref{eq:dividebyQ},  we moreover obtain $\widetilde{\J} \subset \J$. 
For the reverse containment, consider $P\in \J$. 
Since $\mathcal{S}$ is a triangular set, let $N$ be a positive integer and  $P_0 \in \CC[%
\bfmu, \bs{x}]\{ u\}$ (see~\cite[Section~4.2]{Hubert1})
such that  
    \[
            Q^N P - P_0 \in (\mathcal{S}) \subset \mathcal{J}.
\]
Hence, $P_0 \in \J$, so $P_0 = 0$ by Lemma~\ref{lem:empty_intersection}. 
Then $Q^NP \in \widetilde{\J}$, so $P \in \widetilde{\J}$.

For the primality of
        $\J$,                 
consider $P_1$ and $P_2$ such that $P_1 \cdot P_2
\in \J$.  Let $N$ be such that $Q^N P_1$ and $Q^N P_2$ are equivalent to
elements $\widetilde{P}_1$ and $\widetilde{P}_2$ of  $\CC[\bs{%
\mu}, \bs{x}] \{ u \}$ modulo $\J$.  If $\widetilde{P}_1\widetilde{P}%
_2 = 0$, then $P_1\in \J$ or $P_2 \in \J$.  If $\widetilde{P}_1\widetilde{P}%
_2 \ne 0$, then, by Lemma~\ref{lem:empty_intersection}  and a straightforward argument, 
    \[
        Q^N P_1 \cdot Q^N P_2\notin \J,
\]
        so $P_1P_2\notin \J$, which is a contradiction.
\end{proof}


\subsection{Algebraic Criterion: Non-constructive Version}

\label{sec:equiv} 

\begin{notation}
    \label{not:SFE}  We will use the following notation.
    \begin{enumerate}[(i)]
\item    Let
    \glssymbolX{cT}{$\cT:=$} 
            $\glssymbolY{R}{\R}/\glssymbolY{J}{\J}$            
    and
            \glssymbolX{cF}{$\cF$} $:=\glssymbolY{quot}{\Quot}(\glssymbolY{cT}{\cT})$ (Recall~\eqref{eq:defR} and~\eqref{eq:defJ}).
 The latter is well-defined because $\J$ is prime, so $\cT$ is a domain.
    
\item Note that 
    \glssymbolY{cF}{$\cF$}
    is generated by the images of 
$\glssymbolY{systemparams}{\bfmu},
\glssymbolY{states}{\bfx},
\glssymbolY{input}{u}, \glssymbolY{input}{u}^\prime,\ldots$.
            
    \item
        Let   \glssymbolX{cE}{\ensuremath{\cE}}
         denote the subfield of \glssymbolY{cF}{$\cF$} generated by
        the image of $\CC\{\glssymbolY{outputs}{\ensuremath{\bfy}},\glssymbolY{input}{\ensuremath{u}}\}$.     
\item
We will denote elements of
            \glssymbolY{R}{$\R$}
and their images in
    \glssymbolY{cT}{$\cT$}
            by the same symbols.
\item  
\glssymbolY{parameters}{\ensuremath{\bftheta}}
will be understood as a tuple $(
\glssymbolY{systemparams}{\ensuremath{\bfmu}},
\glssymbolY{initialstates}{\ensuremath{\bfxs}}
)$ if it is considered as the tuple of parameters of
\glssymbolY{model}{$\Sigma$}
and  as a tuple of variables $(
\glssymbolY{systemparams}{\ensuremath{\bfmu}},
x_1, \ldots, x_n)$
if it is considered as a tuple of elements of \glssymbolY{R}{$\R$} or its subalgebras.

\item
For every $\hat{\bftheta} \in \CC^s$, let
\glssymbolX{omegahatbftheta}{
                $\Omega_{\hat{\bftheta}}$}
            $:= \{ \hat{u} \in \glssymbolY{analytic}{\ensuremath{\cinfty}} \:|\: \glssymbolY{Q}{\ensuremath{Q}}(\hat{\bftheta}, \hat{u})
            \neq 0\} \subset \glssymbolY{analytic}{\ensuremath{\cinfty}}$.
\end{enumerate}
\end{notation}

\begin{proposition}
\label{prop:equiv_fields}  For every parameter $\theta$ of system $\Sigma$, the
following statements are equivalent 
\begin{enumerate}[(a)]
\item \label{item:new_def} $\theta$ is globally (resp., locally)
identifiable according to  Definition~\ref{def:identifiability}, that is,
\begin{gather*}
\exists\,\Theta\in \glssymbolY{zariskiOpenAlgebraic}{\tau(\CC^{s})} \ \ \exists\, U\in \glssymbolY{zariskiOpenAlgebraic}{\tau(\CC^{\infty}(0))}\ \ 
\forall\,(\hat{\bftheta},\hat{u})\in  \glssymbolY{regular}{\ensuremath{\Omega}} \cap \left(  \Theta\times U\right) \\ \text{the size
of\ } S_{\theta}\big(  \hat{\bftheta},\hat{u}\big)  \ \text{is one
(resp., finite), }%
\end{gather*}
where
\[
S_\theta(\hat{\bftheta},\hat{u}):=\operatorname{proj}_{\theta} \big\{
\tilde{\bftheta}\:\big|\:({\tilde{\bftheta}},\hat{u})\in \glssymbolY{regular}{\ensuremath{\Omega}}\ \text{
and }\ Y\big(  \hat{\bftheta},\hat{u}\big)  =Y\big(
\tilde{\bftheta},\hat{u}\big)\big\}.
\]
\item \label{item:old_def} $\theta$ is globally (resp., locally)
identifiable according to much of the previous literature 
 \cite{CD1980,WL1982,V1983,V1984,MRCW2001,DAISY,WP1997}, that is,  
\[\exists\,\Theta\in\glssymbolY{zariskiOpenAlgebraic}{\tau(\CC^{s})} \ \
\forall\,\hat{\bftheta}\in   \Theta  \ \ \text{the size
of\ } S_{\theta}'\big(  \hat{\bftheta}\big)  \ \text{is one
(resp., finite),}\]
where
\[
  S'_{\theta}(\hat{\bftheta}):= 
  \operatorname{proj}_{\theta}
  \big\{ \tilde{\bftheta}\: \big|\:
\glssymbolY{omegahatbftheta}{\ensuremath{\Omega_{\tilde{\bftheta}}}}
        \neq \varnothing \ \ \text{and}\ \ \forall\, \hat{u} \in \glssymbolY{omegahatbftheta}{\ensuremath{\Omega_{\tilde{\bftheta}}}} \cap \glssymbolY{omegahatbftheta}{\ensuremath{\Omega_{\hat{\bftheta}}}} \;\; Y(\tilde{\bftheta}, \hat{u}) = Y(\hat{\bftheta}, \hat{u}) \big\}.
\]
Note that the main difference with  Definition~\ref{def:identifiability}
is that ``$\forall\, \hat{u}$'' has been put inside of $S$. 

\item \label{item:fields} the fields
    \glssymbolY{cE}{\ensuremath{\cE}}
and $\glssymbolY{cE}{\ensuremath{\cE}}(\theta)$ coincide (resp.,
the extension $\glssymbolY{cE}{\ensuremath{\cE}} \subset \glssymbolY{cE}{\ensuremath{\cE}}(\theta)$ is algebraic). 
\end{enumerate}
\end{proposition}


\begin{proof}
\underline{\ref{item:fields} $\implies$~\ref{item:new_def}.}
Assume~\ref{item:fields}. 
Considering the preimage of a minimal polynomial of $\theta$ over $\cT$ in $\R$, we obtain a nonzero polynomial (resp.,
nonzero polynomial linear in $\theta$) in $\CC\{ \bs{y}, u \}
[\theta] \cap \J$ such that its leading coefficient $\ell$ does not lie in
        $\J$.
From Lemma~\ref{lem:gens_and_prime} and the triangular set constructed in its proof, we obtain that there exists $M$ such that $Q^M \ell$ is equal to some $\ell_0
\in \CC[\bftheta]\{ u \}$ modulo $\J$. We apply Lemma~\ref%
{lem:nonzero_poly} (see next) to the polynomial $\ell$ and obtain
nonempty open $\Theta \subset \CC^s$ and $U \subset \CC^{\infty}(0)$.
We claim that Definition~\ref{def:identifiability} holds for this choice of
open sets. Consider any  $(\hat{\bftheta}, \hat{u}) \in (\Theta
\times U)\cap \Omega$. The choice of $\Theta$ and $U$ implies that $\ell_0$ does not vanish at $(\hat{\bftheta}, \hat{u})$. 
Since $Q$ does not vanish at $(\hat{\bftheta}, \hat{u})$, then $%
\ell$ does not vanish at this point, so there are only finitely many (resp.,
only one) possible values for every $\theta$ provided that
    $\hat{\bs{y}} := Y(\hat{\bftheta}, \hat{u})$ and $\hat{u}$ are fixed. Thus,
    we have~\ref{item:new_def}.

\underline{\ref{item:new_def} $\implies$~\ref{item:old_def}.}
Assume~\ref{item:new_def} and let $\Theta_0 \subset \CC^s$ and 
$U \subset \CC^{\infty}(0)$ be the open subsets from the definition.
We claim that (b) holds with  
        \[
        \Theta_1 := \Theta_0 \cap \big\{\hat{\bftheta} \:\big|\:
                Q(\hat{\bftheta}, u) \neq 0\big\}.
        \]
Assume the contrary. Then there exists $\hat{\bftheta} \in\Theta_1$ such that
$|S'_{\theta}(\hat{\bftheta})| > 1$ (resp., $S'_{\theta}(\hat{\bftheta})$ is infinite).
The inequation $Q(\hat{\bftheta}, u) \neq 0$ implies  
$\Omega_{\hat{\bftheta}} \neq \varnothing$.
Now consider local and global identifiability separately
\begin{enumerate}[(i)]
  \item \underline{Global identifiability.}
  Let $\tilde{\theta} \in S'_{\theta}(\hat{\bftheta})$ such that $\tilde{\theta} \neq \hat{\theta}$.
  Consider the corresponding $\tilde{\bftheta}$ from the definition of $S'_\theta$.
  Since $U, \Omega_{\hat{\bftheta}}, \Omega_{\tilde{\bftheta}}$ are nonempty Zariski open sets,
  there exists $\hat{u}$ in their intersection \ by Lemma~\ref{lem:countable_intersection}.
  Then $\hat{\theta}, \tilde{\theta} \in S_\theta(\hat{\bftheta}, \hat{u})$, but
  this contradicts~\ref{item:new_def}. 
  
  \item \underline{Local identifiability.}
  Let $\tilde{\theta}_1, \tilde{\theta}_2, \ldots \in S'_{\theta}(\hat{\bftheta})$ be distinct elements.
  Consider the corresponding $\tilde{\bftheta}_1, \tilde{\bftheta}_2, \ldots$ 
  from the definition of $S'_\theta$.
  Since $U, \Omega_{\tilde{\bftheta}_1}, \Omega_{\tilde{\bftheta}_2}, \ldots$ are nonempty 
  Zariski open sets,
 Lemma~\ref{lem:countable_intersection} implies that there exists $\hat{u}$ in their intersection.
  Then $\tilde{\theta}_1, \tilde{\theta}_2, \ldots \in S_\theta(\hat{\bftheta}, \hat{u})$, but
  this contradicts~\ref{item:new_def}.
  \end{enumerate}

\underline{\ref{item:old_def} $\implies$~\ref{item:fields}
for global identifiability.} Assume~\ref{item:old_def} but the containment 
    $\cE
        \subset \cE(\theta)$ is proper.
Let $Q_1 \in \CC[\bftheta]$ be a polynomial
defining the complement to $\Theta$ from~\ref{item:old_def} and
$Q_2 \in \CC[\bftheta]$ be any non-zero coefficient of $Q$ viewed as
    a polynomial in
$u$.
Due to~\cite[Theorem~2.6]{Kaplansky} applied with $K$ being $\cE$, $L$ being $\cF$, and
$s$ being $\theta$, there is a differential field $\widehat{\cF} \supset \cF$ and 
a differential automorphism $\alpha\colon \widehat{\cF} \to \widehat{\cF}$ over
   $\cE$
such that
$\alpha(\theta) \neq \theta$.
For a finitely generated subalgebra 
\[A := \CC[1/Q_1(\bftheta), \bftheta, \alpha(\bftheta), 1/(\theta - \alpha(\theta)), 1/Q_2(\bftheta), 1/Q_2(\alpha(\bftheta)) ]\] of $\widehat{\cF}$,
there exists a $\CC$-algebra homomorphism
  $\varepsilon \colon  A \to \CC$.
Let $\hat{\bftheta} := \varepsilon(\bftheta)$ and $\tilde{\bftheta} := \varepsilon(\alpha(\bftheta))$.
Then $\hat{\bftheta} \in \Theta$, $\hat{\theta} \neq \tilde{\theta}$, and
$\Omega_{\hat{\bftheta}}$ and $\Omega_{\tilde{\bftheta}}$ are nonempty.
Applying Lemma~\ref{lem:two_algebras} to the natural embedding $\cT \to \widehat{\cF}$
and the restriction of $\alpha$ to $\cT \to \widehat{\cF}$ as $\beta_2$ and $\beta_1$
and to the appropriate restriction of $\varepsilon$ as $\gamma$,
we show that  $\tilde{\theta} \in S'_{\theta}(\hat{\bftheta})$.
Thus, $|S'_{\theta}(\hat{\bftheta})|\geqslant 2$, and this contradicts~\ref{item:old_def}.

\underline{\ref{item:old_def} $\implies$~\ref{item:fields}
for local identifiability.} Assume~\ref{item:old_def} but $\theta$ is
transcendental over $\cE$.
Let $Q_1 \in \CC[\bftheta]$ be a polynomial
defining the complement to $\Theta$ from~\ref{item:old_def} and 
$Q_2 \in \CC[\bftheta]$ be any non-zero coefficient of $Q$ viewed as a polynomial in $u$.
Lemma~\ref{lem:models} implies that there exists a differential field $\widetilde{\cF} \supset \cF$ 
and automorphisms $\alpha_1 = \mathrm{id}, \alpha_2, \alpha_3, \ldots$ of $\widetilde{\cF}$
over its differential subfield $\cE$ such that $\alpha_1(\theta), \alpha_2(\theta), \ldots$
are all distinct.
Since the subalgebra $\mathcal{A}:= \CC[S_1,S_2]$, where
\begin{align*}
&S_1:= \left\{1/Q_1(\bftheta), \alpha_1(\bftheta), \alpha_2(\bftheta), \ldots , 1/Q_2(\alpha_1(\bftheta)), 1/Q_2(\alpha_2(\bftheta)), \ldots\right\},\\
&S_2 := \left\{1/(\alpha_i(\theta) - \alpha_j(\theta)) \:|\: 1 \leqslant i < j \right\},
\end{align*}
of $\widetilde{\cF}$
is countably generated, there exists a $\CC$-algebra homomorphism
$\varepsilon\colon \mathcal{A} \to \CC$~\cite[Theorem~10.34.11]{stacks-project}.
Let $\hat{\bftheta} := \varepsilon(\bftheta)$, then $\hat{\bftheta} \in \Theta$ and all $\varepsilon(\alpha_i(\theta))$ are distinct.
Moreover, for every $i \geqslant 0$, $\Omega_{\varepsilon(\alpha_i(\theta))} \neq \varnothing$.
Consider $i \geqslant 1$ and apply Lemma~\ref{lem:two_algebras} to the restrictions
of $\alpha_1$ and $\alpha_i$ to $\cT \to \widetilde{\cF}$ as $\beta_2$ and $\beta_1$
and to the appropriate restriction of $\varepsilon$ as $\gamma$.
We obtain that, for every $i \geqslant 1$, $\varepsilon(\alpha_i(\theta)) \in S'_\theta(\hat{\bftheta})$, which is a contradiction.
\end{proof}


\begin{lemma}
\label{lem:nonzero_poly} Let $P(\glssymbolY{parameters}{\bftheta},\glssymbolY{input}{u},\ldots,\glssymbolY{input}{u}^{(N)})\in%
\CC [\glssymbolY{parameters}{\bftheta}]
    \{\glssymbolY{input}{u} \}$
be nonzero. Then there exist nonempty
Zariski open subsets $\Theta\in\CC^{s}$ and $U\subset \glssymbolY{analytic}{\ensuremath{\cinfty}}$
such that, for every $\hat{\bftheta}\in\Theta$ and $\hat{u}\in U
$, the function $P(\hat{\bftheta}, \hat{u},\ldots,(\hat{u})^{(N)})$ 
is a nonzero element of $\glssymbolY{analytic}{\ensuremath{\cinfty}}$.
\end{lemma}

\begin{proof}[Proof (incorrect, see the erratum at the end of the paper)]
We write $zP(\bftheta, u, \ldots, u^{(N)}) = \sum\limits_{i =
1}^\ell c_i(\bftheta) m_i(u)$, where $m_1, \ldots, m_\ell$ are
distinct monomials from $\CC\{ u \}$ and $c_1(\bftheta),
\ldots, c_\ell(\bftheta) \in \CC[\bftheta]$.
Let $W(u) \in \CC\{ u \}$ be the determinant of the Wronskian matrix
of $m_1, \ldots, m_{\ell}$. We set 
\begin{equation*}
\Theta := \left\{ \bftheta\in \CC^s \:|\: c_1(\bs{%
\theta}) \neq 0 \right\} \ \ \text{and}\ \ U = \left\{ u\in \CC^\infty(0)
\:|\: W(u)|_{t = 0} \neq 0 \right\}. 
\end{equation*}
Since $c_1(\bftheta)$ is a nonzero polynomial, $%
\Theta\ne\varnothing$. The differential polynomial $W$ can be considered as
an algebraic polynomial 
$W\big(u, u^{\prime}, \ldots, u^{(M)}\big) \in \CC\big[u,u^{\prime}, \ldots, u^{(M)}\big]$ 
for some $M$. Let $(a_0, \ldots, a_M) \in \CC^{M
+ 1}$ be such that $W(a_0, \ldots, a_M) \neq 0$. We set $\hat{u}(t) :=
\sum\limits_{i = 0}^M a_i \frac{t^i}{i!}$. A direct computation shows that 
\begin{equation*}
W\big(\hat{u}, \ldots, \hat{u}^{(M)}\big)|_{t = 0} = W(a_0, \ldots, a_M)
\neq 0. 
\end{equation*}
Thus, $\hat{u} \in U$, so $U\ne \varnothing$. Let $\hat{\bftheta}
\in \Theta$ and $\hat{u} \in U$. If the function $P\big(\hat{\bftheta}, \hat{u}, \ldots, \hat{u}^{(N)} \big)$ is zero, it provides a
nontrivial (because of $c_1(\hat{\bftheta}) \neq 0$) linear
dependence of $m_1(\hat{u}), \ldots, m_\ell(\hat{u})$. Since $W(\hat{u}) \neq 0$%
, such a dependence does not exist due to \cite[Proposition 2.8]{Magid}.
\end{proof}


\begin{lemma}
\label{lem:solution} Let $\varphi\colon
            \glssymbolY{R}{\R}
    \rightarrow\CC$
 be a $\CC$-algebra homomorphism such that $\varphi(\glssymbolY{Q}{\ensuremath{Q}}) \neq 0$ and
 $\glssymbolY{J}{\J}\subset \func{Ker}\varphi$.
 We define power series 
  \begin{equation}  \label{eq:uphik}
    u_{\varphi}(t)=\sum\limits_{i=0}^{\infty}\frac{\varphi\big(\glssymbolY{input}{\ensuremath{u}}^{(i)}\big)}
      {i!} t^{i},\quad\bfy_{\varphi}=\sum\limits_{i=0}^{\infty}\frac{\varphi%
    \big(\glssymbolY{outputs}{\ensuremath{\bfy}}^{(i)}\big)}{i!}t^{i}.
  \end{equation}
  If $u_\varphi$ converges in some neighborhood of $0$, 
  then $\bfy_\varphi$ defines a function in $\glssymbolY{analytic}{\ensuremath{\cinfty}}$ and 
  \begin{equation*}
    \bfy_{\varphi}=Y\big(  \varphi(\glssymbolY{parameters}{\bftheta}),u_{\varphi}\big) . 
  \end{equation*}
\end{lemma}

\begin{proof}
A direct computation shows that, for every $k \geqslant 0$, $u_{\varphi}^{(k)}(0) = \varphi(u^{(k)})$. We also
define \[\bs{x}_\varphi := \sum\limits_{i = 0}^\infty \frac{%
\varphi\left( \bfx^{(i)} \right)}{i!} t^i.\]
By the theorem of existence and uniqueness of solutions for differential
equations {\cite[Theorem~2.2.2]{Hille}}, there exist unique \[X\big((\varphi(%
\bfmu),\varphi(\bs{x})), u_\varphi\big), Y\big( (\varphi(%
\bfmu),\varphi(\bs{x})), u_{\varphi} \big) \in
\CC^{\infty}(0)\] satisfying the instance $\Sigma((\varphi(\bfmu
,\varphi(\bs{x})), u_\varphi)$, as in Notation~\ref{not:1}. We
denote these functions by $\hat{\bs{x}}$ and $\hat{\bfy}$,
respectively. 
We prove that 
\begin{equation*}
(\hat{x}_i)^{(j)}(0) = \varphi\big(x_i^{(j)}\big)
\end{equation*}
for every $1 \leqslant i \leqslant n$ and $j \geqslant 0$ by induction on $j$%
. The base case is $j = 0$. Then $(\hat{x}_i)(0) = \varphi(x_i)$, because $%
\varphi(x_i)$ is the initial condition. Assume that, for all $i$ and $k$, $1
\leqslant i \leqslant n$ and $0 \leqslant k \leqslant j$, 
\begin{equation*}
(\hat{x}_i)^{(k)}(0) = \varphi\big(x_i^{(k)}\big).
\end{equation*}
We write the differential polynomial $\left( Q x_i^{\prime }-
F_i\right)^{(j)}$ in the form $Q x_i^{(j + 1)} + P$, where $P$ only involves
derivatives of $\bs{x}$ of order at most $j$. Since this
differential polynomial belongs to
        $\J$,
\begin{equation*}
\varphi\left(x_i^{(j + 1)} \right) = -\frac{\varphi(P)}{\varphi(Q)}. 
\end{equation*}
The inductive hypothesis implies that the right-hand side is equal to 
\begin{equation*}
-\frac{P(\hat{\bs{x}}, \bfmu, u_{\varphi})}{Q(\hat{%
\bs{x}}, \bfmu, u_\varphi)} \Big|_{t = 0} = (\hat{x}%
_i)^{(j + 1)}(0). 
\end{equation*}
Hence, $\hat{\bs{x}}$ and $\bs{x}_\varphi$ have the same
Taylor expansion, so coincide. Using this, one can analogously prove that $%
\hat{\bfy}$ and $\bfy_\varphi$ coincide.
\end{proof}


\begin{notation}
        We denote $\cF_{\bftheta} := \CC(\glssymbolY{parameters}{\bftheta})$ and
        $\glssymbolX{cFu}{\cF_u} := \CC(\glssymbolY{input}{\ensuremath{u}},
                \glssymbolY{input}{\ensuremath{u}}', \ldots)$,
        which are subfields of \glssymbolY{cF}{$\cF$}.
\end{notation}

\begin{lemma}\label{lem:fields_intersection}
For all
\begin{itemize}
\item differential fields $K$  containing 
\glssymbolY{cFu}{$\cF_u$}
and
    
\item $\psi_i\colon \glssymbolY{cF}{\cF} \to K$, $i = 1, 2$,
homomorphisms of differential fields over 
\glssymbolY{cFu}{$\cF_u$},
  \end{itemize}
         $\glssymbolY{input}{\ensuremath{u}},
                \glssymbolY{input}{\ensuremath{u}}', \ldots$
are algebraically independent over $K_0$, where
        \[K_0:= \CC(\psi_1(\cF_{\bftheta}) \cup \psi_2(\cF_{\bftheta}))
                \subset K.\]
\end{lemma}

\begin{proof}
  Consider the following set of indices
  \[
    H := \big\{ h\in\mathbb{Z}_{\geqslant 0} \:|\: u^{(h)} \text{ is algebraic over } \CC\big(K_0, u, \ldots, u^{(h - 1)}\big) \big\} \subset \mathbb{Z}_{\geqslant 0}.  
  \]
  
  \textbf{Claim:} \emph{The set $H$ is finite.} 
  Assume the contrary and let $H = \{h_1, h_2, \ldots\}$, where $h_1 < h_2 < \ldots$.
  Consider $M := 2s + 1$, where $s = |\bftheta|$.
  Then the definition of $H$ implies that
  \[
    \operatorname{trdeg}_\CC \CC\big(K_0, u, \ldots, u^{(h_M)}\big) \leqslant \operatorname{trdeg}_\CC K_0 + h_M + 1 - M.  
  \]
  Since $\operatorname{trdeg}_\CC K_0 \leqslant 2s$, the latter expression does not exceed $h_M$.
  On the other hand, since $\CC\big(K_0, u, \ldots, u^{(h_M)}\big)$ contains $\CC\big(u, \ldots, u^{(h_M)}\big)$, its transcendence degree over $\CC$ is at least $h_M + 1$.
  The obtained contradiction proves the claim.
  
  \textbf{Claim:} \emph{The set $H$ is empty}.
  Assume the contrary.
  Since we already know that $H$ is finite, we can consider the maximal element in $H$, say $h$.
  Let $P(z_0, \ldots, z_h) \in K_0[z_0, \ldots, z_h]$ be a polynomial of the smallest degree such that 
  $P\big(u, u', \ldots, u^{(h)}\big) = 0$ and $P$ depends on $z_h$.
  Then $\frac{\partial P}{\partial{z}_h} \big(u, u', \ldots, u^{(h)}\big) \neq 0$.
  Using $K_0' \subset K_0(u)$, we obtain
  \[
  u^{(h + 1)}\cdot \frac{\partial P}{\partial{z}_h} \big(u, u', \ldots, u^{(h)}\big) \in K_0\big(u, u', \ldots, u^{(h)}\big).
  \]
  Hence, $h + 1 \in H$, so the contradiction with the maximality of $h$ proves the claim.
  
  The latter claim proves the lemma.
\end{proof}


\begin{lemma}\label{lem:two_algebras}
Let $Q_2 \in \CC[\glssymbolY{parameters}{\bftheta}]$ be any non-zero coefficient of
    \glssymbolY{Q}{$Q$}
viewed as a polynomial in
\glssymbolY{input}{$u$}.
For every
\begin{itemize}
\item differential field $K$ over
   \glssymbolY{cFu}{$\cF_u$},
 \item
  injective $\CC\{\glssymbolY{input}{\ensuremath{u}}\}$-algebra homomorphisms $\beta_1, \beta_2\colon \glssymbolY{cT}{\cT} \to K$ 
  such that 
  $\beta_1|_{\CC\{\glssymbolY{outputs}{\ensuremath{\bfy}}, \glssymbolY{input}{\ensuremath{u}}\}} =\beta_2|_{\CC\{\glssymbolY{outputs}{\ensuremath{\bfy}}, \glssymbolY{input}{\ensuremath{u}}\}}$,  
  \item $\CC$-algebra  homomorphism $\gamma\colon B \to \CC$, where 
  \[B := \CC[\beta_1(\glssymbolY{parameters}{\bftheta}), \beta_2(\glssymbolY{parameters}{\bftheta}), 1/Q_2(\beta_1(\glssymbolY{parameters}{\bftheta})), 1/Q_2(\beta_2(\glssymbolY{parameters}{\bftheta}))] \subset K,\]
  \item  parameter $\theta \in \glssymbolY{parameters}{\bftheta}$,
  \end{itemize}
   \[\gamma\circ\beta_1(\theta) \in S'_{\theta}(\gamma\circ\beta_2(\glssymbolY{parameters}{\bftheta})).\]
\end{lemma}

\begin{proof}
  We set \[\tilde{\bftheta} := \gamma\circ\beta_1(\bftheta)\quad 
  \text{and}\quad \hat{\bftheta} := \gamma\circ\beta_2(\bftheta).\]
  Since $\gamma(Q_2(\beta_1(\bftheta))) \neq 0$ and $\gamma(Q_2(\beta_2(\bftheta))) \neq 0$,
 the  sets $\Omega_{\hat{\bftheta}}$ and $\Omega_{\tilde{\bftheta}}$ are nonempty.
  Consider $\hat{u} \in \Omega_{\hat{\bftheta}} \cap \Omega_{\tilde{\bftheta}}$.
  Lemma~\ref{lem:fields_intersection} applied to the extensions of $\beta_1$ and $\beta_2$
  to $\cF$, implies that $u, u', \ldots$ are algebraically independent over $B$.
  Hence, the homomorphism $\gamma$ can be extended to 
  $\gamma\colon B\{u\} \to \CC$ in such a way that  
  $\gamma\big(u^{(h)}\big) = \hat{u}^{(h)}(0)$ for every $h \geqslant 0$.
  Hence, since $\hat{u} \in \Omega_{\hat{\bftheta}} \cap
    \Omega_{\tilde{\bftheta}}$,
  \[
  \gamma(Q(\beta_1(\bftheta), u))\ne 0\quad\text{and}\quad\gamma(Q(\beta_2(\bftheta), u))\ne 0.
  \] 
  So, $\gamma$ can be extended to
  \[
    \gamma\colon B\{u\}\left[ 1/\beta_1(Q), 1/\beta_2(Q) \right] \to \CC,
  \]
  where the domain contains both $\beta_1(\cT)$ and $\beta_2(\cT)$ 
because $\cT \subset \CC\{u\}[\bftheta,1 / Q]$. 
  Let $\pi\colon \R \to \cT$ be the natural
  surjection.
  We apply Lemma~\ref{lem:solution} to $\gamma\circ\beta_1 \circ\pi$ and $\gamma\circ\beta_2\circ\pi$  and obtain
  \[
  Y(\tilde{\bftheta}, \hat{u}) = \sum\limits_{j = 0} \frac{\gamma\circ\beta_1\circ\pi\big(\bs{y}^{(j)}\big)}{j!} t^j\quad \text{and}\quad Y(\hat{\bftheta}, \hat{u}) = \sum\limits_{j = 0} \frac{\gamma\circ\beta_2 \circ \pi\big(\bs{y}^{(j)}\big)}{j!} t^j.
  \]
  Since $\beta_1\big(\pi\big(\bs{y}^{(j)}\big)\big) = \beta_2\big(\pi\big(\bs{y}^{(j)}\big)\big)$, 
  we have $Y(\hat{\bftheta}, \hat{u}) = Y(\tilde{\bftheta}, \hat{u})$.
  Hence, $\tilde{\theta} \in S'_{\theta} (\hat{\bftheta})$.
\end{proof}


\begin{lemma}\label{lem:countable_intersection}
  Let $A_1, A_2, \ldots$ be nonempty Zariski open subsets of $\glssymbolY{analytic}{\ensuremath{\cinfty}}$.
  Then $\bigcap\limits_{i = 1}^\infty A_i$ is nonempty.
\end{lemma}

\begin{proof}
  For every $i \geqslant 0$, let $P_i(u) \in \CC\{u\}$ denote a differential polynomial defining the complement to $A_i$.
  Let $h_i := \ord_u P_i$.
  We will inductively construct an infinite sequence of complex numbers $u_0, u_1, u_2, \ldots$ such that
  \begin{itemize}
    \item $|u_i| < 1$ for every $i \geqslant 0$;
    \item for every $i \geqslant 0$ and $j \geqslant 1$, $P_j$ does not vanish after substituting
    $u_k$ instead of $u^{(k)}$ for every $k \leqslant i$.
  \end{itemize}
  Assume that we have already constructed first $i \geqslant 0$ elements $u_0, \ldots, u_{i - 1}$ of the sequence.
  For every $j \geqslant 1$, there are only finitely many complex numbers $z$ such that 
  $P_j$ vanishes after substituting $u_k$ instead of $u^{(k)}$ for every $k < i$ 
  and $z$ instead of $u^{(i)}$. 
  Thus, at most countably many values of $u^{(i)}$ will vanish 
  at least one
  of $P_1, P_2, \ldots$.
  Since there are uncountably many complex numbers $z$ with $|z| < 1$, there exists $u_i$
  satisfying both requirements.
  
  We set $\hat{u}(t) := \sum\limits_{i = 0}^\infty u_i \frac{t^i}{i!}$.
  Since $|u_i| < 1$ for every $i \geqslant 0$, $\hat{u}$ defines an element of $\CC^\infty(0)$.
  For every $j \geqslant 1$, we have
  \[
  P_j(\hat{u})|_{t = 0} = P_j(u_0, u_1, \ldots, u_{h_j}) \neq 0.  
  \]
  Hence, $\hat{u} \in A_j$ for every $j \geqslant 1$.
\end{proof}

\begin{lemma}\label{lem:models}
 For every extension $E \subset F$  of differential fields and $a \in F$
  transcendental over $E$,
  there exists an extension of differential fields $F \subset K$ and infinitely many
  differential automorphisms
  $\alpha_1, \alpha_2, \ldots$ of $K$ over $E$ such that 
 the elements $\alpha_1(a), \alpha_2(a), \ldots$ are all distinct.
\end{lemma}

\begin{proof}
  In this proof, we will use some methods and notions from model theory of differential fields~\cite{Marker}.
  Let $K$ be a differential closure of $F$~\cite[Definition, p.~49]{Marker}.
  For every $b \in K$, we denote the differential subfield generated by $b$ and $E$ in $K$ by $E \langle b \rangle$.
  \cite[Theorem~2.9(a)]{Marker} implies that $K$ is atomic over $E$, so there is a first-order formula
  $\varphi(x)$ in the language of differential fields with parameters from $E$ such that,
 for all $b \in K$,  $\varphi(b)$ is equivalent to
  \[
  \exists\text{ differential field isomorphism } f : E\langle a\rangle\to E\langle b\rangle \text{ over } E \text{ such that } f(a)=b.
  \]
  Since $a$ is transcendental over $E$, \cite[Lemma~5.1]{Marker} implies that there are infinitely 
  many elements $a_1, a_2, \ldots \in K$ such that, for every $i\geqslant 1$, $\varphi(a_i)$ is true.
  For every $i \geqslant 1$, we introduce an isomorphism 
  $\alpha_i \colon E\langle a\rangle \to E\langle a_i \rangle$ sending $a$ to $a_i$.
  Since $K$ is a differential closure of both $E\langle a \rangle$ and $E\langle a_i \rangle$,
  \cite[Corollary~2.10]{Marker} implies that $\alpha_i$ can be extended to an isomorphism of
  differential fields $K \to K$.
\end{proof}


\subsection{Algebraic Criterion: Constructive Version}

\label{subsec:alg_crit} 

Although Proposition~\ref{prop:equiv_fields} provides us with an algebraic
criterion, it involves the quotient field of an infinitely generated
algebra, so is not constructive enough.
In this section, we will show how to find the order $\bs{h}'$ of derivatives that is sufficient to consider to make conclusion about identifiability.
This will reduce deciding identifiability to a question about the size of the generic fiber
of a projection of finite-dimensional algebraic varieties.

\begin{notation}\label{not:subs_der}
  Let $\bs{z}$ be an $\ell$-tuple of differential indeterminates, and $\bs{h} \in \mathbb{Z}_{\geqslant 0}^\ell$.
  Then we define
  \[
    \bs{z}_{\bs{h}} := \{ z_i^{(j)} \mid 0 \leqslant j < h_i,\; 1 \leqslant i \leqslant \ell \}.
  \]
\end{notation}

\begin{notation}
\label{not:Sh}  Let $\bs{h} = (h_1, \ldots, h_m) \in \mathbb{Z}%
^m_{\geqslant 0}$.  We construct 
a set of differential polynomials $\cS_{\bs{h}}
$ by the following procedure. 
    \begin{enumerate}[(a)]
\item We put $\left( \glssymbolY{Q}{\ensuremath{Q}} y_k - G_k\right)^{(j)}$ (see the beginning of Section~\ref{subsec:alg_prep} for notation) into $\cS_{\bs{h}}$
for every $1 \leqslant k \leqslant m$ and $0 \leqslant j < h_k$; 

\item \label{not3:Step2} While there exist $1 \leqslant i \leqslant n$ and $%
j \geqslant 1$ such that $x_i^{(j)}$ appears in some element of $\cS_{\bs{h}}$ but  
$\left( \glssymbolY{Q}{\ensuremath{Q}} x_i^{\prime }- F_i\right)^{(j-1)}\notin \cS_{\bs{h}}$, we
put $\left( \glssymbolY{Q}{\ensuremath{Q}} x_i^{\prime }- F_i\right)^{(j-1)}$  into $\cS_{\bs{h}}$. 
\end{enumerate}
Since the orders of all variables involved in the above procedure do not
exceed $\max\{ h_1, \ldots, h_m \}$,  the second step will terminate after a
finite number of iterations.

We introduce tuples $\Out(\bs{h}) \in \mathbb{Z}_{\geqslant 0}^m, \glssymbolX{State}{\State}(\bs{h}) \in \mathbb{Z}_{\geqslant 0}^n$, and $\glssymbolX{In}{\In}(\bs{h}) \in \mathbb{Z}_{\geqslant 0}^1$ bounding the orders of derivatives of outputs, states, and input appearing in $\cS_{\bs{h}}$:
\begin{align*}
    &\Out(\bs{h})_i := \max \{\glssymbolY{difforder}{\ord}_{y_i} P \mid P \in \cS_{\bs{h}}\} + 1, &&\text{ for } i = 1, \ldots, m,\\
    &\State(\bs{h})_i \; \; := \max \{\glssymbolY{difforder}{\ord}_{x_i} P \mid P \in \cS_{\bs{h}}\} + 1, &&\text{ for } i = 1, \ldots, n,\\
    &\In(\bs{h}) \; := \max \{\glssymbolY{difforder}{\ord}_{u} \;P \mid P \in \cS_{\bs{h}}\} + 1. 
\end{align*}
Observe that, due to the construction of $\cS_{\bs{h}}$, we will always have $\Out(\bs{h}) = \bs{h}$, so we will use $\bs{h}$ instead of $\Out(\bs{h})$ to keep the notation simple.
Using Notation~\ref{not:subs_der}, we can express all the derivatives appearing in $\cS_{\bs{h}}$ as $\bs{y}_{\bs{h}}, \bs{x}_{\State(\bs{h})},$ and $u_{\In(\bs{h})}$.

Finally, we introduce the smallest polynomial ring $\R_{\bs{h}}$ containing $\cS_{\bs{h}}$ and the corresponding ideal $\J_{\bs{h}}$ as
\[
  \R_{\bs{h}} := \CC[\glssymbolY{systemparams}{\bfmu}, \bs{x}_{\State(\bs{h})}, \bs{y}_{\bs{h}},u_{\In(\bs{h})}] \quad\text{ and }\quad \J_{\bs{h}} := (\cS_{\bs{h}}):\glssymbolY{Q}{\ensuremath{Q}}^\infty \subset \R_{\bs{h}}. 
\]
\end{notation}

\begin{remark}
\label{rem:triangular}  For every $\bs{h} \in \mathbb{Z}%
^m_{\geqslant 0}$, $\cS_{\bs{h}}$, being a subset of $\cS$, is a
triangular set for $\J_{\bs{h}}$  with respect to every ordering
described in the proof of Lemma~\ref{lem:gens_and_prime} such that  the free
variables are exactly $\glssymbolY{states}{\ensuremath{\bfx}}, \glssymbolY{systemparams}{\bfmu}$, and $u_{\In(\bs{h})}$.
\end{remark}

\begin{lemma}
\label{lem:intersection}  For every tuple $\bs{h} \in\mathbb{Z}%
_{\geqslant0}^{m}$, $\glssymbolY{J}{\J} \cap \R_{\bs{h}} = \J_{\bs{h}}$.  In
particular, $\J_{\bs{h}}$ is prime.
\end{lemma}

\begin{proof}
By Lemma~\ref{lem:gens_and_prime},
        $\J$
is prime. The statement now follows
from~\cite[Proposition~4.5]{Hubert1}.
\end{proof}

We denote the $i$-th standard basis vector in $\mathbb{Z}_{\geqslant0}^{m}$
by $\bs{1}_{i}$ and set 
\begin{equation*}
\bs{1} := \bs{1}_{1} + \ldots + \bs{1}_{m}.
\end{equation*}
For all $\bs{h} \in \mathbb{Z}^m_{\geqslant 0}$, we denote the zero
set of $\J_{\bs{h}}$ by $\glssymbolX{Zh}{Z_{\bs{h}}}$.
We will denote the $d$-dimensional affine $\CC$-space by~$\mathbb{A}^d$.

A morphism $f:X\to Y$ between two algebraic varieties
is said to be \glssymbolX{dominant}{dominant} (or dominating)
if its image $f(X)$ is dense in $Y$,
i.e., $Y=\overline{f(X)}$
(see~\cite[Definition 1, p.~48]{Mumford}).

\begin{theorem}
\label{thm:alg_criterion}  For all $\bs{h} \in \mathbb{Z}%
^m_{\geqslant 0}$ such that 

    \begin{enumerate}[(a)]
\item \label{req:1} the projection of $Z_{ \bs{h} }$ to $(%
\bs{y}_{\bs{h}}, u_{\In(\bs{h})})$-coordinates is
\glssymbolY{dominant}{dominant}, 

\item \label{req:2} the projection of $Z_{\bs{h} + \bs{1}_i}$
to the $(\bs{y}_{\bs{h} + \bs{1}_i}, u_{\In(\bs{h} + \bs{1}_i)})$-coordinates  is not dominant for every $1 \leqslant
i \leqslant m$, 
\end{enumerate}
for every non-empty subset $\bftheta^{\#} \subset 
\glssymbolY{parameters}{\bftheta}$ and every $\bs{h}' \in \mathbb{Z}^m$ such that $\bs{h}' - \bs{h} \in \mathbb{Z}^m_{> 0}$, we have
\begin{equation*}
\begin{pmatrix}
\text{every parameter in }\bftheta^{\#} \\ 
\text{is \emph{globally identifiable}}%
\end{pmatrix}
\iff 
\begin{pmatrix}
\text{the generic fiber of the projection of } Z \text{ to the} \\ 
 (\bs{y}_{\bs{h}'}, u_{\In(\bs{h}^{\prime}}))\text{-coordinates is of \emph{cardinality
one}}%
\end{pmatrix},%
\end{equation*}
where $Z$ is  the Zariski closure of the projection of $Z_{\bs{h}'}$ 
to the subspace with coordinates $(\bftheta^{\#},%
\bs{y}_{\bs{h}'}, u_{\In(\bs{h}^{\prime}}))$. 
\end{theorem}

\begin{remark}
One can check efficiently whether $\bs{h}$ satisfies requirements~%
\ref{req:1} and~\ref{req:2}
    from Theorem~\ref{thm:alg_criterion}  using
Proposition~\ref{prop:dominant} (see next).  Corollary~\ref{cor:existence} (see next) implies that
such an $\bs{h}$ always exists. The intuition behind such an  $\bs{h}$ is that we are looking for a prolongation that would determine all Taylor coefficients of  
$\bs{y}$ from the first $\bs{h}+\bs{1}$ and $\glssymbolY{In}{\In}(\bs{h}+\bs{1})$ Taylor coefficients of $\bs{y}$  and $\bs{u}$, respectively.
\end{remark}

\begin{lemma}
\label{lem:max_transc}  Let $\bs{h}$ be the tuple from the statement
of Theorem~\ref{thm:alg_criterion}.  Then 
    \glssymbolY{cE}{\ensuremath{\cE}} 
 is generated in $\cF$ (see Notation~\ref{not:SFE}) by the image of $\CC[\bs{y}_{%
\bs{h} + \bs{1}}]\{\glssymbolY{input}{\ensuremath{u}}\}$.
\end{lemma}

\begin{proof}
Consider $i$, $1 \leqslant i \leqslant m$. Since the projection of $Z_{%
\bs{h} + \bs{1}_i}$ to the $(\bs{y}_{\bs{h} + \bs{1}_i}, u_{\In(\bs{h} + \bs{1}_i)})$-coordinates is
not \glssymbolY{dominant}{dominant} and the projection of $Z_{\bs{h}}$ to the $(\bs{%
y}_{\bs{h}}, u_{\In(\bs{h})})$-coordinates is dominant, 
        \begin{equation*}
        \J_{\bs{h} + \bs{1}_i} \cap \CC[\bs{y}_{%
        \bs{h} + \bs{1}_i}, u_{\In(\bs{h} + \bs{1}_i)}]
        \neq \{0\} \text{ and } \J_{\bs{h}} \cap \CC[\bs{y}_{%
        \bs{h}}, u_{\In(\bs{h})}] = \{0\}. 
        \end{equation*}
Consider a nonzero $P \in \J_{\bs{h} + \bs{1}_i} \cap 
\CC[\bs{y}_{\bs{h} + \bs{1}_i}, u_{\In(\bs{h} + \bs{1}_i)}]$. Since $\cS_{\bs{h} + 
\bs{1}_i}$ is a triangular set, $\J_{\bs{h} + \bs{1}%
_i}=(\cS_{\bs{h} + \bs{1}_i}):Q^\infty$, and the ideal $\J_{%
\bs{h} + \bs{1}_i}$ is prime, $\cS_{\bs{h} + 
\bs{1}_i}$ is a characteristic set of $\J_{\bs{h} + 
\bs{1}_i}$ (see \cite[Definitions~5.5 and~5.10, Theorem~5.13]%
{Hubert1}). Hence, $P$ can be reduced to zero with respect to $\cS_{%
\bs{h} + \bs{1}_i}$. If $P$ does not involve $y^{(h_i)}$ (which is an element of $\bs{y}_{\bs{h} + \bs{1}_i}$, see Notation~\ref{not:Sh}),
then the reduction uses only $\cS_{\bs{h}}$, hence all coefficients
of $P$ as a polynomial in $u_{\In(\bs{h} + \bs{1}_i)} \setminus
u_{\In(\bs{h})}$ are elements of $\J_{\bs{h}} \cap \CC[%
\bs{y}_{\bs{h}}, u_{\In(\bs{h})}]$. Since the latter is
zero, every nonzero element in $\J_{\bs{h} + \bs{1}_i} \cap 
\CC[\bs{y}_{\bs{h} + \bs{1}_i}, u_{\In(\bs{h} + \bs{1}_i)}]$ depends on $y_i^{(h_i)}$. Consider such
an element $P_i$ of minimal possible degree in $y_i^{(h_i)}$. 

We will prove
by induction on $k$ that the subfield of $\cF$ generated by the image of $%
\CC[\bs{y}_{\bs{h} + k\bs{1}}]\{u\}$
coincides with the subfield of $\cF$ generated by the image of $\CC[%
\bs{y}_{\bs{h} + \bs{1}}]\{u\}$ for every $k
\geqslant 2$. Let $k = 2$. Consider $i$, $1 \leqslant i \leqslant m$, and $%
P_i^{\prime }= S_i y_i^{(h_i + 1)} + T_i$, where $S_i, T_i \in \CC[%
\bs{y}_{\bs{h} + \bs{1}}]\{u\}$ and $S_i \notin \J_{%
\bs{h} + \bs{1}}$. Hence, the image of $y_i^{(h_i + 1)}$ in $%
\cF$ belongs to the subfield generated by $\CC[\bs{y}_{%
\bs{h} + \bs{1}}]\{u\}$, so the base case is proved. Let $k
> 2$. Consider $i$, $1 \leqslant i \leqslant m$. The inductive hypothesis
implies that there are $A, B \in \CC[\bs{y}_{\bs{h} +
(k - 1)\bs{1}}]\{u\}$ such that $Ay_i^{(h_i + k - 1)} + B \in \J$,
but $A \notin \J$. Taking the derivative of $Ay_i^{(h_i + k-1)} + B$, we
obtain 
        \begin{equation*}
        Ay_i^{(h_i + k)} + A^{\prime }y_i^{(h_i + k - 1)} + B^{\prime }\in \J.
        \end{equation*}
This implies that the image of $y_i^{(h_i + k)}$ in $\cF$ belongs to the
subfield generated by the image of $\CC[\bs{y}_{\bs{h}
+ k\bs{1}}]\{u\}$. By the inductive hypothesis, the latter coincides
with the subfield generated by the image of $\CC[\bs{y}_{%
\bs{h} + \bs{1}}]\{u\}$.
\end{proof}

\begin{proof}[Proof of~Theorem~\ref{thm:alg_criterion}]  If the generic fiber of
the projection of $Z$ to the $(\bs{y}_{\bs{h}'}, u_{\In(\bs{h}^{\prime})})$-coordinates is of cardinality one,
then every $\theta \in \bftheta^{\#}$ is algebraic of degree one
over  $\CC[\bs{y}_{\bs{h}'}, u_{\In(\bs{h}^{\prime})}]$ 
modulo $\J_{\bs{h}'}$ (see \cite[page~562]{Schost}).  Then the images of
    $\bftheta%
^{\#}$ in $\cF$ belong to the subfield generated by the image if $\CC%
\{u\}[y_{\bs{h}'}]$  in $\cF$, so they belong to
   $\cE$.
Proposition~\ref{prop:equiv_fields} implies that $\bftheta^{\#}$
are globally identifiable.

Let $\bftheta^{\#}$ be globally identifiable. Proposition~\ref%
{prop:equiv_fields} implies that the image of  every $\theta \in \bs{%
\theta}^{\#}$ in $\cF$ belongs to $\cE$.  Consider $\theta \in \bs{%
\theta}^\#$.  Lemma~\ref{lem:max_transc} implies that
   $\cE$
is generated by
the image of $\CC\{ u\}[\bs{y}_{\bs{h} + \bs{1%
}}]$.  Hence, $\theta$ is algebraic of degree one over $\CC\{ u\}[y_{%
\bs{h}'}]$ modulo  $\J_{\bs{h}'} \cap \CC\{ u\}[\bs{y}_{\bs{h}'}]$%
. Let $P_{\theta}$ be such a relation.  Since $\cS_{ 
\bs{h}'}$ is a triangular set, $\J_{\bs{h}'}=(%
\cS_{\bs{h}'}):Q^\infty$, and the ideal $\J_{%
\bs{h}'}$ is prime, $\cS_{\bs{h}'}$ is a characteristic set of $\J_{\bs{h}'}$ (see \cite[Definitions~5.5 and~5.10, Theorem~5.13]{Hubert1}%
).

Therefore, $P_\theta$ can be reduced to zero with respect to $\cS_{%
\bs{h}'}$.  Consider $P_\theta$ as a polynomial in
all derivatives of %
\glssymbolY{input}{$u$}
that do not belong to $u_{\In(\bs{h}^{\prime})}$,  
i.e., do not occur in $\cS_{\bs{h}'%
}$.  Then each of these coefficients can also be reduced to zero with
respect to $\cS_{\bs{h}'}$,  and so each of these
coefficients belongs to $\J_{\bs{h}'}$, therefore, 
to $\J_{\bs{h}'}\cap \CC[\bs{y}_{%
\bs{h}'}, u_{\In(\bs{h}^{\prime})}]$. 
Taking one of them involving $\theta$, we obtain a polynomial equation of
degree one  for $\theta$ over $\CC[y_{\bs{h}'%
}, u_{\In(\bs{h}^{\prime})}]$ modulo  $\J_{ 
\bs{h}'} \cap \CC[\bs{y}_{\bs{h}'}, u_{\In(\bs{h}^{\prime})}]$.  Since this holds for every $%
\theta \in \bftheta^\#$,  the cardinality of the generic fiber of
the projection of $Z$ onto $(\bs{y}_{\bs{h}'%
}, u_{\In(\bs{h}^{\prime})})$-coordinates is one.
\end{proof}

The following statement can be proved in the same way as Theorem~\ref%
{thm:alg_criterion} using the part of the statement of Proposition~\ref%
{prop:equiv_fields} about local identifiability.

\begin{proposition}
\label{prop:alg_criterion_local}  We will use the notation from Theorem~\ref%
{thm:alg_criterion}. Parameters $\bftheta^\# \subset \glssymbolY{parameters}{\bftheta}$  are locally identifiable if and only if  the generic fiber of the
projection of $Z$ to the $\big( \bs{y}_{\bs{h}'}, u_{\glssymbolY{In}{\In}(\bs{h}')}\big)$-coordinates is finite.
\end{proposition}

\begin{proposition}
\label{prop:dominant}  For all $\bs{h} \in \mathbb{Z}^m_{\geqslant 0}
$, the projection of  $\glssymbolY{Zh}{Z_{\bs{h}}}$ to the $(\bs{y}_{%
\bs{h}}, u_{\In(\bs{h})})$-coordinates is \glssymbolY{dominant}{dominant} if and only if
the rank of the Jacobian of $\cS_{\bs{h}}$ with respect to $(%
\glssymbolY{systemparams}{\bfmu},\bs{x}_{\glssymbolY{State}{\State}(\bs{h})})$  is equal to $|\cS_{%
\bs{h}}|$  on a dense open subset of $Z_{\bs{h}}$.
\end{proposition}

\begin{proof}
We set 
\begin{equation*}
N := |\bs{y}_{\bs{h}}| + |\bs{x}_{\State(\bs{h})}| +
|\bfmu| + |u_{\In(\bs{h})}|\quad \text{and}\quad N_0 := |%
\bs{y}_{\bs{h}}| + |u_{\In(\bs{h})}|.
\end{equation*}
Then ${Z}_{\bs{h}} \subset \mathbb{A}^N$, and we denote the
projection to the $(\bs{y}_{\bs{h}}, u_{\In(\bs{h})})$%
-coordinates  by $\pi\colon \mathbb{A}^N \to \mathbb{A}^{N_0}$.  Let $M$ be
the Jacobian matrix of $\cS_{\bs{h}}$ with respect to $(\bs{x%
}_{\State(\bs{h})}, \bfmu)$.  Since $\cS_{\bs{h}}$ is a
triangular set and $\J_{\bs{h}}=(\cS_{\bs{h}}):Q^\infty$,  $%
\func{codim}Z_{\bs{h}}=|\cS_{\bs{h}}|$ by~\cite[Theorem~4.4]%
{Hubert1}.

Assume that $\pi(Z_{\bs{h}})$ is not dense in $\mathbb{A}^{N_0}$. 
Then, by~\cite[Theorem~1.25]{Shafarevich}, for every $x \in Z_{\bs{h}%
}$,  
\begin{equation}  \label{eq:ZhN0}
\dim \pi^{-1}(\pi(x)) \cap Z_{\bs{h}} > \dim Z_{\bs{h}} -
N_0.
\end{equation}
Let $p := \big(\hat{\bs{y}}_{\bs{h}}, \hat{u}_{\In(\bs{h})}\big)\in\pi(Z_{\bs{h}})$.  Then 
\begin{equation*}
\widehat{M} := M\big[ \bs{y}_{\bs{h}} \leftarrow \hat{\bs{y}}_{\bs{h}}, u_{\In(\bs{h})} \leftarrow \hat{u}_{\In(\bs{h})}\big]
\end{equation*}
can be viewed as the Jacobian of the polynomials $\cS_{\bs{h}}\big[ 
\bs{y}_{\bs{h}} \leftarrow \hat{\bs{y}}_{\bs{h}%
}, u_{\glssymbolY{In}{\In}(\bs{h})} \leftarrow \hat{u}_{\In(\bs{h})}\big]$  in $(%
\bs{x}_{\State(\bs{h})}, \bfmu)$,  
which all vanish on $\pi^{-1}(p) \cap Z_{\bs{h}}$.  Then the rank of 
$\widehat{M}$ at every point of $\pi^{-1}(p) \cap Z_{\bs{h}}$ does not
exceed the codimension of  $\pi^{-1}(p) \cap Z_{\bs{h}}$ in $%
\pi^{-1}(p)$~\cite[Theorem~16.19]{Eisenbud}, which, by~\eqref{eq:ZhN0}, is
less than 
\begin{equation*}
(N - N_0) - (\dim Z_{\bs{h}} - N_0) = \func{codim} Z_{\bs{h}%
}.
\end{equation*}

Assume that $\pi(Z_{\bs{h}})$ is dense in $\mathbb{A}^{N_0}$.  Let $%
A := \CC[\bs{y}_{\bs{h}}, u_{\In(\bs{u})}]$ and $B
:= \func{Quot}(A)$.  Hence, $\J_{\bs{h}} \cap A = 0$.  We introduce
a new variable $z$ and set $\widetilde{\R} := \R_{\bs{h}} [z]$  and $%
\widetilde{\J} := (\cS_{\bs{h}}, Q z - 1)$. Then $\widetilde{\J}
\cap \R_{\bs{h}} = \J_{\bs{h}}$ by \cite[Theorem~14, page~205%
]{CLO}.  The ideal $\widetilde{\J}$ is prime because $\widetilde{\R} / 
\widetilde{\J}$ can be obtained from the domain  $\R_{\bs{h}} / \J_{%
\bs{h}}$ by inverting the image of
    \glssymbolY{Q}{$Q$}.
Since $\widetilde{\J} \cap
A = \{0\}$, $B \cdot \widetilde{\J}$ is a prime ideal in the $B$-algebra $B
\otimes_A \widetilde{\R}$. Note that 
\begin{equation*}
\big(B \otimes_A \widetilde{\R}\big/B \cdot \widetilde{\J}\big)_{(0)}
\end{equation*}
is a regular local ring. Let $c$ denote the codimension of $B \cdot 
\widetilde{\J}$ in $B \otimes_A \widetilde{\R}$. \cite[Corollary~16.20]%
{Eisenbud} implies the ideal of $B\otimes_A \widetilde{\R}\big/B \cdot 
\widetilde{\J}$ generated by the $c\times c$ minors of $\widetilde M$, the
Jacobian of $\{ \cS_{\bs{h}}, Qz - 1\}$ with respect to  $(z, 
\bfmu, \bs{x}_{\State(\bs{h})})$, strictly contains $(0)$.  Therefore, $\func{rank}\widetilde M \geqslant c$. By \cite[Theorem~16.19]%
{Eisenbud}, $\func{rank}\widetilde M \leqslant c$.  Since $\widetilde M$ is
of the form  
\begin{equation*}
\begin{pmatrix}
Q & \ast \\ 
0 & M%
\end{pmatrix}%
,  
\end{equation*}
its rank is equal to the rank of $M$ modulo $\J_{\bs{h}}$ plus one. 
Since $\widetilde{\J}\cap A =\{0\}$, the codimension of $B\widetilde{\J}$ in 
$B \otimes_A \widetilde{\R}$ is equal to the codimension  of $\widetilde{\J}$
in $\widetilde{\R}$, and the latter is equal to one plus the codimension of $%
\J_{\bs{h}}$ in $\R_{\bs{h}}$.  Thus, the rank of $M$ modulo 
$\J_{\bs{h}}$ is equal to $|\cS_{\bs{h}}|$.
\end{proof}

\begin{corollary}
\label{cor:existence}  There exists $\bs{h} \in \mathbb{Z}%
_{\geqslant 0}^m$ satisfying requirements~\ref{req:1} and~\ref{req:2}
from Theorem~\ref{thm:alg_criterion}.  Moreover, for every $\bs{h}%
=(h_1,\ldots,h_m)$ satisfying requirement~\ref{req:1}, $h_1+\ldots+h_m
\leqslant s=|\glssymbolY{parameters}{\bftheta}|$.
\end{corollary}

\begin{proof}
Let $\bs{h} = (h_1,\ldots,h_m)$ satisfy requirement~\ref{req:1}
from Theorem~\ref{thm:alg_criterion}.  By Proposition~\ref{prop:dominant},  $%
|\cS_{\bs{h}}| \leqslant |\bs{x}_{\glssymbolY{State}{\State}(\bs{h})}| + |%
\bfmu|$.  Since also $|\cS_{\bs{h}}|$ is equal to the sum
of $|\bs{y}_{\bs{h}}|$ and the number of the elements of $%
\bs{x}_{\State(\bs{h})}$  of nonzero order, 
\begin{equation*}
|\cS_{\bs{h}}|\geqslant|\bs{y}_{\bs{h}}| + |%
\bs{x}_{\State(\bs{h})}| - n = h_1 + \ldots + h_m + |\bs{x}_{\State(\bs{h})}| - n.
\end{equation*}
Therefore, $h_1 + \ldots + h_m \leqslant n + |\bfmu| = s$.  Thus, $%
h_1 + \ldots + h_m \leqslant s$.

Let $\bs{H}$ be the set of all $\bs{h} \in \mathbb{Z}%
_{\geqslant 0}^{m}$ that satisfy requirement~\ref{req:1} from Theorem~\ref%
{thm:alg_criterion}. Since $h_1 + \ldots + h_m \leqslant s$ for every $%
\bs{h} = (h_1,\ldots,h_m) \in \bs{H}$,  $\bs{H}$ is
finite.  Then it contains a maximal element with respect to the
coordinate-wise partial ordering.  Such an element also satisfies
requirement~\ref{req:2} from Theorem~\ref{thm:alg_criterion}.
\end{proof}

\section{Probabilistic Criterion}
        \label{sec:probabilistic} 

The goal of the present section is to provide theoretical grounds for our
probabilistic algorithm for checking the fiber condition in the statement of
Theorem~\ref{thm:alg_criterion} efficiently.
For an example of using these results in Algorithm~\ref{alg:identifiability}, we refer to Example~\ref{ex:alg}.

\begin{notation}
\label{not:prob} Let $\bs{h} \in \mathbb{Z}^m_{\geqslant 0}$ be a
tuple from the statement of Theorem~\ref{thm:alg_criterion}, $\bs{%
\theta}^{\#}$ be any non-empty subset of $\glssymbolY{parameters}{\bftheta}$, and 
$\bs{h}' \in \mathbb{Z}^m$ be a tuple such that $\bs{h}' - \bs{h} \in \mathbb{Z}_{> 0}^m$. 
We introduce the following affine spaces:
    \begin{enumerate}[(i)]
\item the ambient space $\glssymbolX{V}{V}$ with coordinates $( \glssymbolY{systemparams}{\bfmu},\bs{%
x}_{\glssymbolY{State}{\State}(\bs{h}')}, \bs{y}_{ 
\bs{h}'}, u_{\glssymbolY{In}{\In}(\bs{h}^{\prime})})$; 

\item the input-output space $\glssymbolX{Vio}{V_{io}} \subset V$ with coordinates $(%
\bs{y}_{\bs{h}'}, u_{\In(\bs{h}^{\prime})})$; 

\item the identification space $\glssymbolX{Vsharp}{V_\#} \subset V$ with coordinates $\big(%
\bftheta^{\#},\bs{y}_{\bs{h}'},
u_{\In(\bs{h}^{\prime})}\big)$, where $\bftheta^{\#}
\subset \glssymbolY{parameters}{\bftheta}$ is the set of parameters whose global
identifiability we would like to check.
\end{enumerate}
Below are several natural projections between them: 
\begin{itemize}
  \item $\glssymbolX{piio}{\pi_{io}} \colon V \to V_{io}$ is the projection of what we consider to what we observe;
  \item $\glssymbolX{pisharp}{\pi_\#} \colon V \to V_\#$ is the projection of what we consider to what we care about;
  \item $\glssymbolX{pi}{\pi} \colon V_\# \to V_{io}$ is the projection of what care about to what we observe.
\end{itemize}
Using this notation, we can define $Z$ from the statement of Theorem~\ref%
{thm:alg_criterion} as \[\glssymbolX{Z}{Z} = \overline{\pi_\#(\glssymbolY{Zh}{Z_{\bs{h}'}})}.\]
\end{notation}

\begin{theorem}
\label{thm:probabilictic}  
If parameters $\bftheta^{\#}$ are locally identifiable, then 
there exists a polynomial $P \in \CC[V_\#]$
such that 
    \begin{enumerate}[(i)]
\item $P$ does not vanish everywhere on $Z$; 

\item $\deg P \leqslant (2 + |\bftheta^{\#}|) \deg Z_{\bs{\bs{h}'}}$; 

\item \label{it:lastreq} For all $a \in Z$ such that $P(a) \neq 0$, the
following statements are equivalent: 

            \begin{enumerate}[(a)]
\item \label{thm2:1} every $\theta \in \bftheta^{\#}$ is globally
identifiable, 

\item \label{thm2:2}
        $\big(\J_{\bs{h}'} +
                     \glssymbolY{ideal}{I}(\pi(a))\cdot \CC[V] \big) \cap \CC[V_{\#}] =  \glssymbolY{ideal}{I}(a)$,

\item \label{thm2:3} the zero set of $\big( \J_{\bs{h}'} +  \glssymbolY{ideal}{I}(\pi(a))\cdot \CC[V] \big) \cap \CC[V_{\#}]$ is $\{a\}$. 
\end{enumerate}
\end{enumerate}
\end{theorem}

Before proving Theorem~\ref{thm:probabilictic}, we formulate and prove two
technical lemmas, in which we do \emph{not} use any notation introduced
beyond Section~\ref{sec:notation}. 
We will denote the $d$-dimensional affine and  projective $\CC$-spaces by $\mathbb{A}^d$ and $\mathbb{P}^d$, respectively.


\begin{lemma}
\label{lem:finite_fiber}  Let $n$, $m$, and $r$ be non-negative integers
such that $n = m + r$, $X \subset \mathbb{A}^n = \mathbb{A}^r \times \mathbb{%
A}^m$ an irreducible variety, and $\pi \colon \mathbb{A}^n \to \mathbb{A}^m$
the projection onto the second component.  Assume that the generic fiber of $%
\pi|_X$ is finite. 
    \begin{enumerate}[(i)]
\item \label{item:1lem7} Then there exists a proper subvariety $Y \subset 
\overline{\pi(X)}$ such that 
            \begin{enumerate}[(a)]
\item $\deg Y \leqslant \deg X$ and 

\item for every point $p \in \overline{\pi(X)} \setminus Y$, there exists a
closed (in the standard topology) ball $B \subset \mathbb{A}^m$  centered at 
$p$ such that $\pi^{-1}(B) \cap X$ is compact (in the standard topology) and 
$\pi^{-1}(p)\cap X\ne \varnothing$. 
\end{enumerate}

\item \label{item:2lem7} Then there exists a hypersurface $H \subset \mathbb{%
                A}^n$ (possibly, empty) not containing $X$ such that 
            \begin{enumerate}[(a)]
\item $\deg H \leqslant r\cdot ( \deg X-1)$ and 

\item for every $p \in X \setminus H$, there exists a closed ball (in the
standard topology) $B \subset \mathbb{A}^m$  centered at $\pi(p)$ such that $%
\pi$ defines a bijection between the connected component (in the standard
topology) of $\pi^{-1}(B) \cap X$  containing $p$ and its image. 
\end{enumerate}
\item \label{item:3lem7} If the generic fiber of $\pi|_X$ is of cardinality
        one,  then there exists a hypersurface $H \subset \mathbb{A}^n$ (possibly empty) not
containing $X$ such that 
            \begin{enumerate}[(a)]
\item $\deg H \leqslant r\cdot \deg X$ and 

\item for every $p \in X \setminus H$,  
\begin{equation*}
        \glssymbolY{ideal}{I}(X) +  \glssymbolY{ideal}{I}(\pi(p))\cdot\CC[\mathbb{A}^n] =  \glssymbolY{ideal}{I}(p).
\end{equation*}
\end{enumerate}
\end{enumerate}
\end{lemma}

\begin{proof}
We introduce coordinates $\bs{x} := (x_1, \ldots, x_r)$ in $\mathbb{A%
}^r$ and $\bs{y} := (y_1, \ldots, y_m)$ in $\mathbb{A}^m$.  The
condition that the generic fiber of $\pi|_X$ is finite implies that  $\dim X
= \dim \pi(X)$~\cite[Theorem 1.25(ii)]{Shafarevich}, so each element of $%
\bs{x}$ is algebraic over $\CC[\bs{y}]$ modulo $I(X)$%
. We prove each claim.

    \begin{enumerate}[(i)]
\item Embed $\mathbb{A}^r$ into $\mathbb{P}^r$, then $\pi$ can be extended
to $\pi_\mathbb{P} \colon \mathbb{P}^r \times \mathbb{A}^m \to \mathbb{A}^m$%
.  Let  
\begin{equation*}
H_{\infty} := \mathbb{P}^r \times \mathbb{A}^m \setminus \mathbb{A}^r \times 
\mathbb{A}^m  
\end{equation*}
and $\overline{X}^\mathbb{P}$ be the closure of $X$ in $\mathbb{P}^r \times 
\mathbb{A}^m$.  We set  
\begin{equation*}
Y := \pi_{\mathbb{P}} \big( \overline{X}^\mathbb{P} \cap H_\infty \big).  
\end{equation*}
By \cite[Corollary~9, page 431]{CLO}, $Y\subset\overline{\pi(X)}$.  Since $%
\pi|_X$ has finite generic fiber, $\dim X = \dim \overline{\pi(X)}$ by \cite[%
Theorem~1.25(ii)]{Shafarevich}.  On the other hand,  
\begin{equation*}
\dim Y \leqslant \dim \overline{X}^\mathbb{P} \cap H_{\infty} \leqslant \dim
X - 1,  
\end{equation*}
hence $Y$ is a proper subvariety in $\overline{\pi(X)}$.  Also,  
\begin{equation*}
\deg Y \leqslant \deg \overline{X}^\mathbb{P} \cap H_{\infty} \leqslant \deg
X.  
\end{equation*}

Let $p \in \overline{ \pi(X) } \setminus Y$ and $B \subset \mathbb{A}^m$ a
closed ball (of a positive radius) centered at $p$  such that $\pi^{-1}(B)
\cap X$ is compact.  Such a $B$ exists by \cite[Lemma~2]{LazardRouillier}. 
Moreover, $\pi^{-1}(B) \cap X\ne \varnothing$. Indeed, \cite[Theorem~1,
page~58]{Mumford}  implies that $\pi(X)$ is dense in $\overline{\pi(X)}$
with respect to the standard topology,  so $B$ contains at least one point
of $\pi(X)$.  Let us show that $\pi^{-1}(p)\cap X\ne \varnothing$. For this,
let $p_1, p_2, \ldots \in \pi(X) \cap B$ be a sequence of points  converging
to $p$, which exists because $p \in \overline{\pi(X)}$. Let $q_1,q_2,\ldots
\in \pi^{-1}(B) \cap X$ be such that $\pi(q_i)=p_i$, $i\geqslant 1$. Since $%
\pi^{-1}(B) \cap X$ is compact, there exists a converging subsequence  of
the sequence $q_1, q_2, \ldots$ with a limit $q \in X$. Since $p_1, p_2,
\ldots$ converge to $p$ and $\pi$ is continuous,  $\pi(q) = p$.

\item We choose a subset $S \subset \{ \bs{y} \}$ that is a
transcendence basis of $\CC[\bs{y}]$ modulo $I(X)$
over $%
\CC$.  For every $i$, $1 \leqslant i \leqslant r$, the projection of $%
X$ to the $(x_i, S)$-coordinates is an irreducible hypersurface  of degree
at most $\deg X$.  We denote its defining irreducible polynomial by $%
P_i(x_i, \bs{y})$.  For every $i$, $1\leqslant i \leqslant r$, $%
\frac{\partial}{\partial x_i} P_i$ does not vanish everywhere on $X$. 
Hence,  
\begin{equation*}
P:= \prod\limits_{i = 1}^r \frac{\partial}{\partial x_i} P_i  
\end{equation*}
does not vanish everywhere on $X$. Let $H := Z(P)$.
We prove that $H$ satisfies the requirements.

Consider $p \in X \setminus H$, and let $p = (\hat{\bs{x}}, 
\hat{\bs{y}})$.  Let $\widetilde{X} := Z(P_1, \ldots, P_r)$, a
subvariety in $\mathbb{A}^r \times \mathbb{A}^m$.  Then $X \subset 
\widetilde{X}$.  Since $P(p)\neq 0$, the Jacobian of $P_1, \ldots, P_r$ with
respect to $\bs{x}$  is invertible at $p$.  The implicit function
theorem~\cite[Theorem~3.1.4]{ComplexAnalysis} (applied to $X = \mathbb{A}^m$%
, $Y = \mathbb{A}^r$, $f = (P_1, \ldots, P_r)$)  implies that there exist
neighborhoods $U_1 \subset \mathbb{A}^m$ and $U_2 \subset \mathbb{A}^r$ of  $%
\hat{\bs{y}}$ and $\hat{\bs{x}}$, respectively, such that $S(%
\widetilde{X} \cap U_2\times U_1)$  is a graph of some function $%
\varphi\colon U_1 \to U_2$, where $S(\bs{a},\bs{b}):=(%
\bs{b},\bs{a})$, $\bs{a}\in \mathbb{A}^r$, $%
\bs{b} \in \mathbb{A}^m$.  Hence, $\pi$ defines a bijection between $%
\widetilde{X} \cap U_2\times U_1$ and its image under $\pi$.

Let $B \subset U_{1}$ be a closed ball centered at $\pi(p)$ and $C$ the
connected component of  $\widetilde{X} \cap \pi^{-1}(B)$ containing $p$. 
Let $G_B$ denote the graph of $\varphi|_B$.  We claim that $C \subset S(G_B)$%
.  Consider the intersection  
\begin{equation*}
C \cap S(G_B) = C\cap \big(\widetilde{X} \cap (U_2\times B)\big).  
\end{equation*}
Since $C \subset \widetilde{X}$, the latter intersection is equal to $C\cap
(U_2 \times B)$.  Since $\pi(C) \subset B$, the latter intersection is the
same as $C\cap (U_2 \times U_1)$.  Hence, 
\begin{equation}  \label{eq:SGB}
C \cap S(G_B) = C\cap (U_2\times U_1).
\end{equation}
Since $S(G_B)$ is closed, $U_2 \times U_1$ is open, and $C$ is connected,  
\begin{equation*}
C\cap (U_2\times U_1)=\varnothing\quad\text{or}\quad C\cap (U_2\times U_1)=C.
\end{equation*}
Since $p\in C\cap (U_2\times U_1)$, by~\eqref{eq:SGB}, we have $C \cap
S(G_B)=C$, and so $C \subset S(G_B)$.  Since $S(G_B)$ maps bijectively onto $%
\pi(S(G_B))$, $C$ also maps bijectively onto $\pi(C)$.  Since $X \subset 
\widetilde{X}$, the connected component $C_0$ of $\pi^{-1}(B) \cap X$
containing $p$ is a subset of $C$,  so $\pi$ defines a bijection from $C_0
$ to $\pi(C_0)$.

\item Since each element of $\bs{x}$ is algebraic over $\CC[%
\bs{y}]$ modulo $I(X)$,
there exists a representation of $I(X)$ as
a triangular set $P_1, \ldots, P_r,\ldots$ with respect to an ordering of
the form  
\begin{equation*}
x_1 > x_2 > \ldots > x_r > \bs{y} \text{ in some order}  
\end{equation*}
so that $x_i$ is the leading variable of $P_i$, $1\leqslant i\leqslant r$. 
Since the cardinality of the generic fiber is one, $\deg_{x_i} P_i = 1$ for
every $i$, $1 \leqslant i \leqslant r$ (see \cite[page~562]{Schost}).  Since 
$P_i$ is reduced with respect to $P_{i + 1}, \ldots, P_{r}$, $P_i$ does not
depend on any variable in $\bs{x}$  except for $x_i$.  \cite[%
Theorem~2]{SchostDahan} implies that the triangular set can be chosen in
such a way that the degrees of the  coefficients of $P_i$ as a polynomial
with respect to $x_i$ do not exceed $\deg X$.  We set $P$ to be the product
of the leading coefficients of $P_1, \ldots, P_r$ and $H := Z(P)$
in $\mathbb{A}^n$.  Then \[\deg P \leqslant r\cdot \deg X.\]

Let $p \in X \setminus H$.  Then $\pi(p)$ is not a zero of $P$.  We observe
that $P_1, \ldots, P_r \in I(X)$ since $I(X)$ is a prime ideal.  Then the
ideal 
\begin{equation*}
J := I(X) + I(\pi(p))\cdot \CC[\mathbb{A}^n]
\end{equation*}
contains a polynomial $P_i(\mathbf{x}, \pi(p)) = a_i x_i + b_i$, where $%
a_i\in \CC^\ast$, $b_i \in \CC$, for every $1 \leqslant i
\leqslant r$.  Since $p$ is a common zero of all polynomials in $J$, $\frac{%
-b_i}{a_i}$ is the value of the $x_i$-th coordinate of $p$.  Since $I(\pi(p))
$ also contains a polynomial of the form $y_j - c_j$, where $c_j \in \mathbb{%
C}$ is the value of the $y_j$-th coordinate of $p$, for every $j$, $1
\leqslant j \leqslant m$, the ideal $J$ is simply $I(p)$.\qedhere
\end{enumerate}
\end{proof}


\begin{lemma}
\label{lem:surjective}  Let $n$, $m$, and $r$ be non-negative integers such
that $n=m+r$, $X \subset \mathbb{A}^n = \mathbb{A}^r \times \mathbb{A}^m$ an
irreducible variety, and $\pi \colon \mathbb{A}^n \to \mathbb{A}^m$ the
projection onto the second component.  Then there exists a proper subvariety 
$Y \subset \overline{\pi(X)}$ such that 

    \begin{enumerate}[(i)]
\item $\deg Y \leqslant \deg X$ and 

\item for every $p \in \overline{\pi(X)} \setminus Y$, $\pi^{-1}(p) \cap
X\ne\varnothing$. 
\end{enumerate}
\end{lemma}

\begin{proof}
If $\dim \overline{\pi(X)}=0$, then the irreducibility of $X$ implies that $%
\overline{\pi(X)}$ is a single point. Hence, we can choose $Y=\varnothing$.
Before finishing the proof, we will first prove the following.

\begin{claim}
Assume that both the generic fiber of $\pi|_X$ and $\overline{\pi(X)}$  are
not zero-dimensional.  Then there exists a non-empty open subset $\mathcal{U}
$ in the space of all hyperplanes in $\mathbb{A}^n$ such that, for all $H\in 
\mathcal{U}$, 
\begin{itemize}
\item $X \not\subset H$, 

\item $H\cap X$ is irreducible, and 

\item $H\cap X$ projects \glssymbolY{dominant}{dominantly} onto $\pi(X)$. 
\end{itemize}
\end{claim}

\begin{proof}
In this case, $\dim X\geqslant 2$ by~\cite[Theorem~1.25(ii)]{Shafarevich},
so Bertini's theorem~\cite[III.7.(i)]{Baldassarri}  implies that $H \cap X$
is irreducible for a generic $H$.  Let $U$ be a non-empty open subset of $%
\overline{\pi(X)}$ such that, for every $p \in U$, 
\begin{equation*}
\dim \pi^{-1}(p) \cap X > 0.
\end{equation*}
Such a $U$ exists by~\cite[Theorem~1.25(ii)]{Shafarevich}.  Since $U$ is
dense in $\overline{\pi(X)}$, there exists a set of points $p_1, \ldots, p_N
\in U$ for some $N$  such that, if a polynomial  
of degree at most $\deg X$  vanishes at $p_1, \ldots, p_N$, then it vanishes
on $\overline{\pi(X)}$.  Consider a hyperplane $H$ such that 
    \begin{enumerate}[(a)]
\item \label{cond:1} $H \cap X$ is irreducible and 
\item \label{cond:2} for every $i$, $1 \leqslant i \leqslant N$,  
\begin{equation*}
H \cap \pi^{-1}(p_i) \cap X\ne\varnothing.
\end{equation*}
\end{enumerate}
Since conditions~\ref{cond:1} and~\ref{cond:2} are generic, the
conjunction is also generic.  Let $Y := \overline{\pi(H \cap X)}$, then $%
p_1, \ldots, p_N \in Y$.  Since $\deg Y \leqslant \deg X$, \cite[%
Proposition~3]{Heintz} implies that $Y$ can be defined by polynomials of
degree at most $\deg X$.  If $Y\subsetneq\overline{\pi(X)}$, then there
exists a polynomial  
of degree at most $\deg X$ that vanishes  on $Y$ (and, in particular, at $%
p_1, \ldots, p_N$), but does not vanish on $\overline{\pi(X)}$.  This is
impossible.  . 
\end{proof}

We now return to the proof of Lemma~\ref{lem:surjective}. Assume that the
generic fiber of $\pi|_X$ has dimension $d$. Applying the claim $d$ times,
we obtain an affine subspace $L$ such that  $X \cap L$ is irreducible, $X
\cap L$ projects \glssymbolY{dominant}{dominantly} onto $\overline{\pi(X)}$, and the generic fiber
of $\pi|_{X \cap L}$ is finite by~\cite[Theorem~1.25(ii)]{Shafarevich}. 
Applying statement~\ref{item:1lem7} of Lemma~\ref{lem:finite_fiber} to $X
\cap L$, we obtain a subvariety $Y \subset \overline{\pi(X)}$ of degree  at
most $\deg X$ such that every point in $\overline{\pi(X)} \setminus Y$ has a
preimage in $X \cap L$.  Then it has a preimage in $X$.
\end{proof}


\begin{proof}[Proof of Theorem~\ref{thm:probabilictic}]  Since $\bftheta%
^{\#}$ are locally identifiable $\pi|_{\glssymbolY{Z}{Z}}$ has finite  generic fiber due to
Proposition~\ref{prop:alg_criterion_local}.  We will construct a polynomial $%
P \in \CC[\glssymbolY{Vsharp}{V_\#}]$ as follows 
\begin{itemize}
\item If $\bftheta^{\#}$ is globally identifiable, then we set $P$
to be the product  of the polynomials $P_{1}$ and $P_{\func{lift}}$ defined
below. 
        \begin{enumerate}[(a)]
\item Applying statement~\ref{item:3lem7} of Lemma~\ref{lem:finite_fiber} 
with $X = \glssymbolY{Z}{Z}$ and $\pi = \glssymbolY{pi}{\pi}$, we obtain a hypersurface $H \subset \glssymbolY{Vio}{V_{io}}$ 
of degree at most $|\bftheta^\#| \deg Z$.  We set $P_1$ to be the
defining polynomial of $H$.

\item Applying Lemma~\ref{lem:surjective} to $X = \glssymbolY{Zh}{Z_{\bs{h}'}}$ and $\pi = \glssymbolY{pisharp}{\pi_\#}$, we obtain a proper subvariety $Y_2
\subset Z$  of degree at most $\deg Z_{\bs{h}'}$. 
Since the generic fiber of $\glssymbolY{pi}{\pi}|_Z$ is finite, \cite[Theorem~1.25(ii)]%
{Shafarevich} implies  that $\dim Z = \dim \pi(Z)$, so $\overline{\pi(Y_2)}$
is a proper subvariety of $\overline{\pi(Z)}$.  \cite[Proposition~3]{Heintz}
implies that $\overline{\pi(Y_2)}$ can be defined  by polynomials of degree
at most $\deg Z_{\bs{h}'}$.  We set $P_{\func{lift}}$
to be one of these polynomials that  does not vanish everywhere on $Z$. 
\end{enumerate}

\item If $\bftheta^{\#}$ is not globally identifiable, then we
set $P$ to be the product  of the polynomial $P_{\func{lift}}$ defined above
and the polynomials $P_{\infty}$ and $P_{\func{mult}}$ defined below. 
        \begin{enumerate}[(a)]
\item Applying statement~\ref{item:1lem7} of Lemma~\ref{lem:finite_fiber}
with $X = \glssymbolY{Z}{Z}$ and $\pi = \glssymbolY{pi}{\pi}$,  we obtain a proper subvariety $Y_1 \subset 
\overline{\pi(Z)}$ of degree at most $\deg Z$.  \cite[Proposition~3]{Heintz}
implies that $Y_1$ can be defined by polynomials of degree at most $\deg Z$.
We set $P_{\infty}$ to be one of these polynomials that does not vanish
everywhere on $\overline{\pi(Z)}$.

\item Applying statement~\ref{item:2lem7} of Lemma~\ref{lem:finite_fiber}
with $X = Z$ and $\pi = \pi$,  we obtain a hypersurface $H \subset \glssymbolY{Vsharp}{V_\#}$ of
degree at most $|\bftheta^\#| (\deg Z - 1)$.  We set $P_{\func{%
mult}}$ to be the irreducible defining polynomial of~$H$. 
\end{enumerate}
\end{itemize}

Summing up the degree bounds, we obtain  
\begin{equation*}
\deg P \leqslant \max(\deg (P_1 \cdot P_{\func{lift}}), \deg (P_\infty \cdot
P_{\func{mult}} \cdot P_{\func{lift}})) \leqslant \big(2 + |\bs{%
\theta}^\#|\big)\cdot \deg \glssymbolY{Zh}{Z_{\bs{h}'}}.  
\end{equation*}
In order to prove that $P$ satisfies requirement~\ref{it:lastreq}, we 
consider $a \in Z$ such that $P(a) \neq 0$.

To prove \underline{\ref{thm2:1}$\implies$\ref{thm2:2}}, assume that 
the parameters $\bftheta^{\#}$ are globally identifiable.  Since $%
P_1(a) \neq 0$, the choice of $P_1$  (see statement~\ref{item:3lem7} of
Lemma~\ref{lem:finite_fiber}) implies that  
\begin{equation}  \label{eq:IZIa}
I(Z) + I(\pi(a))\cdot\CC[V_\#] = I(a).
\end{equation}
Since $I(Z) = I(\J_{\bs{h}'}) \cap \CC[V_{\#}]
$,  
\begin{equation}  \label{eq:glob_ident_inclusion}
I(Z) + I(\pi(a))\cdot\CC[V_\#] \subset \left( \J_{\bs{h}'} + I(\pi(a))\cdot \CC[\glssymbolY{V}{V}] \right) \cap \CC%
[V_{\#}].
\end{equation}
Since $P_{\func{lift}}(a)\neq 0$, 
\begin{equation*}
\pi_{\#}^{-1}(a) \cap Z_{\bs{h}'}\ne \varnothing.
\end{equation*}
Hence, the ideal $\J_{\bs{h}'} + I(\pi(a))\cdot 
\CC[V]$ is proper,  and so the right-hand side of~%
\eqref{eq:glob_ident_inclusion} is a proper ideal of $\CC[V_{\#}]$. 
Since, by~\eqref{eq:IZIa}, it contains the maximal ideal $I(a)$, it
coincides with $I(a)$.

The implication \underline{\ref{thm2:2}$\implies$\ref{thm2:3}} follows
from $Z(I(a))=\{ a\}$.

To prove \underline{\ref{thm2:3}$\implies$\ref{thm2:1}}, we assume that 
the parameters $\bftheta^{\#}$ are not globally identifiable. 
Denote the cardinality of the generic fiber of $\pi|_{Z}$ by $d > 1$.  We
define  
\begin{equation*}
C(a) := \glssymbolY{pisharp}{\pi_\#}\left( \glssymbolY{piio}{\pi_{io}^{-1}}(\pi(a)) \cap Z_{\bs{h}'} \right).  
\end{equation*}
A direct computation shows that  
        \begin{equation*} 
           J
        := \left( \J_{\bs{h}'} + I(\pi(a))\cdot \CC%
        [V] \right) \cap \CC[V_\#]  
        \end{equation*}
vanishes at all the points of $C(a)$.  Thus, if we prove that $|C(a)| > 1$,
this would imply that the zero set of $J$ in not $\{ a \}$.

Since $P_{\func{lift}}(a) \neq 0$ and $P \in C[V_{io}]$, for all $b\in
\pi^{-1}(\pi(a))$, $P_{\func{lift}}(b) \neq 0$. Hence, the choice of $P_{%
\func{lift}}$ implies (see Lemma~\ref{lem:surjective}) that, for all $p \in
\pi^{-1}(\pi(a)) \cap Z$,  
\begin{equation}  \label{eq:pisharpZhe}
\pi_\#^{-1}(p)\cap Z_{\bs{h}'} \ne\varnothing.
\end{equation}
Since $\pi_\#^{-1}(p)\subset\pi_{io}^{-1}(\pi(p))$ and $\pi(a)=\pi(p)$, we
have $\pi_\#^{-1}(p)\subset \pi_{io}^{-1}(\pi(a))$.  Hence,  
\begin{equation*}
\pi_\#^{-1}(p)\cap Z_{\bs{h}'} \subset
\pi_{io}^{-1}(\pi(a))\cap Z_{\bs{h}'}.
\end{equation*}
Therefore, using~\eqref{eq:pisharpZhe},  
\begin{equation*}
\pi^{-1}(\glssymbolY{pi}{\pi}(a)) \cap Z \subset\pi_\#\left( \pi_{io}^{-1}(\pi(a)) \cap \glssymbolY{Zh}{Z_{%
\bs{h}'}} \right).
\end{equation*}
Thus,  
\begin{equation*}
|C(a)|\geqslant |\pi^{-1}(\pi(a)) \cap Z|.
\end{equation*}
Let $B_1$ and $B_2$ be closed balls in $V_{io}$ centered at $\pi(a)$ such
that 
\begin{itemize}
\item $\pi^{-1}(B_1)\cap Z$ is compact and 

\item $\pi$ defines a bijection between the connected component of $%
\pi^{-1}(B_2)\cap Z$ containing $a$ and its image. 
\end{itemize}
The existence of such $B_1$ and $B_2$ is implied by statements~%
\ref{item:1lem7} and~\ref{item:2lem7} of Lemma~\ref{lem:finite_fiber}
because $P_{\func{mult}}(a) \ne 0$ and $P_\infty(a)\ne0$ . We set $B =
B_1\cap B_2$.  Let $\mathcal{C}(B)$ be the set of connected components of $%
\pi^{-1}(B) \cap Z$. Since $\pi^{-1}(B) \cap Z$ is compact, the set $%
\mathcal{C}(B)$ is finite. Let $D$ be the union of all $C \in \mathcal{C}(B)$
such that 
\begin{equation*}
C\cap \pi^{-1}(\pi(a)) \cap Z=\varnothing.
\end{equation*}
Suppose that $D\ne \varnothing$. Since $D$ is compact, $\pi(D)$ is compact,
therefore, closed. Moreover, $\pi(a)\notin\pi(D)$. Let $B^{\prime }$ be a
closed ball centered at $\pi(a)$ such that $\pi(D)\cap B^{\prime
}=\varnothing$. We have $B^{\prime }\subset B$ and, for every $C \in 
\mathcal{C}(B^{\prime })$, 
\begin{equation*}
C\cap \pi^{-1}(\pi(a)) \cap Z\ne\varnothing.
\end{equation*}
If $D = \varnothing$, we set $B^{\prime }=B$.  Assume that \[|\pi^{-1}(\pi(a)) \cap \glssymbolY{Z}{Z}| = 1.\] Hence, $\pi^{-1}(B^{\prime }) \cap Z$ has
exactly one connected component.  
Therefore, for every $p^{\prime }\in \pi(Z) \cap B^{\prime }$, we have  
\begin{equation*}
|\pi^{-1}(p^{\prime })\cap Z| = 1.
\end{equation*}
Since $\pi(Z)\cap B^{\prime }$ is Zariski dense in $\pi(Z)$, we arrive at a
contradiction with the assumption that the generic fiber is of cardinality $%
d > 1$.  Hence, $|C(a)| > 1$.\qedhere
\end{proof}

\begin{notation}
Recall that
    \glssymbolY{Q}{$Q$}
is the common denominator of
    \glssymbolY{statefunc}{$\bff$}
and 
    \glssymbolY{outputfunc}{$\bfg$}.
Let  $d_{0}=\max(\deg \glssymbolY{Q}{\ensuremath{Q}}\glssymbolY{outputfunc}{\bfg},\deg \glssymbolY{Q}{\ensuremath{Q}}\glssymbolY{statefunc}{\bff})$. 
\end{notation}

The following statement (with Theorem~\ref{thm:probabilictic} and
Proposition~\ref{prop:dominant}) is used in our design of Algorithm~\ref%
{alg:identifiability}.

\begin{proposition}
\label{prop:subst}  For all $\bs{h}^{\prime }\in \mathbb{Z}%
_{\geqslant 0}^m$ and $P \in \R_{\bs{h}^{\prime }}$ such that  $P$
does not vanish everywhere on $Z_{\bs{h}^{\prime }}$,  there exists
a polynomial $\widetilde{P} \in \CC[\glssymbolY{parameters}{\bftheta}, u_{\In(\bs{h}^{\prime})}]$ such that
\begin{itemize}
\item  $\deg \widetilde{P} \leqslant (1 + d_0\cdot(2\max 
\bs{h}^{\prime }- 1))\cdot \deg P$ and
\item for all $\tilde{%
\bftheta}$ and $\tilde{u}_{\In(\bs{h}^{\prime})}$,  
\begin{equation*}
\widetilde{P}(\tilde{\bftheta}, \tilde{u}_{\In(\bs{h}^{\prime})}) \neq 0 \implies P \text{ does not vanish at } \pi_{ip}^{-1}(\tilde{%
\bftheta}, \tilde{u}_{\In(\bs{h}^{\prime})}) \cap Z_{%
\bs{h}^{\prime }},  
\end{equation*}
where $\pi_{ip}$ is the projection from the space with coordinates $( 
\glssymbolY{systemparams}{\bfmu},\bs{x}_{\glssymbolY{State}{\State}(\bs{h^{\prime})}}, \bs{y}_{%
\bs{h^{\prime }}}, u_{\glssymbolY{In}{\In}(\bs{h^{\prime})}})$ to the space with
coordinates $(\glssymbolY{parameters}{\bftheta}, u_{\In(\bs{h^{\prime})}})=(%
\glssymbolY{systemparams}{\bfmu},x_1,\ldots,x_n, u_{\In(\bs{h^{\prime})}})$.
\end{itemize}
\end{proposition}

\begin{proof}
Consider the ranking on $\R_{\bs{h}^{\prime }}$ defined in the proof
of Lemma~\ref{lem:gens_and_prime}.  We will define a linear operator $%
r\colon \R_{\bs{h}^{\prime }} \to \R_{\bs{h}^{\prime }}$. 
Consider a monomial $m \in \R_{\bs{h}^{\prime }}$.  Let $v$ be the
leading variable of $m$, so $m = v \cdot \tilde{m}$.  Then  
\begin{equation*}
r(m) :=  
\begin{cases}
\left(Qx_i^{(j + 1)} - (Qx_i^{\prime }- F_i)^{(j)} \right)\tilde{m}, & v =
x_i^{(j + 1)}, \\ 
\left(Qy_i^{(j)} - (Qy_i - G_i)^{(j)} \right)\tilde{m}, & v = y_i^{(j)}, \\ 
Q\cdot m, & \text{otherwise}.%
\end{cases}
\end{equation*}
By the definition of $r(P)$ and $\J_{\bs{h}^{\prime }}$,  
\begin{equation}  \label{eq:r_prop}
QP - r(P) \in \J_{\bs{h}^{\prime }} \text{ for every } P \in \R_{%
\bs{h}^{\prime }}.
\end{equation}
We introduce the following weight function $w$ by  
\begin{equation*}
\begin{cases}
w\big(u^{(i)}\big) = 0, & i \geqslant 0, \\ 
w\big(x_j^{(i)}\big) = \max(0, 2i - 1), & 1 \leqslant j \leqslant n,\ i
\geqslant 0, \\ 
w\big(y_j^{(i)}\big) = 2i + 1, & 1 \leqslant j \leqslant m,\ i \geqslant 0,
\\ 
w(\mu_i) = 0, & \mu_i \text{ in }\bfmu,%
\end{cases}
\end{equation*}
(extending $w$ multiplicatively to monomials and as the $\max$ to sums of
monomials).  A direct computation shows that
\[r(P) \neq QP \implies w(P) > w(r(P)). 
\]   
Thus, there exists a finite sequence $P_0, \ldots, P_q$ such that  
\begin{equation*}
P_0 = P,\ \ \ P_{i + 1} = r(P_i) \neq QP_i\ \text{ for all }\ 0 \leqslant i
< q,\ \ \text{and}\ \ \ r(P_q) = QP_q.
\end{equation*}
We set $\widetilde{P} := P_q$.  Since $w(P) \leqslant \deg P\cdot
(2\cdot\max \bs{h}^{\prime }- 1)$, we have 
\begin{equation*}
q\leqslant\deg P\cdot (2\cdot\max \bs{h}^{\prime }- 1).
\end{equation*}
Therefore,  
\begin{equation*}
\deg \widetilde{P} \leqslant \deg P + \deg P \cdot d_0 \cdot (2\cdot\max 
\bs{h}^{\prime }- 1) = (1 + d_0\cdot (2\cdot\max \bs{h}%
^{\prime }- 1))\cdot \deg P.  
\end{equation*}
Since $r(\widetilde{P}) = Q\widetilde{P}$, we have $\widetilde{P} \in 
\CC[\bftheta, u_{\In(\bs{h}^{\prime})}]$.  Due to~%
\eqref{eq:r_prop}, we have 
\begin{equation*}
Q^q \cdot P - \widetilde{P} \in \J_{\bs{h}^{\prime }}.
\end{equation*}
Therefore, for all $\tilde{\bftheta}$, $\tilde{u}_{\In(\bs{h}^{\prime})}$, and $p \in \pi_{ip}^{-1}\big(\tilde{\bftheta}, 
\tilde{u}_{\In(\bs{h}^{\prime})}\big) \cap Z_{\bs{h}^{\prime }}$%
,  we have 
\begin{equation*}
\widetilde{P}\big(\tilde{\bftheta}, \tilde{u}_{\In(\bs{h}^{\prime})}\big)=Q^q\big(\tilde{\bftheta}, \tilde{u}_{\In(\bs{h}^{\prime})}\big)\cdot P(p).
\end{equation*}
Hence, if $\widetilde{P}\big(\tilde{\bftheta}, \tilde{u}_{\In(\bs{h}^{\prime})}\big) \ne 0$, then $P(p)\ne 0$.
\end{proof}

\section{Algorithm}

In this section, by integrating Theorems~\ref{thm:alg_criterion} and~\ref%
{thm:probabilictic} and Propositions~\ref{prop:dominant} and~\ref{prop:subst}%
, we provide a probabilistic algorithm for checking global
identifiability~(Algorithm \ref{alg:identifiability}) and prove its
correctness (Theorem~\ref{thm:correctness}).

\begin{remark}
\label{rem:probability}
Note that the input-output specification of
Algorithm~\ref{alg:identifiability} states that   $\bftheta^g$ is equal, with probability at least $p$, to the set of all globally identifiable parameters in $\bftheta^\ell$. Let us state it more formally. Let $\bftheta^{ga}$ stand for the set of all globally identifiable parameters in $\bftheta^\ell$.
Imagine running the algorithm $n$ times with the same input. 
Let  $c_n$ stand for the number of runs with the correct output, that is, $\bftheta^{g}=\bftheta^{ga}$.
Then the following is guaranteed
\begin{equation*}
\lim_{n\to\infty} \frac{c_n}{n}\geqslant p.
\end{equation*}
\end{remark}

\label{sec:algorithm} 

\begin{notation}
In the steps of Algorithm~\ref{alg:identifiability}, we use Notation~\ref{not:subs_der} and 
\[
  \bs{u}=\left( \glssymbolY{input}{\ensuremath{u}}, \glssymbolY{input}{\ensuremath{u}}^{(1)}, \ldots, \glssymbolY{input}{\ensuremath{u}}^{(s)} \right),\quad \text{ where } s
= |\glssymbolY{parameters}{\bftheta}|.
\]
\end{notation}

\begin{algorithm}[p]
\caption{\sl Global\_Identifiability}\label{alg:identifiability}
\begin{description}[style=nextline,leftmargin=0.3cm,labelwidth=1.6em,align=parleft]
  \item[In] 
    \begin{enumerate}
    \item[$\Sigma$:]
      an algebraic differential model 
      
      given by rational functions $\glssymbolY{statefunc}{\bff}(\glssymbolY{states}{\ensuremath{\bfx}},\glssymbolY{systemparams}{\bfmu},\glssymbolY{input}{\ensuremath{u}})$ and $\glssymbolY{outputfunc}{\bfg}(\glssymbolY{states}{\ensuremath{\bfx}},\glssymbolY{systemparams}{\bfmu},\glssymbolY{input}{\ensuremath{u}})$    
    \item[$\bftheta^\ell$:] 
 a subset of 
$\glssymbolY{parameters}{\bftheta}
=\glssymbolY{systemparams}{\bfmu}\cup \glssymbolY{initialstates}{\ensuremath{\bfxs}}$ 
 with
every parameter in $\bs{\protect\theta}^\ell$  
locally identifiable
\item[$p$\hfill:] 
      an element of $(0, 1)$
    \end{enumerate}
 \item[Out]
     \begin{enumerate}
     \item[$\bftheta^g$:] a subset of $\bftheta^\ell$ that is equal, with probability at least $p,$ to
     the set of 
     
     all globally identifiable parameters in $\bftheta^\ell$ (see Remark~\ref{rem:probability})
    \end{enumerate}
\end{description}

\begin{enumerate}[leftmargin=-0.1cm,labelwidth=1.5em,label=\arabic*]
\item {[Construct the maximal system $E$ of algebraic equations]}

\begin{enumerate}[leftmargin=1cm]
  \item $s\leftarrow |\glssymbolY{parameters}{\bftheta}|$,
      $Q$ $\leftarrow$ the common denominator of
    \glssymbolY{statefunc}{$\bff$}
        and
    \glssymbolY{outputfunc}{$\bfg$}

  \item\label{step:1b} $X_{i0}\leftarrow x_{i}^{\left(  0\right)  }-x_{i}^{\ast}\ $for $1\leqslant i\leqslant n$

  \item\label{step:1c} $X_{ij} \leftarrow \left(Qx_i' - Qf_i \right)^{(j-1)}$ 
  for $1\leqslant i\leqslant n$ and $1\leqslant j\,\leqslant s$

  \item\label{step:1d} $Y_{ij} \leftarrow \left(Qy_i - Qg_i \right)^{(j)}$ for
  $1\leqslant i\leqslant m$ and $0\leqslant j\,\leqslant s$

  \item $E\leftarrow$ the set of all $X_{ij}$ and $Y_{ij}$ computed in Steps~\ref{step:1b},~\ref{step:1c}, and~\ref{step:1d}.
\end{enumerate}

\item\label{alg1:Step2} {[Truncate the system $E$, obtaining  $E^t$]}
\begin{enumerate}[leftmargin=1cm]
  \item $d_{0}\leftarrow\max(\deg Q\glssymbolY{statefunc}{\bff},\deg Q\glssymbolY{outputfunc}{\bfg}, \deg Q)$

  \item $D_{1}\leftarrow 2d_{0}s(n + 1)(1 + 2d_{0}s)/(1-p)$
  \item\label{alg1:Step2c} $\bs{\hat{\theta}},\hat{\bs{u}%
  }\;\leftarrow$ 
random vectors  of
integers  from $[1,D_{1}]$  and $Q(\bs{\hat{\theta}}, \hat{u}_0) \neq 0$

  \item\label{alg1:Step2d} Find the unique solution of the triangular system 
  $E\big[\glssymbolY{parameters}{\bftheta}\leftarrow\bs{\hat{\theta}},\bs{u}\leftarrow\hat{\bs{u}}\big]$ 
  
  (in which each equation is linear in its leader) for all the variables. 
  
  Denote the $x_i^{(j)}$- and $y_i^{(j)}$-components of the solutions 
  by $\hat{x}_{ij}$ and $\hat{y}_{ij}$.

  \item Let $\bs{\alpha}\leftarrow\left(  1,\ldots,1\right)  \in\mathbb{Z}_{\geqslant 0}^{n}$, 
  $\bs{\beta}\leftarrow\left(  0,\ldots,0\right)  \in\mathbb{Z}_{\geqslant 0}^{m}$,  $E^t\leftarrow\left\{ X_{10}, \ldots, X_{n0} \right\}  $

  \item\label{step2:f} While there exists $k$ such that 
  
  the rank of
  $\func{Jac}_{\bftheta,\bs{x}_{\bs{\alpha}}}\left(  E^t%
  \cup\{Y_{k \beta_k}\}\right)  $ at $\big(\hat{\bftheta},  \hat{\bs{x}}%
  _{\bs{\alpha}},\hat{\bs{y}}_{\bs{\beta}}%
  ,\hat{\bs{u}}\big)  $ is equal to $|E^t| + 1$

\begin{enumerate}
\item Add $Y_{k,\beta_{k}}$ to $E^t$ and then increment $\beta_k$

\item\label{step:2f2} While $x_{i}^{(j)}$ appears in $E^t \cup \big\{Y_{1 \beta_1},  \ldots, Y_{m \beta_m}\big\}$ but $X_{ij} \notin E^{t}$, 

\quad add $X_{ij}$ to $E^t$

\item\label{step:2f3} Set $\alpha_i \leftarrow 
\max\limits_{P \in E^t} \glssymbolY{difforder}{\ord}_{x_i} P+1$ for every $1 \leqslant i \leqslant n$
\end{enumerate}

\item\label{step:2g} While for some $1 \leqslant i \leqslant m$ all $\bs{x}$-variables
in $Y_{i(\beta_i + 1)}$ belong to $\bs{x}_{\bs{\alpha}}$, 

\quad add $Y_{i(\beta_i + 1)}$ to $E^t$ and increment $\beta_i$

\end{enumerate}

\item {[Randomize some variables in $E^t$, obtaining $\widehat{E^t}$]}
\begin{enumerate}[leftmargin=1cm]
\item $D_{2} \leftarrow 6|\bftheta^\ell|\left(  \prod\limits_{P\in
E^t}\deg P\right)  (1 + 2d_{0}\max\bs{\beta})/(1-p)$

\item \label{alg1:Step3b}$\bs{\hat{\theta}},\hat{\bs{u}%
}\;\leftarrow$ random vectors  of  
integers
from
$[1,D_{2}]$ and $Q(\bs{\hat{\theta}}, \hat{u}_0) \neq 0$

\item Find the unique solution of the triangular system 
  $E^t\big[\glssymbolY{parameters}{\bftheta}\leftarrow\bs{\hat{{\theta}}},
        \bs{u}\leftarrow\hat{\bs{u}}\big]$  
        
        (in which each equation is linear in its leader) for all the variables.  
  
Denote the $x_i^{(j)}$- and $y_i^{(j)}$-components of the solutions 
  by $\hat{x}_{ij}$ and $\hat{y}_{ij}$.
  
\item $\widehat{E^t} \leftarrow E^t\left[  \bs{y}_{\bs{\beta}}\leftarrow \hat{\bs{y}}_{\bs{\beta}},\ \bs{u}\leftarrow
\hat{\bs{u}}\right],\;\;\;\;\;\;  \widehat{Q} \leftarrow  Q[\bs{u} \leftarrow \hat{\bs{u}}]$

\end{enumerate}
\item\label{alg1:Step4} {[Determine $\bftheta^g$ from  $\widehat{E^t}$]} 
\begin{enumerate}[leftmargin=1cm]
  \item $\bftheta^g \leftarrow 
    \big\{ \theta \in \bftheta^\ell \:\big|\:  \text{the system } 
    \widehat{E^t}=0\: \&\: \widehat{Q} \neq 0\: \&\: \theta\neq \hat{\theta} \text{ is inconsistent}\big\}
  $
\end{enumerate}
\end{enumerate}
\end{algorithm}


\begin{remark}
\label{rem:proj_comp} In Step 4 of Algorithm~\ref{alg:identifiability}, we
check the consistency of a system of equations and inequations. This can be
done in many different ways. We list a few.
    \begin{enumerate}[leftmargin=2.2cm,labelsep=0.25cm,label=Method~(\roman*)]
\item \label{method:1} For each $\theta \in \bftheta^\ell$, check
the inconsistency of the following system of equations  
\begin{equation*}
\widehat{E^t} \cup \big\{z\cdot\hat{Q}-1,w(\theta - \hat{\theta}) - 1\big\},
\end{equation*}%
where $z$ and $w$ are new variables introduced for the Rabinowitsch trick.
This can be done by using, for instance, Gr\"obner bases.

\item \label{method:2} For each $\theta \in \bftheta^\ell$, check
\begin{equation}  \label{eq:viaGB}
\theta - \hat{\theta} \in \func{Ideal}\big(\widehat{E^t} \cup \{z\cdot\widehat{Q}%
-1\}\big).
\end{equation}
This can be done by using, for instance, Gr\"obner bases.  Due to Theorem~%
\ref{thm:probabilictic}, condition~\eqref{eq:viaGB} is equivalent to  the
consistency condition from Step~\ref{alg1:Step4} of Algorithm~\ref%
{alg:identifiability}  under assumption~\eqref{eq:cond} on the sampled point
made in the proof of Theorem~\ref{thm:correctness} for the index $i$ of this
particular $\theta$.

\item Check it directly using regular chains, which also allows inequations
(see, e.g., \cite[Section~2.2]{Wang2001}). In practice, we observed that
this method is generally slower than~\ref{method:1} and~\ref{method:2} for
the problems that we considered.

\item \label{method:4}  Use homotopy continuation methods (for example, in
Bertini~\cite{Bertini}) as follows 
\begin{enumerate}
\item Take a square subsystem $\widehat{E^{ts}} = 0$ from the over-determined
system $\widehat{E^t} = 0$ and find all of its roots using homotopy
continuation; 

\item Let $S$ be the set of all the roots of $\widehat{E^{ts}} = 0$ that  are
also roots of $\widehat{E^t} = 0$ and are not roots of $\widehat{Q} = 0$; 

\item For each parameter $\theta \in \bftheta^\ell$, we put $%
\theta$ into $\bftheta^g$ if and only  if the $\theta$-coordinate
of every point in $S$ is equal to $\hat{\theta}$. 
\end{enumerate}

However the number of paths in problems of moderate size  (like Examples~\ref%
{ex:CRN} and~\ref{ex:Cholera}) is too large for practical computation.
\end{enumerate}
\end{remark}


\begin{example}
\label{ex:alg} We will illustrate the steps of the algorithm on a system of small size.

\begin{description}[style=nextline] 

\item[In] 

\begin{enumerate}

\item[$\Sigma$] $:=%
\begin{cases}
x^{\prime }= \mu_2 x + \mu_1, \\ 
y = x^2, \\ 
x(0) = x^\ast.%
\end{cases}%
$ 

\item[$\bs{\protect\theta}^\ell$] $:=\{ \mu_1, \mu_2, x^\ast \}$ 
\textit{(one can show that these parameters are locally identifiable)} 

\item[$p$\hfill] :$= 0.8$ 
\end{enumerate}
\end{description}

\begin{enumerate}[leftmargin=0.5cm,labelwidth=1.5em,label=\arabic*.]
\item {[Construct the maximal system $E$ of algebraic equations]}
\begin{enumerate}

\item $s \leftarrow 3$, $Q \leftarrow 1$

\emph{Since there is only one state variable and only one output variable,
from now on, we will drop the first index from $X$ and $Y$ for brevity, for
instance $X_0$ and $Y_0$ instead $X_{1,0}$ and $Y_{1,0}$, etc.} 

\item $X_{0}\leftarrow x - x^\ast$

\item $X_{1}\leftarrow x^{(1)} - \mu_2 x - \mu_1$, $X_2 \leftarrow x^{(2)} - \mu_2 x^{(1)}$, $X_3 \leftarrow x^{(3)} - \mu_2 x^{(2)}$

\item $Y_0 \leftarrow y - x^2, \ldots, Y_3 \leftarrow y^{(3)} - 2xx^{(3)} - 6x^{(1)} x^{(2)}$

\item $E\leftarrow \{ X_0, \ldots, X_3, Y_0, \ldots, Y_3\}$.
\end{enumerate}

\item {[Truncate the system $E$, obtaining $E^t$]}

\begin{enumerate}

\item $d_0 \leftarrow \deg (\mu_2 x + \mu_1) = 2$.

\item $D_{1} \leftarrow 1560$

\item $(\hat{\mu}_1, \hat{\mu}_2, \hat{x}^\ast) \leftarrow (1210, 896, 453)$

\item \label{ex:step2d} Solve the triangular system \[E\big[\bftheta\leftarrow\bs{\hat{\theta}},\bs{u}\leftarrow\hat{\bs{u}}\big]=E\big[\mu_1\leftarrow\hat\mu_1,\mu_2\leftarrow\hat\mu_2,x^*\leftarrow \hat x^*\big]\]
\begin{align*}
&y^{(3)} - 2x x^{(3)} - 6 x^{(1)}x^{(2)} = 0, & & x^{(3)} - 896 x^{(2)} =
0, \\
&y^{(2)} - 2x x^{(2)} - 2 (x^{(1)})^2 = 0, & & x^{(2)} - 896 x^{(1)} = 0, \\
&y^{(1)} - 2 x x^{(1)} = 0, & & x^{(1)} - 896 x - 1210 = 0, \\
&y - x^2 = 0, & &x - 453 = 0.
\end{align*}
iteratively: first the right column from the bottom to the top, then the
left column.  We obtain  
\begin{align*}
&\hat{x}_0 \leftarrow 453  &&  \hat{y}_0  \leftarrow 205209\\ 
&\hat{x}_1\leftarrow 407098 && \hat{y}_1 \leftarrow 368830788\\
&\hat{x}_2 \leftarrow 364759808&& \hat{y}_2 \leftarrow 661929949256 \\
&\hat{x}_3 \leftarrow 326824787968 && \hat{y}_3 \leftarrow 1187061187802112
\end{align*}

\item $\bs{\alpha} \leftarrow (1)$, $\bs{\beta} \leftarrow (0)$, $E^t \leftarrow \{
X_0 \}$.

\item Iteration 1: Since  $\bs{\alpha} = (0)$ and
\begin{equation*}
\func{Jac}_{(\mu_1, \mu_2, x^\ast, x)}(X_0, Y_0)=  
\begin{pmatrix}
0 & 0 & -1 & 1 \\ 
0 & 0 & 0 & -2x%
\end{pmatrix}
\end{equation*}
at $(\hat{\mu}_1, \hat{\mu}_2, \hat{x}^\ast, \hat{x}_0)$ has rank equal to the number or rows,  we
add $Y_0$ to $E^t$ and set $\bs{\beta} \leftarrow (1)$.  Since $Y_1$
involves $x^{(1)}$, we add $X_1$ to $E^t$ and set $\bs{\alpha} \leftarrow (2)
$.

Iteration 2: Since  
\begin{equation*}
\func{Jac}_{(\mu_1, \mu_2, x^\ast, x, x^{(1)})}(X_0, X_1, Y_0, Y_1) =  
\begin{pmatrix}
0 & 0 & -1 & 1 & 0 \\ 
-1 & -x & 0 & -\mu_2 & 1 \\ 
0 & 0 & 0 & -2x & 0 \\ 
0 & 0 & 0 & -2x^{(1)} & -2 x%
\end{pmatrix}
\end{equation*}
at $(\hat{\mu}_1, \hat{\mu}_2, \hat{x}^\ast, \hat{x}_0, \hat{x}_1)$ has 
rank equal to the number of rows,  we add $Y_1$ to $E^t$ and set $\bs{\beta} \leftarrow (2)$.  Since $Y_2
$ involves $x^{(2)}$, we add $X_2$ to $E^t$ and set $\bs{\alpha} \leftarrow
(3)$.

Carrying out similar computations, we obtain a Jacobian with rank equal to the number of rows in the
next iteration,  but obtain a Jacobian with rank less than the number or rows in the following
iteration, so we stop the while loop.  We obtain $\bs{\alpha} = (4)
$, $\bs{\beta} = (3)$,  and $E^t = \{ X_0, X_1, X_2, X_3, Y_0, Y_1,
Y_2 \}$.

\item We add $Y_3$ to $E^t$, set $\bs{\beta} \leftarrow (4)$. 
\end{enumerate}

\item {[Randomize some variables in $E^t$, obtaining $\widehat{E^t}$]}

\begin{enumerate}

\item $D_{2} \leftarrow 195840$

\item $(\hat{\mu}_1, \hat{\mu}_2, \hat{x}^\ast) \leftarrow (2440, 171852, 68794)$

\item We perform the same computation as in Step~\ref{ex:step2d}, but with
the new numbers, and thus find numerical values $\hat y_0, \hat y_1, \hat y_2, \hat y_3$.

\item $\widehat{E^t}\leftarrow E^t\left[ y \leftarrow \hat{y}_0, y^{(1)} \leftarrow 
\hat{y}_1, y^{(2)} \leftarrow \hat{y}_2, y^{(3)} \leftarrow \hat{y}_3 \right]%
,\;\;\;\;\;\;\hat{Q} \leftarrow 
1
$
\end{enumerate}

\item {[Determine $\bftheta^g$ from $\widehat{E^t}$]}

\item[] This can be done in several ways as described in Remark~\ref%
{rem:proj_comp}.

\begin{itemize}
\item Following \ref{method:1} from Remark~\ref{rem:proj_comp}, we compute 
a Gr\"obner basis for each of the following  
\begin{gather*}
\widehat{E}^t \cup \{ z - 1, (\mu_1 - 2440)w - 1 \},\;\; \widehat{E}^t
\cup \{ z - 1, (\mu_2 - 171852)w - 1 \}, \\
\widehat{E}^t \cup \{ z - 1, (x^\ast - 68794)w - 1 \}.
\end{gather*}
with respect to a monomial ordering (any choice). 
A computation shows that only the second Gr\"obner basis contains 1. 
Thus $\mu_2$  is globally identifiable, and $\mu_1$ and $x^\ast$ are not.

\item Following \ref{method:2} from Remark~\ref{rem:proj_comp}, we compute
the reduced  Gr\"obner basis of $\widehat{E}^t \cup \{ z - 1 \}$  with
respect to (we make this particular choice here just to obtain a concrete answer) the degree reverse lexicographic ordering with the ordering  $%
\mu_1 > \mu_2 > x^\ast > x > \ldots > x^{(3)} > z$ on the variables:  
\begin{align*}
&z - 1, & &171852 x^{(2)} - x^{(3)},\\
&29533109904 x^{(1)} - x^{(3)}, & & 174575955769228371456 x - 34397 x^{(3)},\\
&\mu_2 - 171852, & & 43643988942307092864 \mu_1 - 305 x^{(3)},\\ 
&(x^{(3)})^2 - 349151911538456742912^2, & & x^{\ast} - x,
\end{align*}

Observe that the Gr\"obner basis contains $\mu_2 - 171852$ but does not contain $x^\ast - 68794$ or $%
\mu_1 - 2440$. 
Thus, $\mu_2$  is globally
identifiable and $\mu_1$ and $x^\ast$ are not.
\end{itemize}
\end{enumerate}

\begin{description}[style=nextline]

\item[Out] 

\begin{enumerate}

\item[$\bs{\protect\theta}^g$]$\leftarrow\{ \mu_2 \}$ 
\end{enumerate}
\end{description}
\end{example}


\begin{theorem}
\label{thm:correctness}  Algorithm~\ref{alg:identifiability} produces 
correct output with probability at least $p$.
\end{theorem}

\begin{proof}
In our probability analysis of random vectors $\hat{\bftheta}$
and $\hat{\bs{u}}$  sampled in steps~\ref{alg1:Step2c} and~\ref%
{alg1:Step3b}, we will use the following observation.  For all polynomials $%
P, Q \in \CC[\bftheta, \bs{u}]$ such that $Q \neq 0
$ and for every sample set of vectors $(\hat{\bftheta}, \hat{%
\bs{u}})$,  the probability of $P(\hat{\bftheta}, \hat{%
\bs{u}}) \neq 0$ given that $Q(\hat{\bftheta}, \hat{%
\bs{u}}) \neq 0$ is greater than or equal to the probability of $P(%
\hat{\bftheta}, \hat{\bs{u}}) Q(\hat{\bftheta}%
, \hat{\bs{u}}) \neq 0$.

We \textbf{claim} that the value of $\bs{\beta}$ right before  Step~%
\ref{step:2g} will satisfy the requirements for $\bs{h}$ of Theorem~%
\ref{thm:alg_criterion}  with probability at least $\frac{1 + p}{2}$.  We
run Algorithm~\ref{alg:identifiability} replacing checking of the rank of
the Jacobian  in Step~\ref{step2:f} with checking the fact that $Z_{%
\bs{\beta} + \bs{1}_k}$ projects \glssymbolY{dominant}{dominantly}  to the $(%
\bs{y}_{\bs{\beta} + \bs{1}_k}, u_{\In(\bs{\beta}
+ \bs{1}_k)})$-coordinates.  Denote the number of iterations of the
while loop in Step~\ref{step2:f} by $q$.  For all $i$, $1 \leqslant i
\leqslant q$, we denote the values of $\bs{\beta}$ and $E^t$ after
the $i$-th such iteration  by $\bs{\beta}_i$ and $E^t_i$,
respectively, where $\bs{\beta}_i := (\beta_{i1}, \ldots, \beta_{im})
$.  Let $\bs{\beta}_0$ and $E^t_0$ denote the initial values.  Due
to the construction, $\bs{\beta}_q$ satisfies the requirements for  $%
\bs{h}$ of Theorem~\ref{thm:alg_criterion}.

Comparing Step~\ref{step:2f2} of Algorithm~\ref{alg:identifiability} and 
Step~\ref{not3:Step2} of Notation~\ref{not:Sh}, we see that, for every $%
\bs{\beta}$,  
\begin{equation}  \label{eq:connection_S_and_E1}
E^t  \sqcup \big\{ Y_{1, \beta_1}, \ldots, Y_{m, \beta_m}\big\} = \cS_{%
\bs{\beta} + \bs{1}}  \sqcup \big\{X_{10}, \ldots, X_{n0}%
\big\}
\end{equation}
This and Notation~\ref{not:Sh} imply that  
\begin{equation}  \label{eq:connection_S_and_E2}
\begin{aligned}
E^t = &\cS_{\bs{\beta}} \sqcup \big\{X_{10}, \ldots, X_{n0}\big\}\\ %
 &\sqcup \big\{X_{ij}\in \cS_{\bs{\beta}+\bs{1}} \:\big|\:
x_i^{(j)} \text{ does not appear in } \cS_{\bs{\beta}},\ 0 < j,\
1\leqslant i\leqslant m\big\}.
\end{aligned}
\end{equation}
Proposition~\ref{prop:dominant} together with~\eqref{eq:connection_S_and_E2}
imply that, for all $k$, $1\leqslant k \leqslant m$,  if $Z_{\bs{%
\beta} + \bs{1}_k}$ does not project dominantly  to the $(%
\bs{y}_{\bs{\beta} + \bs{1}_k}, u_{\In(\bs{\beta}
+ \bs{1}_k)})$-coordinates,  then the Jacobian condition of the while
loop in Step~\ref{step2:f} is also false.

Let $P_1$ denote a $(|E_q^t| + 1) \times (|E_q^t| + 1)$-minor of the Jacobian of $\cS_{\bs{\beta}%
_q}$  with respect to $(\bfmu,\bs{x}_{\bs{\alpha}_q})$ that is nonzero modulo $\J_{\bs{\beta}_q}$.  Decomposition~%
\eqref{eq:connection_S_and_E2} implies that there  exists a $(|E_q^t| + 1) \times (|E_q^t| + 1)$-minor in
the Jacobian of $E^t_q$ with respect to  $(\bftheta, \bs{x%
}_{\bs{\alpha}_q})$ with the determinant equal to $Q^a P_1$ for some 
$a$.  If $Q \cdot P_1$ does not vanish after the substitution  
\begin{equation}  \label{eq:subs}
\bs{x}_{\bs{\alpha}_q} \leftarrow \hat{\bs{x}}_{\bs{\alpha}_q}, \ \bs{y}_{\bs{\beta}_q} \leftarrow 
\hat{\bs{y}}_{\bs{\beta}_q},\ \bs{u} \leftarrow \hat{%
\bs{u}},\ \bftheta \leftarrow \hat{\bftheta},
\end{equation}
then $\func{Jac}_{\bftheta,\bs{x}_{\bs{\alpha}%
_q}} (E^t_q)$ has rank $|E^t_q| + 1$  at  $(\hat{\bftheta}, \hat{%
        \bs{x}}_{\bs{\alpha}_q},\hat{\bs{y}}_{\bs{%
\beta}_q})$.  Since $E^t_i \subset E^t_q$ for every $1 \leqslant i \leqslant
q$,  the corresponding Jacobian of $E^t_i$ has rank $|E^t_i| + 1$ 
at the corresponding point for $1 \leqslant i \leqslant q$. Thus, with this choice of $(\hat{%
\bftheta}, \hat{\bs{u}})$, the value of $\bs{\beta%
}$  right before Step~\ref{step:2g} will be $\bs{\beta}_q$ and,
therefore,  satisfy the requirements for $\bs{h}$ of Theorem~\ref%
{thm:alg_criterion}.

We will bound the probability of non-vanishing of $Q P_1$ after substitution~\eqref{eq:subs}. Let $\bs{\beta}_q = (\beta_{q1},\ldots,\beta_{qm})$.  Corollary~\ref{cor:existence} implies that $\beta_{q1} + \ldots +
\beta_{qm} \leqslant s$,  so $|\cS_{\bs{\beta}_q}| \leqslant s + ns$.  Hence, 
\begin{equation*}
\deg P_1 \leqslant s(n + 1) d_0.
\end{equation*}
Applying Proposition~\ref{prop:subst} to $P_1$ with this degree bound, we obtain
$\widetilde{P}_1 \in \CC[\bftheta, \bs{u}]$ such that,
\begin{itemize}
\item $\deg \widetilde{P}_1\leqslant d_0 s(n + 1)(1 + d_0 (2s - 1))$ and 

\item for all $\tilde{\bftheta}$ and $\tilde{\bs{u}}$,  $%
\widetilde{P}_1\big(\tilde{\bftheta}, \tilde{\bs{u}}\big)%
\ne 0$ implies that $P_1$ does not vanish at $\pi_{ip}^{-1}(\tilde{%
\bftheta}, \tilde{\bs{u}}) \cap Z_{\bs{\beta}_q}$.
\end{itemize}
Since the coordinates of $(\hat{%
\bftheta}, \hat{\bs{u}})$ are sampled from $1$ to $D_1$, 
the Demillo-Lipton-Schwartz-Zippel lemma (see~\cite[Proposition~98]{Zippel})
implies that $Q(\hat{\bftheta}, \hat{\bs{u}}) \widetilde{P%
}_1(\hat{\bftheta}, \hat{\bs{u}})\ne 0$  with probability
at least  
\begin{multline*}
1 - \frac{d_0 s(n + 1)(1 + d_0 (2s - 1)) + d_0}{D_1}  =1 - \frac{D_1(1-p)-d_0^2 s(n + 1) + d_0}{2D_1} \\
\geqslant 1-\frac{1-p}{2}= \frac{1 + p}{2%
}.  
\end{multline*}
The \textbf{claim} is proved.

The values of $\bs{\alpha}$, $\bs{\beta}$, and $E^t$ right
after  Step~\ref{alg1:Step2} of Algorithm~\ref{alg:identifiability} 
have the following properties:
\begin{enumerate}[(a)]
  \item\label{it:propa} $\bs{\alpha} = \bs{\alpha}_q$;
  \item\label{it:propb} $\bs{\beta} - \bs{\beta}_q \in \mathbb{Z}_{> 0}^m$;
  \item\label{it:propc} $E^t \supset E^t_q \cup \{Y_{1\beta_{q1}}, \ldots, Y_{m\beta_{qm}}\}$.
\end{enumerate}
By~\eqref{eq:connection_S_and_E1} applied to $\bs{\beta}_q$ and by~\ref{it:propc}, after Step~\ref%
{alg1:Step2}, we have 
\[
E^t \supset \cS_{\bs{\beta}_{q}+\bs{1}}  \sqcup \big\{X_{10},\ldots, X_{n0}\big\}.
\]
By~\ref{it:propa},  $\cS_{\bs{\beta}}$ and $\cS_{\bs{\beta}_q + \bs{1}}$ contain the same polynomials of the form $X_{i, j}$.
Hence, by the construction of  $\cS_{\bs{\beta}}$ (see Step~\ref{not3:Step2} of Notation~\ref{not:Sh}) and Step~\ref{step:2g} of Algorithm~\ref{alg:identifiability}, 
\begin{equation}  \label{eq:connectionS_E3}
E^t = \cS_{\bs{\beta}}  \sqcup \big\{X_{10},
\ldots, X_{n0}\big\}.
    \end{equation}
Consider $i$, $1 \leqslant i \leqslant |\bftheta^\ell|$. Let $P_{2,i}$
be a polynomial whose existence is proven in Theorem~\ref{thm:probabilictic}
applied to $\bftheta^\# = (\theta_i^\ell)$, $\bs{h}'=\bs{\beta}$, and $\bs{h}=\bs{\beta}_q$ (see~\ref{it:propb}).  Consider $a :=
\left( \hat{\bfmu}, \hat{\bs{x}}_{\bs{\alpha}}, 
\hat{\bs{y}}_{\bs{\beta}}, \hat{%
\bs{u}}\right)$. Then~\eqref{eq:connectionS_E3} implies that the
projection of the Zariski closure of the zero set of  $\widehat{E^{t}} = 0\
\&\ \widehat{Q} \neq 0$ to the $(\bfmu, \bs{x}_{\bs{\alpha}}, \bs{y}_{\bs{\beta}%
}, \bs{u})$-coordinates is the zero set  of the ideal  
\begin{equation*}
\J_{\bs{\beta}} + I(\pi_{io}(a))\cdot\CC[%
        \bfmu, \bs{x}_{\bs{\alpha}}, \bs{y}_{%
\bs{\beta}}, \bs{u}].  
\end{equation*}
Hence, $\theta_i^{\ell}$ will be added to $\bftheta^g$ if and
only if  
\begin{equation}  \label{eq:step4}
Z\left( \left( \J_{\bs{\beta}} +
I(\pi_{io}(a))\cdot\CC[\bfmu, \bs{x}_{\bs{\alpha}}, \bs{y}_{\bs{\beta}}, 
\bs{u}] \right) \cap \CC[\theta_i^\ell] \right)=\big\{ \hat{%
\theta}^\ell_i \big\}.
\end{equation}
The choice of $P_{2, i}$ implies that~\eqref{eq:step4} is equivalent to the
fact that  $\theta_i^\ell$ is globally identifiable under the assumption
that $P_{2, i}(a) \neq 0$.  Thus, if  
\begin{equation}  \label{eq:cond}
P_{2, i}(a) \neq 0,
\end{equation}
then the decision of Algorithm~\ref{alg:identifiability}  regarding global
identifiability of $\theta_i^\ell$ will be correct.

Let $P_2 := \prod\limits_{i = 1}^ {|\bftheta^\ell|} P_{2, i}$.
Then $P_2(a) \neq 0$ implies that  the output of Algorithm~\ref%
{alg:identifiability} is correct.  Proposition~\ref{prop:subst} applied to$%
\prod\limits_{i = 1}^{ |\bftheta^\ell|} P_{2, i}$  provides a
polynomial $\widetilde{P}_2 \in \CC[\bftheta, \bs{u%
}]$ such that,  for all $\tilde{\bftheta}$ and $\tilde{%
\bs{u}}$,  $\widetilde{P}_2\big(\tilde{\bftheta}, \tilde{%
\bs{u}}\big)\ne 0$ implies that $P_2$ does not vanish at $%
\pi_{ip}^{-1}(\tilde{\bftheta}, \tilde{\bs{u}}) \cap Z_{%
\bs{\beta}}$.  Using the B{\'e}zout bound and the degree bound
from Theorem~\ref{thm:probabilictic}, we obtain  
\begin{equation*}
\deg Q\widetilde{P}_2 \leqslant 3\big|\bftheta^{\ell}\big| \deg
Z_{\bs{\beta}}\cdot (1 + 2d_0 \max \bs{%
\beta}) \leqslant  3\big|\bftheta^{\ell}\big|\left(
\prod\limits_{P \in E^t} \deg P \right)(1 + 2d_0 \max \bs{\beta}).  
\end{equation*}
Since the coordinates of $(\hat{\bftheta}, \hat{\bs{u}})$
are sampled from $1$ to $D_2$,  the Demillo-Lipton-Schwartz-Zippel lemma
(see~\cite[Proposition~98]{Zippel}) implies that the point $(\hat{%
\bftheta}, \hat{\bs{u}})$  is not a zero of $Q\widetilde{P%
}_2$ with probability at least  
\begin{equation*}
1 - \frac{\deg Q\widetilde{P}_2}{D_2}\geqslant 1- \frac{D_2(1-p)}{2D_2}=  \frac{1 + p}{2}  
\end{equation*}
Thus, with probability at least \[1 - \left(\left(1 - \frac{1 + p}{2}\right) + \left(1 - \frac{1 + p}{2}\right)\right) = 2 \cdot \frac{1 + p}{2} - 1 = p,\] the
output of Algorithm~\ref{alg:identifiability} is correct.
\end{proof}


\section{Performance}
\label{sec:examples} 

In this section, we discuss
the performance of an implementation of
Algorithm~\ref{alg:identifiability} using both basic and challenging examples taken from the
literature and also discuss how several other existing software packages perform at these examples.
Briefly, Table~\ref{table:runtimes} below shows that our running time compares favorably to existing software for several
challenging problems taken
from the literature.  
More details may be found in our
companion paper \cite{SIAN}.

We begin by giving a brief descriptions of the implementation of Algorithm~\ref{alg:identifiability} and three other software packages that are available to us.
\begin{description}
\descitem{Algorithm 1}
The current implementation of  Algorithm~\ref{alg:identifiability} is done on the computer algebra system \textsc{Maple}~$2017$. It takes advantage of parallel computing
because
(1) the built-in \textsc{Maple}
function for computing Gr\"obner bases is parallelized at the thread level
and 
(2) while using \ref{method:1} for Step~\ref{alg1:Step4}, we
naturally compute each Gr\"obner basis in a separate process. 
The implementation and examples used in this paper are contained in SIAN v0.5 (available at \url{https://github.com/pogudingleb/SIAN/releases/tag/v0.5}).
Note that the algorithm for Gr\"obner basis computation in \textsc{Maple} is also Monte Carlo with the probability of error at most $10^{-18}$, so this probability should be subtracted from the probability of success of our implementation.

\descitem{DAISY} This is software written in \textsc{Reduce}~\cite{DAISY}.
 For comparison, we used version~$1.9$.
 DAISY takes as an input a positive integer $\func{SEED}$, 
 which is used for sampling random points. 
 We used the default value $35$.
 In our experiments, the software did not appear
 to compute in parallel.  Since the core of the algorithm is a computation of
 a characteristic set decomposition  of a radical differential ideal, it is
 not clear how the algorithm could be efficiently  parallelized.
 
 \descitem{COMBOS} This is a web-based application~\cite{COMBOS}.
 For some of the examples below we received an error message saying ``Model may have been entered incorrectly or cannot be solved with COMBOS algorithms''.
 We will denote such cases by $\ast\ast$ in Table~\ref{table:runtimes}.
 
 \descitem{GenSSI 2.0} This is a package written in \textsc{Matlab}~\cite{LFCBBCH}.
 The algorithm uses \textsc{solve} function from \textsc{Matlab}  to solve
 systems of algebraic equations symbolically.
 For large examples (such as, for example, Examples~\ref{ex:CRN} and~\ref{ex:Cholera}), 
 such a symbolic solution does not exists.
 In such cases, the \textsc{solve} function returns an empty set of solutions together
 with a warning ``Warning: Unable to find explicit solution.'' 
 This means that there might be solutions but they could not be found by \textsc{Matlab}.
The algorithm is unable to conclude whether the parameters
        are globally identifiable due to this \textsc{Matlab} failure
  (see Examples~\ref{ex:CRN} and~\ref{ex:Cholera}).
 We will denote such cases by $\ast$ in Table~\ref{table:runtimes}.
\end{description}

We ran the program on a computer with $96$ CPUs, $2.4$ GHz each, under a Linux
operating system (CentOS 6.9). The runtime is the elapsed time. The sequential time is
the total CPU time spent by all threads of all  processes during the
computation. The former is measured as the \textsc{real} part of the output
of the Linux \textsc{time} command. The latter is  estimated as the sum of the \textsc{%
user} and \textsc{sys} parts of the output of the \textsc{time} command.
We report the runtimes in Table~\ref{table:runtimes} and the sequentual times in the text afterwards.

We determine locally identifiable parameters using software from~\cite{Sed2002}.
In all our examples, it took less than~$5$ seconds; this is negligible compared to
the other timings.

 We ran  Algorithm~\ref{alg:identifiability} 
and DAISY on \emph{all} examples from~\cite{DAISY,DAISY_MED,DAISY_IFAC,DAISY2010}.
The runtimes of both programs 
were below 1 minute.\footnote{
The example built
in~\cite[Section~6]{DAISY2010}   
looks challenging 
since it involves 
42
state variables and parameters. However, it is actually 
not challenging since it can be   straightforwardly divided into several 
non-challenging
problems: one can first analyze the identifiability of the parameters that appear in the
first equation (using only  
this
first equation)  and then add the other equations one-by-one to analyze the remaining parameters. 
The corresponding computation for the first equation takes less than 1 minute for
both programs, and the computations for the other 
equations can be done after that even by hand.}
Thus, from now on, we will elaborate on several (four) 
more challenging
problems taken from  literature
~\cite{NFkB_further,comparison,ConradiShiu,Pharm,Cholera,NFkB_first}
and 
compare the performance of Algorithm~\ref{alg:identifiability} and
the three other  software packages (DAISY, COMBOS, GenSSI 2.0) on them.

Table~\ref{table:runtimes} summarizes the  runtimes.
In the table, our timing is the best of 
timings obtained by performing Step~\ref{alg1:Step4} of Algorithm~\ref{alg:identifiability}
by~\ref{method:1} and~\ref{method:2} from Remark~\ref{rem:proj_comp}.

\begin{table}[h!]
\caption{Runtimes (in minutes)  on challenging problems}\label{table:runtimes}
\begin{center}
\begin{tabular}{|l|r|r|r|r|}
\hline
  Example & \descref{GenSSI 2.0} & \descref{COMBOS}  & \descref{DAISY}  &Algorithm~\ref{alg:identifiability} \\
\hline
  Example~\ref{ex:CRN}  &$\ast$ & $\ast\ast$ & $>$ \hfill $6{,}000$ & $< \hspace{1mm}1$\\
  \hline
  Example~\ref{ex:Cholera} &$\ast$ & $85\;$ & $30$  & $3$\\
  \hline
  Example~\ref{ex:NFkB} & $>$ \hspace{1mm} $12{,}000$ &$\ast\ast$ & $>$ \hfill $6{,}600$ & $48$\\
\hline
  Example~\ref{ex:PharmAeq} & $>$ \hspace{1mm} $12{,}000$ &$\ast\ast$  & $>$ \hfill $7{,}800$ & $993$\\
\hline
\end{tabular}
\end{center}
\small
$\ast$: \;\;\;GenSSI 2.0 returns~``{\it Warning: Unable to find explicit solution.}''\\
$\ast\ast$:\; COMBOS \ returns~``{\it Model may have been entered incorrectly or cannot be solved with COMBOS algorithms.}''
\end{table}

\begin{example}
\label{ex:CRN}  The following system of ODEs corresponds to a chemical
reaction network~\cite[Eq.~3.4]{ConradiShiu},  which is a reduced fully
processive, $n$-site phosphorylation network.  
\begin{align*}
\begin{cases}
\dot x_1 & = -\mu_1 x_1 x_2 + \mu_2 x_4 + \mu_4 x_6, \\ 
\dot x_2 & = -\mu_1 x_1 x_2 + \mu_2 x_4 + \mu_3 x_4, \\ 
\dot x_3 & = \mu_3 x_4 + \mu_5 x_6 - \mu_6 x_3 x_5, \\ 
\dot x_4 & = \mu_1 x_1 x_2 - \mu_2 x_4 - \mu_3 x_4, \\ 
\dot x_5 & = \mu_4 x_6 + \mu_5 x_6 - \mu_6 x_3 x_5, \\ 
\dot x_6 & = -\mu_4 x_6 - \mu_5 x_6 + \mu_6 x_3 x_5%
\end{cases}%
\end{align*}
Setting the outputs $y_1 = x_2$ and $y_2 = x_3$, we obtain a system of the form~%
\eqref{eq:main}.  We run Algorithm~\ref{alg:identifiability} setting $%
\bftheta^\ell := \{\mu_1, \ldots, \mu_6, x_1^\ast,
\ldots, x_6^\ast\}$ (a calculation shows that these parameters are locally
identifiable).  The intermediate results are the following:
\begin{itemize}
\item the system $E$ consists of~$104$ equations in~$116$ variables; 

\item $D_1 = 4{,}204{,}800$; 

\item $\bs{\beta} = (7, 7)$, $\bs{\alpha} = (6, 7, 7, 6, 6,
6)$; 

\item $D_2 = 4{,}936{,}445{,}783\cdot 10^{11}$; 

\item the system $\widehat{E}^t$ consists of $52$ equations in $50$ variables. 
\end{itemize}
The algorithm returns that all the parameters are globally identifiable with
probability at least $99\%$.  If we perform the last step of Algorithm~\ref%
{alg:identifiability} using \ref{method:1} from  Remark~\ref{rem:proj_comp},
the runtime is $0.4$ minutes (the  estimated sequential time is $0.9$ minutes).  If we
use \ref{method:2}, the runtime is $0.4$ minutes (the  estimated sequential time is $0.9$
minutes).  The timings for both methods are similar, and the improvement by
computing in parallel 
is not that  significant because the Gr\"obner bases computations are
relatively easy in this case compared to  the other steps of the algorithm.
\begin{itemize}
  \item \descref{DAISY} did not output any result in $100$ hours.
  \item \descref{COMBOS} returned ``Model may have been entered incorrectly or cannot be solved with COMBOS algorithms''.
  \item \descref{GenSSI 2.0} constructed a polynomial system that could not be
solved symbolically by \textsc{Matlab}.
\textsc{Matlab} returned ``Warning: Unable to find explicit solution.'' and an empty set of solutions.
As a result, the algorithm reports that
$x_2^\ast$ and $x_3^\ast$ are globally identifiable
and the other parameters are  locally identifiable.
It was not able to determine
whether these other parameters are globally identifiable.
\end{itemize}
\end{example}

\begin{example}
\label{ex:Cholera}  The following version of SIWR is an extension of the SIR
model, see~\cite[Eq.~3]{Cholera}:  
\[
\begin{cases}
\dot{s} & = \mu - \beta_I s i - \beta_W s w - \mu s + \alpha r, \\ 
\dot{i} & = \beta_W s w + \beta_I s i - \gamma i - \mu i, \\ 
\dot{w} & = \xi (i - w), \\ 
\dot{r} & = \gamma i - \mu r - \alpha r%
\end{cases}%
\]
where $s$, $i$, and $r$ stand for the fractions of the population that are
susceptible, infectious, and recovered, respectively.  The variable $w$
represents the concentration of the bacteria in the environment. The scalars 
$\alpha, \beta_I, \beta_W, \gamma, \mu, \xi$ are unknown parameters. 
Following~\cite{Cholera}, we assume that we can observe $y_1 = \kappa_1 i$,
where $\kappa_1$ is one more unknown parameter.  We will also assume that
one can measure the total population $s + i + r$, so $y_2 = s + i + r$.  We
run Algorithm~\ref{alg:identifiability} setting $\bftheta^\ell :=
\{\alpha, \beta_I, \beta_W, \gamma, \mu, \xi, \kappa_1, s^\ast, i^\ast,
w^\ast, r^\ast\}$ (it was show in~\cite{Cholera} that they are locally
identifiable).  The intermediate results are the following: 
\begin{itemize}
\item the system $E$ consists of~$66$ equations in~$77$ variables; 

\item $D_1 = 2{,}653{,}200$; 

\item $\bs{\beta} = (10, 8)$, $\bs{\alpha} = (9, 10, 9, 8)$; 

\item $D_2 = 1{,}744{,}556{,}312\cdot 10^{12}$; 

\item the system $\widehat{E}^t$ consists of~$54$ equations in~$47$ variables. 
\end{itemize}
The algorithm   returns
 that all parameters are globally identifiable with
probability at least $99\%$.  If we perform the last step of Algorithm~\ref%
{alg:identifiability} using \ref{method:1} from  Remark~\ref{rem:proj_comp},
the runtime is $3$ minutes (the  estimated sequential time is $77$ minutes).  If we
use \ref{method:2}, the runtime is $15$ minutes (the  estimated sequential time is $%
114$ minutes).  Unlike in Example~\ref{ex:CRN}, in this case, the Gr\"obner
bases computations are the most  time-consuming part of the algorithm. 
Because of this, the parallelization of the Gr\"obner bases computations in 
\textsc{Maple}  gives almost $8$~times speed up while using \ref{method:2}. 
Combined with performing different Gr\"obner bases computations in parallel 
while using \ref{method:1}, it gives an almost $25$~times speed up.
\begin{itemize}
  \item \descref{DAISY} took~$30$ minutes to output the correct result.
  \item \descref{COMBOS} took~$85$ minutes to output the correct result.
  \item \descref{GenSSI 2.0} constructed a polynomial system that could not be
  solved symbolically by \textsc{Matlab}.
  \textsc{Matlab} returned ``Warning: Explicit solution could not be found'' and an empty set of solutions.
   Because of this, the algorithm was not able to determine if the parameters are globally identifiable.
\end{itemize}
\end{example}

\begin{example}\label{ex:NFkB}
  Consider the model of NF$\kappa$B regulatory module proposed in~\cite{NFkB_first}
  (see also~\cite{NFkB_further} and~\cite[Case~6]{comparison}) defined by 
  the following system~\cite[Equation~27]{comparison}
  \begin{equation}
  \begin{cases}
    \dot x_1 = k_{prod} - k_{deg} x_1 - k_1 x_1  u,\\
    \dot x_2 = -k_3 x_2 - k_{deg} x_2 - a_2 x_2 x_{10} + t_1 x_4 - a_3 x_2 x_{13} + t_2 x_5 + (k_1 x_1 - k_2 x_2 x_8) u,\\
    \dot x_3 = k_3 x_2 - k_{deg} x_3 + k_2 x_2 x_8 u,\\
    \dot x_4 = a_2 x_2 x_{10} - t_1 x_4,\\
    \dot x_5 = a_3 x_2 x_{13} - t_2 x_5,\\
    \dot x_6 = c_{6a} x_{13} - a_1 x_6 x_{10} + t_2 x_5 - i_1 x_6,\\
    \dot x_7 = i_1 k_v x_6 - a_1 x_{11} x_7,\\
    \dot x_8 = c_4 x_9 - c_5 x_8,\\
    \dot x_9 = c_2 + c_1 x_7 - c_3 x_9,\\
    \dot x_{10} = -a_2 x_2  x_{10} - a_1 x_{10} x_6 + c_{4a} x_{12} - c_{5a} x_{10} - i_{1a} x_{10} + e_{1a} x_{11},\\
    \dot x_{11} = -a_1 x_{11} x_7 + i_{1a} k_v x_{10} - e_{1a} k_v x_{11},\\
    \dot x_{12} = c_{2a} + c_{1a} x_7 - c_{3a} x_{12},\\
    \dot x_{13} = a_1 x_{10} x_6 - c_{6a} x_{13} - a_3 x_2 x_{13} + e_{2a} x_{14},\\
    \dot x_{14} = a_1 x_{11} x_7 - e_{2a} k_v x_{14},\\
    \dot x_{15} = c_{2c} + c_{1c} x_7 - c_{3c} x_{15}
  \end{cases}  
  \end{equation}
  In the above system,
\glssymbolY{input}{$u$}
    is the input function, $x_1, \ldots, x_{15}$ are the
    state variables, and the rest
  are scalar parameters.
  The outputs are
  $y_1 = x_2$, $y_2 = x_{10} + x_{13}$, $y_3 = x_9$, $y_4 = x_1 + x_2 + x_3$, $y_5 = x_7$, and $y_6 = x_{12}$.
  The values of some of the parameters are known from the existing literature (see~\cite[Table~1]{NFkB_further}), so we run Algorithm~\ref{alg:identifiability} with
  \[
    \bftheta^\ell = (t_1, t_2, c_{3a}, c_{4a}, c_5, k_1, k_2, k_3, k_{prod}, k_{deg}, i_1, e_{2a}, i_{1a}, x_1^\ast, \ldots, x_{14}^\ast).  
  \]
  The intermediate results are the following:
  \begin{itemize}
    \item the system $E$ consists of~$588$ equations in~$623$ variables; 

    \item $D_1 = 80{,}640{,}000$; 

    \item $\bs{\beta} = (8, 7, 6, 7, 6, 6)$, $\bs{\alpha} = (7, 8, 7, 7, 7, 6, 6, 7, 6, 7, 6, 6, 7, 6, 1)$; 

    \item $D_2 = 1{,}856{,}032{,}379 \cdot 10^{24}$; 

    \item the system $\widehat{E}^t$ consists of~$134$ equations in~$121$
    variables. 
  \end{itemize}
  The algorithm
 returns
  that all  parameters are globally identifiable with
  probability at least $99\%$.  If we perform the last step of Algorithm~\ref%
  {alg:identifiability} using \ref{method:1} from  Remark~\ref{rem:proj_comp},
  the runtime is $48$ minutes (the  estimated sequential time is  926 minutes).  
  If we use \ref{method:2}, the runtime is 
  $190$~minutes (the  estimated sequential time is $1{,}830$~minutes).
  \begin{itemize}
  \item \descref{DAISY} did not output any result in $110$ hours.
  \item \descref{COMBOS} returned ``Model may have been entered incorrectly or cannot be solved with COMBOS algorithms''.
  \item \descref{GenSSI 2.0} did not output any result in $200$ hours.
\end{itemize}
 A version of this
   problem was solved by GenSSI-like method in~\cite{NFkB_further} under 
  an additional assumption on the initial conditions~\cite[p.~9]{NFkB_further}, which we do not make.

\end{example}


\begin{example}\label{ex:PharmAeq}
  Consider the following model arising in pharmacokinetics~\cite{Pharm}:
  \begin{equation}\label{eq:Pharm}
  \begin{cases}
    \dot x_1 = a_1 (x_2 - x_1) - \frac{k_a V_m x_1}{k_c k_a + k_c x_3 + k_a x_1},\\
    \dot x_2 = a_2(x_1 - x_2),\\
    \dot x_3 = b_1(x_4 - x_3) - \frac{k_cV_m x_3}{k_c k_a + k_c x_3 + k_a x_1},\\
    \dot x_4 = b_2(x_3 - x_4).
  \end{cases}  
  \end{equation}
  The only output is $y = x_1$.  
  Despite 
 having
  moderate size, the global identifiability of~\eqref{eq:Pharm} is
  out of reach for all software tools we are aware of, see also~\cite[Case~2]{comparison}.
  We consider a simplified version of~\eqref{eq:Pharm} with the additional assumption $a_1 = a_2$.
  We run Algorithm~\ref{alg:identifiability} with
  \[
    \bftheta^\ell = \{ a_1, b_1, b_2, k_a, k_c, V_m, x_1^\ast, \ldots, x_4^\ast\}.  
  \]
  The intermediate results are the following:
  \begin{itemize}
    \item the system $E$ consists of~$55$ equations in~$65$ variables; 

    \item $D_1 = 3{,}240{,}000$; 

    \item $\bs{\beta} = (11)$, $\bs{\alpha} = (11, 10, 10, 9)$; 

    \item $D_2 = 1{,}923{,}937{,}761 \cdot 10^{13}$; 

    \item the system $\widehat{E}^t$ consists of~$51$ equations in~$50$ variables. 
  \end{itemize}
  The algorithm returns that all parameters are globally identifiable with
  probability at least $99\%$.  If we perform the last step of Algorithm~\ref%
  {alg:identifiability} using \ref{method:1} from  Remark~\ref{rem:proj_comp},
  the runtime is $993$ minutes (the  estimated sequential time is  11,944 minutes).  
  If we use \ref{method:2}, the runtime is 
  $6{,}042$ minutes (the  estimated sequential time is $112{,}000$ minutes).
  \begin{itemize}
    \item \descref{DAISY} did not output any result in $130$ hours.
    \item \descref{COMBOS} returned ``Model may have been entered incorrectly or cannot be solved with COMBOS algorithms''.
    \item \descref{GenSSI 2.0} did not output any result in $200$ hours.
  \end{itemize}

\end{example}

        





 \ack 
      This work was partially supported by the NSF grants CCF-0952591,
        CCF-1563942, CCF-1564132, CCF-1319632, DMS-1606334, DMS-1760448, DMS-1853650, by the NSA grant
        \#H98230-15-1-0245, by CUNY CIRG \#2248, by PSC-CUNY grant \#69827-00
        47, by the Austrian Science Fund FWF grant Y464-N18. We
         are
        grateful to
        the CCiS at CUNY Queens College for the computational resources and to the  referees, Julio Banga, Dan Bates, Gr\'egoire Lecerf, Thomas Ligon, David Marker, Nikki Meshkat, Amaury Pouly, Maria Pia Saccomani, Anne Shiu, and Seth Sullivant for useful discussions and feedback.


\frenchspacing
\bibliographystyle{cpam}
\bibliography{bibdata}

\section*{Erratum for Lemma~\ref{lem:nonzero_poly}}

We are grateful to Peter Thompson for pointing out an error in the proof of~Lemma~\ref{lem:nonzero_poly}. 
The original proof worked only under the assumption that $\hat{\theta}$ is a vector of constants.
However, some of the components of $\hat{\bs{\theta}}$ could be the states of the dynamic under consideration, and it was important in Proposition~\ref{prop:equiv_fields} where the lemma was used.

We give a more explicit version of the statement and provide a correct proof.
The desired statement will be deduced from the following:

\begin{lemma}\label{lem:main}
  Consider a system of differential equations 
  \begin{equation}\label{eq:main}
    \begin{cases}
      x_1' = f_1(\bs{x}, \bs{\mu}, \bs{u}),\\
      \vdots\\
      x_n' = f_n(\bs{x}, \bs{\mu}, \bs{u}),
    \end{cases}
  \end{equation}
  where $\bs{x} = (x_1, \ldots, x_n)$ and $\bs{u} = (u_1, \ldots, u_m)$ are tuples of differential indeterminates, $\bs{\mu} = (\mu_1, \ldots, \mu_\lambda)$ are scalar parameters, and $f_1, \ldots, f_n \in \mathbb{C}(\bs{x}, \bs{\mu}, \bs{u})$.
  Let $Q(\bs{x}, \bs{\mu}, \bs{u}) \in \mathbb{C}[\bs{x}, \bs{\mu}, \bs{u}]$ be the LCM of the denominators of $f_1, \ldots, f_n$.
  Let $P \in \mathbb{C}[\bs{x}, \bs{\mu}]\{ \bs{u}\}$ be a nonzero differential polynomial.
  Then there exist nonzero $P_1 \in \mathbb{C}[\bs{x}, {\color{teal}\bs{\mu}}]$ and $P_2 \in \mathbb{C}\{\bs{u}\}$ such that, for every tuple $\hat{\bs{\mu}} \in \mathbb{C}^\lambda$ and every power series solution $(\hat{\bs{x}}, \hat{\bs{u}})$  of~\eqref{eq:main} with parameters $\hat{\bs{\mu}}$ in $\CC[\![t]\!]$ such that
  \[
    Q(\hat{\bs{x}}, \hat{\bs{\mu}}, \hat{\bs{u}})|_{t = 0} \neq 0
  \]
  we have
  \[
    \left(P_1(\hat{\bs{x}}, \hat{\bs{\mu}})|_{t = 0} \neq 0 \;\&\; P_2(\hat{\bs{u}})|_{t = 0} \neq 0\right) \implies P(\hat{\bs{x}}, {\color{teal}\hat{\bs{\mu}},} \hat{\bs{u}}) \neq 0.
  \]
\end{lemma}

\begin{proof}
  Consider the following differential ideal
  \[
    I := \langle (Qx_i' - Qf_i)^{(j)}, P^{(j)} \mid 1 \leqslant i \leqslant n,\; j \geqslant 0 \rangle \colon Q^\infty \subset \mathbb{C}[\bs{\mu}]\{\bs{x}, \bs{u}\}.
  \]
  We claim that $I$ contains a nonzero polynomial of the form $P_1P_2$ such that $P_1 \in \mathbb{C}[\bs{x}, \bs{\mu}]$ and $P_2 \in \mathbb{C}\{\bs{u}\}$.
  First we will show that, if the claim is true, then $P_1$ and $P_2$ satisfy the condition of the lemma.
  Assume the contrary, that there is a power series solution $(\hat{\bs{x}}, \hat{\bs{u}})$ of~\eqref{eq:main} with parameters $\hat{\bs{\mu}}$ such that the constant term of $Q(\hat{\bs{x}}, \hat{\bs{\mu}}, \hat{\bs{u}}) P_1(\hat{\bs{x}}, \hat{\bs{\mu}}) P_2(\hat{\bs{u}})$ is nonzero but $P(\hat{\bs{x}}, \hat{\bs{\mu}}, \hat{\bs{u}}) = 0$.
  Since $(\hat{\bs{x}}, \hat{\bs{\mu}}, \hat{\bs{u}})$ is a zero of differential polynomials $P$ and $Qx_i' - Qf_i$ for every $1 \leqslant i \leqslant n$, it is a zero of the ideal 
  \[
  \langle (Qx_i' - Qf_i)^{(j)}, P^{(j)} \mid 1 \leqslant i \leqslant n,\; j \geqslant 0 \rangle.
  \]
  Since $Q(\hat{\bs{x}}, \hat{\bs{\mu}}, \hat{\bs{u}})|_{t = 0} \neq 0$, every element in $I$ which is the saturation of above ideal at $Q$ also vanishes on $(\hat{\bs{x}}, \hat{\bs{\mu}}, \hat{\bs{u}})$.
  In particular, $P_1P_2$ vanishes on $(\hat{\bs{x}}, \hat{\bs{\mu}}, \hat{\bs{u}})$ , so we arrive at the contradiction with $P_1(\hat{\bs{x}}, \hat{\bs{\mu}}) P_2(\hat{\bs{u}}) \neq 0$.
  
  Now we will prove the claim. 
  Consider the ring $R := \mathbb{C}[\bs{x}, \bs{\mu}]\{\bs{u}\}[1/Q]$.
  Let $J$ be the ideal generated by $I \cap \mathbb{C}[\bs{x}, \bs{\mu}]\{\bs{u}\}$ in $R$.
  The definition of $I$ via the saturation at $Q$ implies that 
  \[
    J \cap  \mathbb{C}[\bs{x}, \bs{\mu}]\{ \bs{u}\} = I \cap \mathbb{C}[\bs{x}, \bs{\mu}]\{ \bs{u}\}.
  \]
  Thus, it is sufficient to prove that there is an element of the form $P_1 P_2$ with $P_1 \in \mathbb{C}[\bs{x}, \bs{\mu}]$ and $P_2 \in \mathbb{C}\{\bs{u}\}$ in $J$.
  We define a derivation $\mathcal{L}$ on $R$ (which is basically the Lie derivative) by
  \[
      \mathcal{L}(g) := \sum\limits_{i = 1}^n f_i\frac{\partial g}{\partial x_i} + \sum\limits_{\ell = 1}^m \sum\limits_{j = 0}^\infty u_\ell^{(j + 1)} \frac{\partial g}{\partial u_\ell^{(j)}}\quad \text{ for } g\in R.
  \]
  Since $Qx_1' - Qf_1, \ldots, Qx_n' - Qf_n \in I$ and $I$ is a differential ideal, $J$ is invariant under $\mathcal{L}$.
  
  Let $\widetilde{R}$ be the localization of $R$ with respect to $\mathbb{C}\{\bs{u}\}$ and $\widetilde{J}$ be the ideal generated by $J$ in this localization.
  The derivation $\mathcal{L}$ can be naturally extended to $\widetilde{R}$, and $\widetilde{J}$ is also $\mathcal{L}$-invariant.
  It is sufficient to prove that $\widetilde{J}\cap \mathbb{C}[\bs{x}, \bs{\mu}] \neq \{0\}$.
  Consider a nonzero element of $\widetilde{J} \cap \mathbb{C}[\bs{x}, \bs{\mu}]\{\bs{u}\}$ with the smallest number of monomials and, among such elements, an element of the smallest total degree. 
  We will call it $S$.
  If $S\in \mathbb{C}[\bs{x}, \bs{\mu}]$, we are done.
  Otherwise, one of $\bs{u}$ appears in $S$, say $u_1$.
  Let $h = \ord_{u_1}S$.
  
  Since $\widetilde{R}$ is a Noetherian ring, there exists $N > 0$ such that 
  \[
    \mathcal{L}^N(S) \in \langle S, \mathcal{L}(S), \ldots, \mathcal{L}^{N - 1}(S) \rangle. 
  \]
  We have $\ord_{u_1} \mathcal{L}^i(S) < N + h$ for $i < N$ and
  \[
    \mathcal{L}^N(S) = \frac{\partial S}{\partial u_1^{(h)}} u_1^{(h + N)} + T, \quad \text{ where } \ord_{u_1}T < N + h.
  \]
  Therefore, we have
  \[
    \frac{\partial S}{\partial u_1^{(h)}} \in \langle S, \mathcal{L}(S), \ldots, \mathcal{L}^{N - 1}(S) \rangle \subset \widetilde{J}.
  \]
  If $S$ were divisible by $u_1^{(h)}$, then $S / u_1^{(h)} \in \widetilde{J}$ would have the same number of monomials but smaller degree, this contradicts to the choice of $S$.
  Therefore, $\frac{\partial S}{\partial u_1^{(h)}}$ has fewer monomials than $S$ thus contradicting the choice of $S$.
\end{proof}

The following corollary is equivalent to Lemma~\ref{lem:nonzero_poly} but explicitly highlights that some of the entries of $\hat{\bs{\theta}}$ may be initial conditions, not only system parameters.

\begin{corollary}[{{Clarified version of~Lemma~\ref{lem:nonzero_poly}}}]
  Let $P(\bs{\mu}, \bs{x}, u, \ldots, u^{(N)})\in%
\CC [\bs{\mu}, \bs{x}]
    \{ u \}$
be nonzero. 
Then there exist nonempty
Zariski open subsets $\Theta{\color{teal}\subset}\CC^{s}$ and $U\subset \ensuremath{\cinfty}$
such that, for every $\hat{\bs{\theta}} = (\hat{\bs{\mu}}, \hat{\bs{x}}^\ast)\in\Theta$,  $\hat{u}\in U$, and the corresponding $\hat{\bs{x}} = X(\hat{\bs{\theta}}, \hat{u})$, the function $P(\hat{\bs{\mu}}, \hat{\bs{x}}, \hat{u},\ldots,(\hat{u})^{(N)})$ 
is a nonzero element of $\ensuremath{\cinfty}$.
\end{corollary}

\begin{proof}
  We apply Lemma~\ref{lem:main} to the model $\Sigma$ and the polynomial $P$ as in the statement,
   and obtain polynomials $P_1(\bs{x}, \bs{\mu})$ and $P_2(\bs{u})$.
  We define Zariski open sets $\Theta$ and $U$ by $P_1 \neq 0$ and $P_2(\bs{u})|_{t = 0} \neq 0$, respectively.
  Then the lemma implies that, for $(\hat{\bs{\mu}}, \hat{\bs{x}}^*) \in \Theta$ and $\hat{u} \in U$, $P(\hat{\bs{\mu}}, \hat{\bs{x}}, \hat{u},\ldots,(\hat{u})^{(N)})$ will be a nonzero function.
\end{proof}

\end{document}